\documentclass[12pt]{amsart}
\usepackage{amsfonts, amssymb, amsmath, amsthm}

\usepackage{color}
\usepackage{enumerate}
\usepackage{soul,xcolor}
\usepackage{cancel}
\usepackage[normalem]{ulem}
\usepackage{enumerate}

\newcommand{\doublearrow}[2]{\overset{#1}{\underset{#2}{\rightrightarrows}}}
\newcommand{\xydr}{\ar@<-.5ex>[r] \ar@<.5ex>[r]}
\newcommand{\xydd}{\ar@<-.5ex>[d] \ar@<.5ex>[d]}

\mathchardef\ordinarycolon\mathcode`\:
\mathcode`\:=\string"8000
\begingroup \catcode`\:=\active
  \gdef:{\mathrel{\mathop\ordinarycolon}}
\endgroup

\newcommand{\mbred}[1]{{\textcolor{red}
{#1}}}

\newcommand{\red}[1]{{\textcolor{red}
{#1}}}

\newcommand{\ZG}{\mathbb Z G }
\newcommand{\Z}{\mathbb{Z}}

\newcommand{\N}{\mathbb{N}}
\newcommand{\R}{\mathbb{R}}
\newcommand{\rr}{\mathcal{R}}

\newcommand{\stabzero}{\textnormal{st0}}
\newcommand{\stabone}{\textnormal{st1}}

\newcommand{\Rcal}{\mathcal{R}}
\newcommand{\calR}{\mathcal{R}}

\newcommand{\gl}{\textnormal{GL}}
\newcommand{\el}{\textnormal{El}}

\newcommand{\As}{\mathcal{A^{\Box}}}
\newcommand{\Bs}{\mathcal{B^{\Box}}}

\newcommand{\ea}{E_A}
\newcommand{\ear}{E(A,\calR)}
\newcommand{\enilar}{E_{\textnormal{Nil}}(A,\calR)}

\newcommand{\ebr}{E(B,\calR )}

\newcommand{\orbglr}{\textnormal{Orb}_{\textnormal{GL}(\calR )}}
\newcommand{\orbglrt}{\textnormal{Orb}_{\textnormal{GL}(\calR[t] )}}

\newcommand{\nil}{\textnormal{Nil}_0}
\newcommand{\orbelr}{\textnormal{Orb}_{\textnormal{EL}(\calR )}}
\newcommand{\orbelrt}{\textnormal{Orb}_{\textnormal{EL}(\calR[t] )}}

\newcommand{\go}{\textnormal{Orb}_{\textnormal{GL}(\calR )}}
\newcommand{\eo}{\textnormal{Orb}_{\textnormal{El}(\calR )}}
\newcommand{\seo}{\textnormal{ElSt} }
\newcommand{\seor}{\textnormal{ElSt}_{\calR}}
\newcommand{\seozg}{\textnormal{ElSt}_{\Z G}}
\newcommand{\seort}{\textnormal{ElSt}_{\calR [t]}}

\newcommand{\eogo}{\textnormal{Orb}_{\textnormal{El}(\calR ),\textnormal{Gl}(\calR )}}

\newcommand{\nk}{NK}
\newcommand{\NK}{NK}
\newcommand{\sk}{SK}

\newcommand{\coker}{\textnormal{coker}}

\DeclareMathOperator{\GL}{GL}
\DeclareMathOperator{\SSE}{SSE}
\DeclareMathOperator{\SE}{SE}
\DeclareMathOperator{\SL}{SL}
\DeclareMathOperator{\EL}{El}
\DeclareMathOperator{\El}{El}

\numberwithin{equation}{section}
\newtheorem{theorem}[equation]{Theorem}
\newtheorem{proposition}[equation]{Proposition}

\newtheorem{corollary}[equation]{Corollary}
\newtheorem{conjecture}[equation]{Conjecture}
\newtheorem{lemma}[equation]{Lemma}
\newtheorem{prop}[equation]{Proposition}

\theoremstyle{definition}
\newtheorem{notation}[equation]{Notational convention}
\newtheorem{definition}[equation]{Definition}

\newtheorem{elstabprob}[equation]{Elementary Stabilizer Problem}

\theoremstyle{remark}
\newtheorem{remark}[equation]{Remark}

                            {\end{enumerate}}
\input xy
\xyoption{all}
\numberwithin{equation}{section}
\begin{document}
\setstcolor{red}

\title{Strong shift equivalence and algebraic K-theory}

\author{Mike Boyle and Scott Schmieding}

\begin{abstract} For a semiring $\Rcal$, the relations of
shift equivalence over $\Rcal$ (SE-$\Rcal$) and strong shift equivalence
over $\Rcal$  (SSE-$\Rcal$) are natural equivalence relations
on square matrices over $\Rcal$, important  for symbolic dynamics.
When $\calR$ is a ring, we prove that the refinement of
SE-$\Rcal$ by SSE-$\Rcal$, in the SE-$\calR$ class of
  a matrix $A$, is classified by the
quotient $\nk_1(\calR )/\ear$ of the
algebraic K-theory group
$\nk_1 (\calR)$. Here, $\ear$ is a certain
  stabilizer group, which we prove must vanish if $A$ is nilpotent
or invertible. For this, we first show for any square
matrix $A$ over $\Rcal$ that the refinement of its
SE-$\Rcal$ class into  SSE-$\Rcal$ classes corresponds
precisely to the refinement of
the $\GL (\Rcal [t])$ equivalence class of $I-tA$
into $\el (\Rcal [t])$ equivalence classes. We then
show this refinement is in bijective correspondence
with $\nk_1 (\calR )/\ear$.
For a general ring $\calR$ and $A$ invertible,
the proof that $\ear$ is trivial
rests on  a theorem of Neeman and Ranicki
 on the K-theory of  noncommutative localizations.
For $\calR$ commutative, we show
$\cup_A \ear = NSK_1(\calR)$; the proof rests on
Nenashev's presentation of $K_1$ of an exact
category.
\end{abstract}
\maketitle

\tableofcontents

\section{Introduction}\label{sec:intro}
Let $\Rcal$ (always assumed to contain 0 and 1)
be a subset of a ring.
Let $A,B$ be square matrices over $\Rcal$ (not necessarily of
equal size).  Matrices $A$ and $B$ over $\Rcal$ are
{\it elementary strong shift equivalent}
over $\Rcal$ (ESSE-$\Rcal$)
if there exist matrices $U,V$ over $\Rcal$ such that
$A=UV$ and $B=VU$. $A$
 and $B$ are  {\it strong shift equivalent}
over $\Rcal$ (SSE-$\Rcal$) if they are connected by a chain of
elementary strong shift equivalences. $A$ and $B$ are
\emph{shift equivalent} over $\Rcal$ (SE-$\Rcal$) if there exist matrices
$U,V$ over $\Rcal$ and $\ell $ in $\mathbb N$ such that the
following hold:
\begin{align*}
A^{\ell} = UV  \ \qquad  B^{\ell} &=VU \\
AU=UB         \qquad  VA &=BV \quad .
\end{align*}
If $A,B$ are  SSE-$\Rcal$, then they are SE-$\Rcal$.

For symbolic dynamics, these are central
relations,  introduced
by Williams \cite{LindMarcus1995,Williams73};
%(primarily for
%$\calR=\[ 0,1\}$ or  $\mathbb Z_+$)
they may be familiar from other settings. For example, idempotent matrices $p, q$ over a unital $C^{*}$-algebra $\mathcal{A}$ are Murray-von Neumann equivalent if and only if $p$ and $q$ are SSE-$\Rcal$ (this can be deduced from \cite[Lemma A.4.4]{HigsenRoeBook}).
%(they make perfect sense for endomorphisms
%in a category in place of matrices).
We give background and motivation in Section \ref{backgroundsec}.
Briefly: shift equivalence is very useful for symbolic dynamics and
reasonably tractable, with several algebraic characterizations when
$\Rcal$ is a  ring (see Theorem \ref{selisttheorem}).
% (For example, if $\Rcal$ is a ring,
%then matrices $A,B$ over $\calR$ are SE-$\Rcal$ if and only if the
%$\calR[t]$-modules $\coker(I-tA)$, $\coker(I-tB)$ are
%isomorphic.)
Strong shift equivalence is a more fundamental
and mysterious relation.

There is an obvious  basic question: assuming  $\calR$ is a ring,
does $\SE-\calR$ imply $\SSE-\calR$?
The answer was shown to be
 yes for $\calR=\Z$ (see Williams' proof
in \cite{Williams1992} on his work from the 70s);
for $\calR$ a principal ideal domain (Effros, 1981, \cite{Effros1981});
and
for $\calR$ a Dedekind domain (Boyle-Handelman, 1993 \cite{BH93}). There were no counterexamples, and no results after \cite{BH93}.
In his 1999 Bulletin AMS survey,
  Wagoner formally posed the ``Algebraic Shift Equivalence
  Problem'' \cite[Problem 2.14]{Wagoner99}:
for what rings $\Lambda$ does SE over $\Lambda$
imply SSE over $\Lambda$?
We will show that for $\calR$ a ring,
in a given $\SE-\calR$ class the refinement of $\SE-\calR$
by $\SSE-\calR$ is captured exactly by
a certain quotient group of the algebraic K-theory group $\nk_1(\calR )$.

%, summarizing Theorem \ref{aplusn}).

%The central result behind this theorem is the
%following correspondence..
%What leads to this connection?
%In the ``positive K-theory'' framework
%for symbolic dynamics, it is natural to
%move from a given matrix $A$ over $\calR$
%to a new matrix $I-tA$ over
%the polynomial ring
%$\calR [t]$ (or just the matrix
%$I-A$ over $\calR$). Invariants of
%interest do not distinguish
%$I-tA$ and $(I-tA) \oplus I_n$, for
%any $I_n$ (the $n\times n$ identity matrix).
%

%\red{The introduction from here has been rewritten,
%without colors.}

From here, let $\Rcal$ be a ring, and
$\mathfrak M_n(\Rcal)$ the $n\times n$
matrices over $\Rcal$. With the maps
 $p_n\colon \mathfrak M_n(\calR)\to \mathfrak M_{n+1}(\calR)$
defined by $M\mapsto M\oplus 1$, we form a direct limit
of semigroups $\mathfrak M(\calR)$,
with a finite matrix $M$ sent to $M_{\stabone}$ in  $\mathfrak M(\calR)$.
%\red{(Here and elsewhere I changed
%    subscripts, e.g. $M_{\infty}$ to $M_{\stabone}$, for some
%    reasons.)}
The maps $p_n$ are the
maps which construct
$\gl(\Rcal )$ and the elementary
group $\textnormal{El}(\Rcal )$
as direct limits.  A $\GL_n(\calR)$ equivalence
$UMV=M'$ gives a $\GL_{n+1}(\calR)$ equivalence
$p_n(U) p_n(M) p_n(V)=p_n(M')$, so
$\GL(\Rcal )$ equivalence  and
$\el(\Rcal )$ equivalence of the objects
$M_{\stabone}$ is well defined.   When we say that
two finite matrices $M$ and $M'$
are $\GL(\Rcal )$ equivalent
or $\el(\Rcal )$ equivalent, we mean
that the relation holds for
$M_{\stabone}$ and $(M')_{\stabone}$,
i.e. $UM_{\stabone}V = (M')_{\stabone}$ for
$U,V \in GL(\calR)$ or $U,V \in \el(\calR)$.
It is natural to identify
$M_{\stabone} $ with an $\N \times \N$ matrix
(see Sec. \ref{backgroundsec}).
%equal to the identity outside finitely many
%entries.

For finite square matrices $A,B$ over $\Rcal$,
we will show
\begin{align} \label{seisgl}
A \text{ and } B \text{ are  SE-}\calR  &\iff  I-tA\text{ and }  I-tB \text{ are GL}(\calR[t])
\text{ equivalent} \\  \label{sseisel}
A\text{ and } B \text{ are  SSE-}\calR  &\iff  I-tA\text{ and }  I-tB \text{ are El}(\calR[t])
\text{ equivalent}
\end{align}
The proof of \eqref{seisgl} in Section  \ref{secfitting}
uses an old stabilization result of Fitting,
following Warfield (see Theorem \ref{WarfieldLemma}).
In Section \ref{sseaselsec},
\eqref{sseisel} is proved.
The formulation
of the correspondence in
Theorem \ref{finecentral}
as induced by a map
 $ I-A \mapsto \As$ is simple and natural.
The matrix arguments of the proof, however, are
nonstandard for $K$-theory,  and
a $K$-theorist may find the details
barbaric:
nonfunctorial,  complicated
and (worst of all?)
bereft of exact sequences. For better
and for worse, this is the proof we have.

% To state the second main result (*),
%use a little notation.
%For a square matrix $M$ over $\Rcal$, let
%$\go (M)$ be the set of square matrices $M'$
%over $\Rcal$ which are
%$\GL(R[t})$ equivalent to $M$.
%Let $\eo (M)$ be the set of square matrices $M'$
%over $\Rcal$ which are
%$\el(R[t})$ equivalent to $M$.
%Define
%\[
%\eogo (M)= \{\eo (M'): M' \in \go (I-A)\} \ .
%\]
Given a ring $\Rcal$ and  a square matrix $M$ over $\Rcal$,  we
define associated sets of square matrices over $\Rcal$:
%\begin{align*}
%\orbglr (M) &=
%\{ UM_{\stabone}V: \ \{U,V\} \subset \GL(\Rcal )\ \}\\
%\orbelr (M) &=
%\{ UM_{\stabone}V: \ \{U,V\} \subset \el(\Rcal )\ \}
%\end{align*}
\begin{align*}
\orbglr (M)&=
\{ M': M' \text{ is } \GL(\calR) \text{ equivalent to }M \}\\
\orbelr (M) &=
\{ M': M' \text{ is }\, \el (\calR) \text{ equivalent to }M \}
\end{align*}
Now suppose $A$ is any square matrix over $\Rcal$.
Then
 $\orbglrt (I-tA)$ is a disjoint union of
the sets $\orbelrt (I-tB)$ such that $I-tB$ is
$ \GL(\calR[t])$  equivalent to $I-tA$ and $B$ has entries in $\Rcal$.
Define the elementary stabilizer
  \[
 \ear = \{ U\in \GL (\calR [t])\colon
  U {\textnormal{Orb}_{\textnormal{El}(\calR[t] )}} (I-tA) \subset {\textnormal{Orb}_{\textnormal{El}(\calR[t] )}} (I-tA)\} \  .
  \]

Because $\el (\calR[t])\subset \ear$, we may also regard
  $\ear$ as a subgroup of $K_1(\calR[t])$; there,
  $\ear \subset \nk_1(\calR)$.
 %(  \eqref{elstabisinnk1}).
   (We will recall definitions in Section \ref{backgroundsec}.)

%\newcommand{\orbglr} =
%\textnormal{Orb}_{\textnormal{GL}(\calR )}}
%\newcommand{\orbglrt}{\textnormal{Orb}_{\textnormal{GL}(\calR[t] )}}
%\newcommand{\orbelr}{\textnormal{Orb}_{\textnormal{GL}(\calR )}}
%\newcommand{\orbglert}{\textnormal{Orb}_{\textnormal{GL}(\calR[t] )}}
We will show that
there is a  bijection
\begin{align}\label{main2}
\nk_1(\calR)/\ear  \ &\ \to \
\{\orbelrt (I-tB): I-tB \in \orbglrt (I-tA)\} \\ \notag
[I-tN]\ & \ \mapsto\  \orbelrt (I-t(A\oplus N)) \ .
\end{align}
%\begin{align}\label{main2}
%\textnormal{Nil}_0(\calR)\ &\ \to \
%\{\orbelrt (I-tB): I-tB \in \orbglrt (I-tA)\} \\ \notag
%[N]\ & \ \mapsto\  \orbelrt (I-t(A\oplus N)) \ .
%\end{align}
In \eqref{main2}, $B$ is a square matrix over $\Rcal$;
$N$ is a nilpotent matrix over $\Rcal$;
and $[I-tN]$ is the class in $\nk_1(\calR)$
containing $I-tN$.

For a square matrix $B$ over $\calR$, let  $
[B]_{SSE-\calR}$ denote the set of matrices SSE-$\calR$  to $B$;
similarly define
$[B]_{SE-\calR}$.
From  \eqref{seisgl} ,\eqref{sseisel} and \eqref{main2}, for
any square matrix $A$ over $\calR$
we get  a well-defined bijection (Theorem \ref{aplusn2}),
%\begin{align}
%\nil (\calR)\  & \to \  \{[B]_{SSE-\calR} \hspace{.05in} |
%\hspace{.05in} [A]_{SE-\calR} = [B]_{SE-\calR}\} \\ \notag
%[N] \ &\mapsto \  [A \oplus N]_{SSE-\calR} \ .
%\end{align}
\begin{align}
\nk_1(\calR)/\ear \  & \to \  \{[B]_{SSE-\calR} \hspace{.05in} |
\hspace{.05in} [A]_{SE-\calR} = [B]_{SE-\calR}\} \\ \notag
[I-tN] \ &\mapsto \  [A \oplus N]_{SSE-\calR} \ .
\end{align}
%\red{This correspondence can also be described in terms
%of $\nil (\calR)$ (Theorem \ref{aplusn2}).}
It is easy to check $\ear$ is trivial if $A$ is nilpotent.
There are rings with nontrivial $\nk_1(\calR)$
  such that  $\ear $ is trivial for every $A$
  (Remark \ref{goodrings}).
We will show $\ear$ is trivial if
$A$ is SE-$\calR$ to a matrix which
is  invertible or idempotent (Theorem
\ref{classesinOmega}),
    and in some other cases when $\calR$ is the
    integral group ring of a finite abelian group
    (Cor. \ref{localizationexample}).

The key to the triviality of $\ear$ for invertible or idempotent
$A$ (important for applications) is
Theorem \ref{fredholmk1injective}, which
shows that the map
$K_{1}(\calR[t]) \to K_{1}(\Omega_{+}^{-1}\calR[t])$
induced by a certain Cohn localization
$\calR[t] \to \Omega_{+}^{-1}\calR[t]$ is injective.
In the case $\Rcal$ is commutative, this can be
handled with a standard localization exact sequence.
But for the generality of all rings $\Rcal$, the proof
depends on the
 work  of  Neeman and Ranicki
%(following Schofield)
on the
$K$-theory of noncommutative localization.
For general $\Rcal$, they extended a
localization  finite exact sequence of
Schofield by a single term
(see Theorem \ref{NeemanRanicki6term}).
We need that extra term
to prove Theorem \ref{fredholmk1injective}.

The elementary stabilizer $\ear$ is not always trivial.
For $\calR$ commutative,
we show (Theorem \ref{nonvanishingstabilizers}) that
\[
 \bigcup_{A\in \mathcal M(\mathcal R )} \ear = NSK_{1}(\mathcal{R})
\]
where
$NSK_{1}(\mathcal{R}) =\{ [M] \in NK_{1}(\mathcal{R})\colon \det
(M)=1\}$. If $\calR$ is a reduced ring (one with no nonzero
nilpotent element), then $NSK_{1}(\mathcal{R}) =NK_{1}(\mathcal{R}) $.
%$$\{[G] \in NK_{1}(\mathcal{R}) \hspace{.02in} | \hspace{.02in}
%\det(G)=1\}.$$
The proof uses  Fitting's stabilization result
(Theorem \ref{WarfieldLemma});  Quillen's
localization sequence in K-theory for the localization of
$\calR[t]$ at the reverse monic polynomials; and
 Nenashev's characterization of $K_1$ of an exact category.
We leave open the problem of finding a more complete 
understanding of the elementary stabilizer (see  
Conjecture \ref{unionofesconjecture} and 
Problem \ref{elstabprob}).

 At the end of Section \ref{secstabilizer}, we
provide some context for
the statement and proof of
Theorem \ref{classesinOmega}.
In Section \ref{nilsec}, we note
  that for nilpotent matrices $N,N'$ over $\calR$,
  $[N]=[N']$ in
  $\text{Nil}_0(\rr)$ if and only if $N$ and
  $N'$ are SSE-$\calR$.

%the triviality of
%$\seort (I-tA) $ in $K_1(\calR[t])$.

%Without this result, we can establish a bijection between (i)
%$\SSE-\calR $ classes of matrices in the
%$\SE -\calR$ class of a given matrix $A$,
%and (ii) a quotient of $\nk_1 (\calR)$ (which a priori might
%depend on $A$). With the technical result,
%we can show the quotient is equal to  $\nk_1 (\calR)$
%(and is therefore independent of $A$).

This paper is entirely about matrices over rings
and related $K$-theory. However, strong motivation
for the paper comes from
symbolic dynamics (where the paper
already has a serious application \cite{BoSc3}),
as indicated in the last two subsections
of Section \ref{backgroundsec}.
The original arXiv post \cite{BoSc1}
of our paper
contained an error (see Remark \ref{meaculpa}).
The implications of that error for the applications
is discussed in Remark \ref{implicationsoferror}.

%At the end of  Section \ref{nilsec} we add a few remarks
%about SSE and $K$-theory.

%We describe some applications and
%relations with algebraic K-theory.

{\it Acknowledgements.}
We thank
Jonathan Rosenberg
for all the K-theory education and consultation.
We thank Wolfgang Steimle for the
content of Remark \ref{failedproofs}
and we thank  David Handelman
for completing the proof of
Proposition \ref{nontrivialkseoexamples}.
%\red{I dropped the first Ranicki thanks.}
%We thank
%Andrew Ranicki for guiding us to
%\cite{RanickiNeeman},  which freed
%our results  from the confines of
%rings with
%invariant basis number.
We
are grateful to A. Ranicki, J. Rosenberg and
C. Weibel for their books
\cite{Ranickibook,Rosenberg1994,WeibelBook},
without which we might not have written this paper.
Mike Boyle is happy to acknowledge support
during this work from
 the Danish National Research Foundation, through
the Centre for Symmetry and Deformation (DNRF92);
and from the
NSERC Discovery grants of David Handelman and of Thierry Giordano,
at the University of Ottawa.
Finally, we are grateful to the referee for a
  very detailed and thorough review, which has
  improved the presentation and accuracy of the paper.
%and
%from the Pacific Institute for the
%Mathematical Sciences and ? .}

\section{Background and applications}\label{backgroundsec}

In this section, we give basic definitions we need for $K$-theory,
shift equivalence and  strong shift equivalence. Then we
give a little background from symbolic dynamics (not needed for
proofs), and  summarize motivations and
applications.

\begin{notation}\label{abuse}
Let $M_{\stabone}$ be defined as in the introduction
from a finite square  matrix $M$. We regard
$M_{\stabone}$ as an
$\N \times \N$ matrix which has $M$ as its upper left
corner and is otherwise equal to the identity matrix.
In the set of $\N \times \N$ matrices,
$I$ denotes the infinite identity matrix. Thus
the direct limit semigroup
$\mathfrak M(\calR)$
 may be identified with
 the set of all $\N \times \N$ matrices
over $\calR$ equal to $I$ outside finitely many entries.
 To avoid a heavier notation, we sometimes
suppress the subscript $_{\stabone}$.
For example, if $M$ is a finite square matrix
and $U$ in $\GL(\calR)$, then
$UM$ means $UM_{\stabone}$.
When we say finite square
matrices $M,M'$ are $\gl(\Rcal )$
equivalent, we mean
 there are $U,V$ in
$\gl(\Rcal )$ such that
$UM_{\stabone} V=(M')_{\stabone} $.

%use the same
%letter to refer to a finite matrix and its infinite version.
%For example, given a finite matrix \red{$I-A$,} we let
%$I-A$ denote the matrix \red{$(I-A) _{\stabone}$.}
%% and an $\N \times \N$
%matrix $I-A$ always means a matrix $I-A_{\stabone}$.
\end{notation}
\begin{remark} \label{remarkableabuse}
If in the introduction for $p_n$ we used $M\mapsto M\oplus 0$ rather
than $M\mapsto M\oplus 1$, we would produce a more standard stable
version of $M$, which we denote $M_{\stabzero}$.
 Consistent with the
matrix interpretation of $M_{\stabone}$, we regard $M_{\stabzero}$
as an $\N \times \N$ matrix which has upper left corner $M$ and
has other entries zero. With this interpretation,
$(I_n-A)_{\stabone} = I -A_{{\stabzero}}$ .
\end{remark}

{\bf Some basic K-theory.} \label{subsec:ktheory}
Throughout this paper, a ring means a ring with unit. Unless mentioned otherwise, for $\calR$ a ring,
an $\calR$-module $M$ is a right $\calR$-module
($r: m \mapsto mr$), and matrix multiplication of vectors is
  multiplication of column vectors.
Everything in the paper would remain true if instead we used left
$\calR$ modules and
multiplication of row vectors.

We briefly review some definitions and notation.
We recommend the books
\cite{Rosenberg1994,WeibelBook} for
an introduction to algebraic K-theory.
\\
\indent Let $\Rcal$ be a ring. The group
$K_{1} (\Rcal)$ is defined by $K_{1} (\Rcal) = \GL(\Rcal)/\EL(\Rcal)$,
where $\GL(\Rcal) = \varinjlim \GL_{n}(\Rcal)$ and $\EL (\Rcal) =
\varinjlim \EL_{n}(\Rcal)$, with $\EL_{n}(\Rcal)$ the group
generated by basic elementary matrices of
size $n$ (those equal to $I$ except possibly in a single
offdiagonal entry).
If $\calR$ is commutative, then
  $\EL (\Rcal) \subset \SL(\Rcal) := \varinjlim \SL_{n}(\Rcal)$,
  and $\sk_1(\Rcal)$ denotes $\{ [M] \in K_{1}(\Rcal): \det M = 1\}$.
%\mbred{How should $\EL$ look ... right now it is a macro.}
As above, we  use $\N \times \N$ matrices
as a notation for these direct limits.
%e.g. $U$ in  $\GL_{n}(\Rcal)$ becomes $U\oplus I_{\infty}$ in
%$\GL (\calR)$.
%If $\Rcal$ is also commutative, the determinant map
%$\det:\Rcal \to \Rcal^{\times}$ is a split surjection,
%and gives a decomposition $K_{1}(\Rcal) \cong SK_{1}(\Rcal)
%\oplus \Rcal^{\times}$, where $SK_{1}(\Rcal)=\ker(\det)$, and
%$\Rcal^{\times}$ denotes the group of units in $\Rcal$.\\
The group $NK_{1}(\Rcal)$ is
the kernel of the homomorphism
$K_{1}(\Rcal[t]) \to K_{1}(\Rcal)$ induced by
the ring homomorphism
$\Rcal[t] \stackrel{t \to 0}\to \Rcal$. The
exact sequence $0 \to t\Rcal[t] \to \Rcal[t] \stackrel{t \to 0}
\to \Rcal \to 0$ is split on the right, giving a decomposition
$K_{1}(\Rcal[t]) \cong NK_{1}(\Rcal) \oplus K_{1}(\Rcal)$.

%{\bf Preparations for describing $\mathbf{NK_1(R)}$.}

%If
%$\Rcal$ is regular, then $NK_{1}(\Rcal)=0$
%orphisms(see \cite{Rosenberg1994} or \cite{WeibelBook}).
%Any Dedekind domain is regular.
%If $R$ is reduced (has no non-trivial nilpotents),
%then one also has $NK_{1}(\Rcal) \subset SK_{1}(\Rcal[t])$.
%

For a category $\mathcal{P}$ with exact sequences and small
skeleton $\mathcal{P}_{0}$, $K_{0}(\mathcal{P})$ is defined to be the free abelian group on $Obj(\mathcal{P}_{0})$, modulo the relations:\\
(1) $[P_{1}] = [P_{2}]$ if $P_{1}$ and $P_{2}$ are isomorphic in $\mathcal{P}$.\\
(2) $[P] = [P_{1}] + [P_{2}]$ if there is a short exact sequence in $\mathcal{P}$
$$ 0 \to P_{1} \to P \to P_{2} \to 0 $$
For a ring $\Rcal$, the nil category \textbf{Nil($\Rcal$)} is the
exact category whose objects are pairs $(P,f)$, where $P$ is an object
in \textbf{Proj($\Rcal$)},
the category of finitely generated projective $\calR$-modules,
and $f$ is a nilpotent endomorphism of $P$.
A morphism $h \colon (P,f)\to (Q,g)$ in
  \textbf{Nil($\Rcal$)}
 is a  morphism $h \colon P \to Q$ in
\textbf{Proj($\Rcal$)} such that \\
\[ \xymatrix{
& P \ar[r]^h \ar[d]^{f} &  Q \ar[d]^{g} \\
& P \ar[r]^h & Q }\\
\]
commutes. There is a split surjective functor \textbf{Nil($\Rcal)$}$
\to $\textbf{Proj $\Rcal$} defined by sending $(P,f)$ to $P$, and we
let $\nil (\Rcal)$ denote the
kernel of $K_{0}$(\textbf{Nil($\Rcal$)})$ \to K_{0}(\Rcal)$, giving a
decomposition $K_{0}$(\textbf{Nil($\Rcal)$}) = $K_{0}(\Rcal) \oplus
\nil (\Rcal)$.

%{\bf $\mathbf{NK_1(R)}$ and $\mathbf{Nil_0(R)}$.}
Every element of
$\nk_{1}(\Rcal)$ contains a matrix
of the form $I-tN$, with $N$ a nilpotent
matrix with entries in $\Rcal$.
It is a classic result that
the map $[I-tN] \to [N]$ defines an isomorphism $\nk_{1}(\Rcal) \to
\nil (\Rcal)$.
A theorem of Farrell  \cite{Farrell1977} shows that when $\nk_{1}(\Rcal)
\ne 0$, $\NK_{1}(\Rcal)$ is not finitely generated
as a group.
If $G$ is a finite group of order $n$, then
$\NK_{1}(\Z G)$ is trivial if $n$ is square-free
\cite{Harmon1987}, but in general may not vanish
\cite{Weibel2009}.

To appreciate that $NK_{1}(\calR )$ is often trivial,
recall that a  (left) Noetherian ring is regular if every finitely generated (left) $\Rcal$-module $M$
has a finite-type projective resolution, i.e. there
exists an exact sequence
$$0 \to P_{n} \to \cdots \to P_{0} \to M \to 0$$
with $P_{i}$ projective for all $i$.
These Noetherian regular rings form a large class, containing
rings of finite global dimension (fields,
principal ideal domains, Dedekind domains ...).
If $\calR$ is regular, then the polynomial ring
$\calR[x_1, \dots ,x_n]$
is regular.
If $\calR$ is a Noetherian regular ring, then
$NK_{1}(\calR )$ is trivial.

{\bf Cohn Localization.} \label{subsec:noncommlocal}
Cohn localization is a fundamental tool
for the study of  noncommutative rings.

\indent Let $\Sigma$ be a collection of matrices over
a ring $\calR$, $\Sigma
= \{A_{i}\}$. The Cohn localization of $\calR$ with respect to
$\Sigma$ consists of a ring
(denoted $\Sigma^{-1}\calR$) with
 a ring homomorphism $\phi\colon \calR \to \Sigma^{-1}\calR$
satisfying two properties:
\begin{enumerate}
\item
For every matrix $A$ in $\Sigma$, $\phi(A)$ is invertible in $\Sigma^{-1}\calR$.
\item
If $\gamma \colon \Rcal \to S$ is any other
ring homomorphism such that $\gamma(A)$ is invertible over $S$
for all $A \in \Sigma$, then
there is a (unique) ring homomorphism
$\delta \colon \Sigma^{-1}\calR \to S$ such that
$\gamma = \phi \circ \delta$.
%$\gamma$ factors through $\phi$:\\
%\[ \xymatrix{
%& \calR \ar[r]^{\phi} \ar[dr]^{\gamma} & \Sigma^{-1}\calR \ar[d] \\
%&  & S }\\
%\]
\end{enumerate}
The ring $\Sigma^{-1}\calR$ is thus a universal $\Sigma$-inverting
ring.
With the usual nontriviality assumption for a ring, $0\neq 1$, there might be no ring over which
the matrices in $\Sigma$ become invertible. Therefore, so that $\Sigma^{-1}\calR$ is always defined, the degenerate possibility $\Sigma^{-1}\calR = \{0\}$ is allowed.
Then $\Sigma^{-1}\calR$ exists
and is essentially unique (see \cite{Schofield}
or \cite{CohnBook}).  \\
\indent The Cohn localization can also be constructed given
a collection of morphisms between finitely generated projective
$\calR$-modules in an analogous fashion. Given such a collection
$\Sigma$, call a ring morphism $\calR \to \mathcal{S}$
$\Sigma$-inverting if
$\sigma \otimes 1\colon P \otimes_{\calR}  \mathcal{S}
\to Q \otimes_{\calR} \mathcal{S}$ is an $\mathcal{S}$-module
isomorphism for every $\sigma \colon P \to Q$ in $\Sigma$. Then the
noncommutative localization is a ring $\Sigma^{-1}\calR$
with a $\Sigma$-inverting map $\calR \to \Sigma^{-1}\calR$ such that $\Sigma^{-1}\calR$ is universal with respect to $\Sigma$-inverting maps, analogous to $(2)$ above. \\
\indent More details regarding the general construction of $\Sigma^{-1}\calR$ may be found in 7.2 of \cite{CohnBook}. \\
\indent Given $\calR$, define $\Omega_{+}$ to be the collection of $\calR[t]$-module homomorphisms satisfying the following:
\begin{enumerate}
\item
Each $f \in \Omega_{+}$ is an $\calR[t]$-module homomorphism $f \colon P \to Q$ between some finitely generated $\calR[t]$-modules $P,Q$.
\item For every $f \in \Omega_{+}$, $f$ is injective, and $\coker(f)$ is a finitely generated projective $\calR$-module.
\end{enumerate}
Following \cite{Ranickibook}, we refer to $\Omega_{+}$ as the set of Fredholm homomorphisms. The localization $\Omega_{+}^{-1}\calR[t]$ has the property that the map $\calR[t] \to \Omega_{+}^{-1}\calR[t]$ is injective \cite[Prop. 10.7]{Ranickibook}.\\
\indent One can alternatively construct the Fredholm localization
using matrices. Let $\Omega_{+}^{\textnormal{mat}}$ denote the set
of
matrices $A$ over $\calR[t]$ such that (with $A$ $m \times n$)
the induced map on free $\calR[t]$-modules
$\calR[t]^{n} \stackrel{A}\to \calR[t]^{m}$ is injective and
$\coker(A)$ is a finitely generated projective
$\calR$-module. We refer to $\Omega_{+}^{\textnormal{mat}}$ as the set of Fredholm
matrices. That the localizations $\Omega_{+}^{-1}\calR[t]$ and
 $(\Omega_{+}^{\textnormal{mat}})^{-1}\calR[t]$ coincide is easy to check. We may occasionally abuse notation and write $\Omega_{+}$ in place of $\Omega_{+}^{\textnormal{mat}}$ when it is clear that matrices are being considered. \\
\indent An alternative construction of $\Omega_{+}^{-1}\calR[t]$ may
be described as follows. Let $\Omega_{M}$ denote the set of monic
matrices over $\calR [t]$,  i.e. the square matrices
%$A = A_{0} + A_{1}t + \cdots A_{d}t^{d}$
$A=\sum_{i=0}^d A_it^i$
with the $A_i$  matrices over $\calR$
%with the highest degree term equals the identity, so
such that  $A_{d}  $ is the identity matrix.
Note that $\Omega_{M} \subset \Omega_{+}$.
In fact, the two localizations coincide
\cite[Prop. 10.7]{Ranickibook}:
$\Omega_{+}^{-1}\calR[t] = \Omega_{M}^{-1}\calR[t]$.
\\

{\bf Shift equivalence} \label{subsec:shiftequivalence}
%\indent
%\cite{Boyle1991}
%\red{Mike: what is the reference here for?}
Two square matrices $A,B$ over $\Rcal$ are called shift equivalent over $\Rcal$ (SE-$\Rcal$) if there exists a positive integer $l$ (the lag) and matrices $R,S$ over $\Rcal$ such that
$$RS=A^{l}, SR=B^{l}, RB=AR, BS=SA.$$
While shift equivalence is an equivalence relation, lag one shift equivalence is not. The transitive closure of lag one shift equivalence is called strong shift equivalence, so two square matrices $A,B$ over $\Rcal$ are strong shift equivalent over $\Rcal$ (SSE-$\Rcal$) if there is a chain of lag one shift equivalences between them.
{\bf
Strong shift equivalence} \label{subsec:strongshiftequivalence}
Let $\Rcal$ be a ring.
The nature of SSE-$\Rcal$ as a kind of stabilized
version of similarity over $\Rcal$
  is shown by the following characterization
from \cite{MallerShub1985}.
The relation SSE-$\Rcal$ is generated by
two relations:\\
 (1) Similarity over $\Rcal$: $A=U^{-1}BU$.\\
 (2) ``Zero extensions'':
\[
\begin{pmatrix} A & U \\ 0&0
\end{pmatrix}
 \sim
A
\sim
\begin{pmatrix} A & 0 \\ U&0
\end{pmatrix}
\]
Similarity over $\calR$ implies SSE-$\calR$,
  since $A=U^{-1}BU$ gives $A=VU$, $B=UV$ with $V=U^{-1}B$.
  Each type of zero extension respects SSE-$\Rcal$, because
  \begin{align*}
    A= \begin{pmatrix} A & U \end{pmatrix}
    \begin{pmatrix} I \\ 0 \end{pmatrix}\ ,&
  \qquad
  \begin{pmatrix} I \\ 0 \end{pmatrix} \begin{pmatrix} A & U \end{pmatrix} = \begin{pmatrix}A & U \\ 0 & 0 \end{pmatrix} \\
A = \begin{pmatrix} I & 0 \end{pmatrix} \begin{pmatrix} A \\ U \end{pmatrix}\ ,&
\qquad
\begin{pmatrix}A \\ U \end{pmatrix} \begin{pmatrix} I & 0 \end{pmatrix} = \begin{pmatrix} A & 0 \\ U & 0 \end{pmatrix}\ .
\end{align*}
\iffalse
\mbred{$(1)$ and $(2)$ together imply SSE-$\calR$, since $A=U^{-1}BU$ gives $A=VU$, $B=UV$ with $V=U^{-1}B$, and
  $$A= \begin{pmatrix} A & U \end{pmatrix}\begin{pmatrix} I \\ 0 \end{pmatrix},
  \qquad
  \begin{pmatrix} I \\ 0 \end{pmatrix} \begin{pmatrix} A & U \end{pmatrix} = \begin{pmatrix}A & U \\ 0 & 0 \end{pmatrix}$$
while
$$A = \begin{pmatrix} I & 0 \end{pmatrix} \begin{pmatrix} A \\ U \end{pmatrix},
\qquad
\begin{pmatrix}A \\ U \end{pmatrix} \begin{pmatrix} I & 0 \end{pmatrix} = \begin{pmatrix} A & 0 \\ U & 0 \end{pmatrix}$$}
\fi
Conversely, given $A=UV,\ B=VU$
we have a similarity:
\begin{equation} \label{ssesim}
\begin{pmatrix} I & 0 \\ V & I
\end{pmatrix}
\begin{pmatrix} A & U \\ 0 & 0
\end{pmatrix}
\ =\
\begin{pmatrix} 0 & U \\ 0 & B
\end{pmatrix}
\begin{pmatrix} I & 0 \\ V & I
\end{pmatrix}
\end{equation}

{\bf Antecedents.}
The connection between
$\nil (\calR)$ and
$\SSE -\calR$ grew for us out of the
``positive K-theory''
\cite{B02posk,BW04} approach to
 classification problems
in symbolic dynamics. That approach grew out of earlier work,
especially \cite{bgmy,S7,S6}, and  Wagoner's background
in algebraic K-theory.  Some classification
problems in symbolic dynamics can be presented,
for a suitable ordered ring $\calR$,
as the problem of classifying
square matrices $A,B$ over $\calR$
up to  $\SSE-\calR_+$.
In the most important example, for the classification of shifts of finite
type,  Williams used
$\calR = \Z_+$ \cite{Williams73}.
For the classification of group extensions of shifts of finite
type by a finite group $G$ for example, Parry used
$\calR = \Z_+G$ \cite{BS05, BoSc2}.
For a group ring $\calR =\Z G$,
the relation $\SSE-\Z_+G$ of $A$ and $B$ is equivalent to
 ``positive''  equivalence
of the matrices $I-tA$ and $I-tB$ \cite[Theorem 7.2]{BW04}.
Here a positive equivalence is a certain
type of
$\EL (\ZG[t])$ equivalence
$U (I-tA)V = I-tB$ (see \cite{B02posk,BW04,BoSc2}
for definitions and explanation).
This by analogy raises the question for rings answered by
\eqref{sseisel}.

%For an integral group ring $\Z G$ with positive set $\Z_+G$,
%a positive matrices $A,B$ over $\calR$
%exists if and only if there is a positive
%equivalence of the matrices $I-tA$ and $I-tB$. }
%The $\el(\calR)$ equivalence
%given by setting $t=1$
%gives  very strong
%invariants of flow equivalence
%\cite{BS05}.}

The elementary stabilizer as a subgroup of $K_1(\calR)$
appeared in a related context in \cite{BS05}
(see Remark \ref{festabBS05remark}).

{\bf Motivation and applications.}
The results in this paper have been used to
answer (in the negative)  a question of
Parry \cite[Sec. 4.4]{pst2009} about a possible extension of
Liv\v{s}ic theory to finite group extensions of
shifts of finite type, and have significantly clarified the
structure of their algebraic invariants \cite{BoSc2}.
They have also been used to show that
two old conjectures about the
algebraic structure of nonnegative
matrices are equivalent \cite{BoSc3}.

\begin{remark}\label{implicationsoferror}
The papers \cite{BoSc3,BoSc2} appealed to
the incorrect claim in our original arXiv post \cite{BoSc1}
that $\ear$ is always trivial
(see Remark \ref{meaculpa}). However, the arguments  of
  \cite{BoSc3} go through unchanged, with
  appropriate reference to
  Theorem \ref{aplusn} in place of
  \cite[Theorem 2.1]{BoSc3}.
  In \cite{BoSc2}, after replacing
  Theorem 2.2(2) with
  a reference to  Theorem \ref{aplusn} below,
  the theorems and proofs remain correct, with
  one amendment:
  in Theorem 6.4 of \cite{BoSc2},
  there should be added the assumption that
  the elementary stabilizer
$ E(A,\Z G)$
(see (Defn. \ref{eardefn}))
  is trivial.  By Theorem \ref{classesinOmega},
for every finite group $G$,
 $ E(A,\Z G)$  is trivial for many
matrices $A$, e.g. for every $A$ invertible
over $\Z G$ (also note Cor. \ref{localizationexample}).
Thus the revised  Theorem 6.4
  still provides for every finite group $G$ with
nontrivial $NK_1(\Z G)$ many cases in
  which the answer to Parry's
  question is decisively no.
  \end{remark}

In \cite{BKR2013}, a three part program
for understanding SSE for positive real
matrices was proposed. One part,
understanding the refinement of
SSE by SE for subrings of $\R$,
is addressed by the current paper.

One ``application'' of a result describing
the refinement of SE by SSE is that one
acquires constraints on what proofs
might possibly work. For example,
the main result of \cite{BKR2013}
had a hypothesis of SSE (not SE) of two
matrices over a subring of $\R$. We
now know that hypothesis is not an
artifact of the proof.

The classification problem for shifts of finite type
is a central open problem for symbolic dynamics.
Wagoner used  $K_2$ of the dual numbers as
an ingredient for producing a counterexample
to Williams' conjecture that SE-$\Z_+$  implies
SSE-$\Z_+$, and suggested further possible
connection between the classification problem
and algebraic K-theory  \cite{Wagoner2000,
Wagoner2004}. The current paper
is, we hope, a step toward understanding that
connection.

\section{$K_1(\calR[t])\to K_1(\Omega^{-1}_{+}\calR[t])$ is injective} \label{injectivesec}

The main purpose of this section is to prove Theorem
\ref{fredholmk1injective}, which we need to prove
  Theorem \ref{trivialkseo}.
%But to establish Theorem \ref{aplusn2} for a completely
%general ring,  we rely on Theorem \ref{fredholmk1injective}.

\begin{theorem}
\label{fredholmk1injective}
  Let $\Omega_{+}$ denote the set of Fredholm
homomorphisms of finitely generated projective modules over $\calR[t]$. Then the natural map
$$K_{1}(\calR[t]) \to K_{1}(\Omega_{+}^{-1}\calR[t])$$
induced by $\calR[t] \to \Omega_{+}^{-1}\calR[t]$ is injective.
\end{theorem}
%\annotation{\red{If we are going to use Fredholm homomorphisms
%%    \ref{fredholmk1injective}, then something clear should be said about them
%  back where the Fredholm matrices were introduced.}}

\iffalse
Before proving Theorem \ref{fredholmk1injective} in the general case,
we give first a proof for the case $\calR$ is commutative, using a
standard
 long exact localization sequence in K-theory \eqref{comlongexact}.
The proof in the general case is a consequence of the
fundamental work of Neeman, Ranicki and Schofield.
In some places, we provide perhaps more
explanation than experts might need, in an effort to make the material
more widely accessible.
\fi
The proof  of Theorem \ref{fredholmk1injective} for
general $\calR$ requires us to delve into the proofs
behind the Neeman and Ranicki
results on the
$K$-theory of Cohn localizations. Before going to that more
difficult work,
we'll give the  (shorter) proof for the case that $\calR$ is commutative.
The proof for this case
uses the  standard K-theory localization exact sequence
\eqref{comlongexact}
with claims appealing to
standard references.
%and
%can be checked by references to standard statements.
After that, we will be better positioned to
understand (and appreciate) how the work
of Neeman and Ranicki fits in.
%In this section,
We provide more
explanation \and reference
than experts might need, in an effort to make the material
more widely accessible and easily checked.

%\emph{Proof of Theorem \ref{fredholmkinjective} when $\calR$ is commutative}:
%\subsection{The Commutative Case} \label{subsec:commutativeproof}
%\emph{In this \red{subsection}, $\calR$ is assumed to be commutative.}\\
%\label{subsec:commutativeproof}

\begin{center}{\bf The Commutative Case} \end{center}

In this subsection, $\calR$ is assumed to be commutative.
\begin{definition}
For a ring $\calR$, we consider the following exact categories:
\begin{enumerate}
\item
$\mathcal{H}_{1}(\calR)$ is the exact category
  whose objects are $\calR$-modules which have a resolution
of length $\le 1$ by finitely generated projective
  $\calR$-modules,
and whose morphisms are the $\calR$-module homomorphisms between
them.
\item
%$\calR$-modules
%which have a resolution
%\red{of length $\le 1$
% by finitely generated projective $\calR$-modules
%which
%are $S$-torsion (i.e. $sM = 0$ for some $s \in S$),
%with morphisms being $\calR$-module homomorphisms.
Given a multiplicatively closed set $S \subset \calR$ of non-zero
divisors, $\mathcal{H}_{1,S}(\Rcal)$ denotes the
full subcategory of $\mathcal{H}_{1}(\calR)$ whose objects
 are the objects of $\mathcal{H}_{1}(\calR)$ which are
$S$-torsion modules
(i.e. $sM = 0$ for some $s \in S$).

\end{enumerate}
\end{definition}
%\red{ I think it would be good to rewrite this paragraph to include
%explicitly the relation of finitely generated projective and
%finitely presented, and also Weibel's definition (or that could be in a footnote
%if it works better). }
Our use of the term exact category matches the standard one, as in
\cite[Definition II.7.0]{WeibelBook}. The notation
$\mathcal{H}_{1}(\calR)$ was chosen to match Weibel's K-Book
\cite[Definition II.7.7]{WeibelBook}.
It follows from the Resolution Theorem
  \cite[V.3.1]{WeibelBook}
that the inclusion of Proj$\calR$ into $\mathcal{H}_{1}(\calR)$
induces an isomorphism $\rho\colon K_{1}(\calR) \to K_{1}(\mathcal{H}_{1}(\calR))$.
%Obviously $\mathcal{H}_{1,S}(\calR)$ may be naturally
%identified as a subcategory of $\mathcal{H}_{1}(\calR)$.\\
The category $\mathcal{H}_{1,S}(\Rcal)$ appears in the standard long
exact sequence
%(which can be found in
\cite[V.7.1]{WeibelBook}
\begin{equation}
\label{comlongexact}
 \cdots \to K_{n}(\mathcal{H}_{1,S}(\calR)) \to K_{n}(\calR) \to
 K_{n}(S^{-1}\calR) \to \cdots
\end{equation}
which holds for the localization of a commutative ring $\calR$ at a multiplicatively closed set $S$ of
central non-zero divisors.\\
\indent Let $S_{+}$ denote the collection of monic polynomials in
$\calR[t]$, i.e. polynomials of the form $p(t) =
\sum_{i=0}^{n}a_{i}t^{i}$ with $a_{n} = 1$. The set $S_{+}$ is a
multiplicatively closed set
of non-zero divisors.
%, and we may consider the localization $S_{+}^{-1}\calR[t]$. Using
Replacing $\calR$ and $S$ in \eqref{comlongexact} with
$\calR[t]$ and $S_{+}$, we get the exact sequence
\begin{equation}\label{commutkseq}
\cdots \to K_{n}(\mathcal{H}_{1,S_{+}}(\calR[t])) \to K_{n}(\calR[t]) \to
 K_{n}(S_{+}^{-1}\calR[t]) \to \cdots
\end{equation}
To prove Theorem \ref{fredholmk1injective} for $\calR$ commutative, it
is now sufficient to show that the map
$\alpha\colon K_{1}(\mathcal{H}_{1,S_{+}}(\calR[t])) \to K_{1}(\calR[t])$ in
\eqref{commutkseq} is the zero map. This map factors through the map
induced by the inclusion functor (see the proof of \cite[V.7.1]{WeibelBook})
$j\colon\mathcal{H}_{1,S_{+}}(\calR[t]) \to
\mathcal{H}_{1}(\calR[t])$,
%\annotation{\red{Is there an explicit reference for ``induced by the
%inclusion functor'' here?}}
giving a diagram
\[ \xymatrix{
 K_{1}(\mathcal{H}_{1,S_{+}}(\calR[t])) \ar[r]^{K_{1}(j)} \ar[dr]^{\alpha} & K_{1}(\mathcal{H}_{1}(\calR[t])) \ar[d]^{\rho^{-1}} \\
 & K_{1}(\calR[t]) \\ }
\]
in which
%where
the vertical map is the inverse to the isomorphism
$K_{1}(\textnormal{Proj}\calR[t]) \to K_{1}(\mathcal{H}_{1}(\calR[t]))$
given by the Resolution Theorem.
%(\cite[V.3.1]{WeibelBook}).
It suffices then to show the map
$$K_{1}(j)\colon K_{1}(\mathcal{H}_{1,S_{+}}(\calR[t])) \to K_{1}(\mathcal{H}_{1}(\calR[t]))$$
is the zero map.

%$\coloneqq$ this commands requires loading the package mathtools

For $M$ in
$\mathcal{H}_{1,S_{+}}(\calR[t])$,
define $\eta(M)  =
M \otimes_{\calR}
\calR[t]$. The right $\calR[t]$-module $\eta (M)$ carries
no memory of the
original action of $t$ on $M$; as an $\calR$-module,
it is isomorphic to a direct sum of countably many copies of
$M$.
 A well known argument
%\annotation{Is there also a K-book ref for this
%well known argument? - unfortunately seems like no}
%(which may be found in
\cite[p. 441]{Grayson1977} shows that every object $M$ in
$\mathcal{H}_{1,S_{+}}(\calR[t])$ is finitely generated projective as
an $\calR$-module.
For $M$ in
$\mathcal{H}_{1,S_{+}}(\calR[t])$, it follows that
$\eta(M)$
is a finitely generated projective  $\calR[t]$-module,
and hence lies in
$\mathcal{H}_{1}(\calR[t])$.
% (in fact in Proj$\calR[t]$).
Let $\eta$ also denote the functor
$ \mathcal{H}_{1,S_{+}}(\calR[t]) \to
\mathcal{H}_{1}(\calR[t])$ which is
$M\mapsto \eta (M)$ on objects and
$f\mapsto f \otimes_{\calR} \textnormal{id}$ on morphisms.
%It is elementary that
The functor $\eta$ is exact, since $\calR[t]$ is free as an $\calR$-module.
%As follows:
%An exact sequence $0\to A \to B \to C \to 0$
%of-modules from $\mathcal{H}_{1}(\calR[t])$
%is exact considered as a sequence of $\calR$-modules.
%Since $\calR[t]$ is free as an $\calR$-module,
%the sequence $0\to \eta (A) \to \eta (B) \to \eta (C) \to 0$
%is exact considered as a sequence of $\calR$-modules.
%Therefore it is exact as a sequence of $\calR [t]$-modules.

\indent Given $M \in \mathcal{H}_{1,S_{+}}(\calR[t])$, let $f_{M}$
denote the endomorphism of $M$ induced by the $\calR[t]$-module
structure of $M$ (so, $f_{M}(x) = x \cdot t$).
\iffalse
From the discussion above we have the exact functor $\eta \colon
\mathcal{H}_{1,S_{+}}(\calR[t]) \to \mathcal{H}_{1}(\calR[t])$, and we
\fi
Let $\pi_M \colon \eta (M) \to M$ be the
 $\calR [t]$
module homomorphism such that
$\pi_M  \colon x\otimes t^i \mapsto (f_M)^i(x)$, for
$i$ in $\mathbb Z_+$.
Recall $j \colon \mathcal{H}_{1,S_{+}}(\calR[t]) \to
\mathcal{H}_{1}(\calR[t])$ denotes the inclusion functor.
For  morphisms $\psi \colon A\to B$ in
$\mathcal{H}_{1,S_{+}}(\calR[t])$,
we define transformations of
functors,
$\mathcal{F} \colon \eta \mapsto
\eta $
and $\mathcal{G} \colon \eta \mapsto j$,
by the following
commutative diagrams
of $\calR[t]$-module homomorphisms,
\[
\xymatrix{
\eta (A)
\ar[d]_{\mathcal F (A)}
\ar[r]^{\eta (\psi )}
&
\eta (B) \ar[d]^{\mathcal F (B)}
&
 \ar@{}[d]|{\ \ \text{=}}
&&
 A \otimes_{\calR} \calR[t]
\ar[r]^{\psi\otimes \text{id}}
\ar[d]_{\text{id} \otimes t-f_A\otimes \text{id}}
&
B \otimes_{\calR} \calR[t]
\ar[d]^{\text{id} \otimes t-f_B\otimes \text{id}}
\\
\eta (A)
\ar[r]^{\eta (\psi )}
&
\eta (B)
&& &
A \otimes_{\calR} \calR[t]
\ar[r]^{\psi\otimes \text{id}}
&
B \otimes_{\calR} \calR[t]
}
\]
and
\[
\xymatrix{
\eta (A)
\ar[d]_{\mathcal G (A)}
\ar[r]^{\eta (\psi )}
&
\eta (B) \ar[d]^{\mathcal G (B)}
&
 \ar@{}[d]|{\ \ \ \text{=}}
&
A \otimes_R \calR[t]
\ar[r]^{\psi\otimes \text{id}}
\ar[d]_{\pi_A}
&
B \otimes_{\calR}\calR[t]
\ar[d]^{\pi_B}
\\
j(A)
\ar[r]^{j (\psi )}
&
j (B)
 &&
A
\ar[r]^{\psi}
&
B
&
}
\]
Because the vertical arrows do not depend on $\psi$,
$\mathcal F$ and $\mathcal G$ are natural transformations.
%$\stackrel{t-f_{M}}\to \eta(M)$
Also $\eta \stackrel{\mathcal{F}}\rightarrowtail \eta
 \stackrel{\mathcal{G}}\twoheadrightarrow j$ is
a short exact sequence of
 functors since for any $M \in \mathcal{H}_{1,S_{+}}(\calR[t])$, the
 sequence
\[
\xymatrix{
0\ar[r] &  M\otimes_{\calR}\calR[t]\ar[r]^-{t-f_{M}} &
M\otimes_{\calR}\calR[t]\ar[r]^-{\pi_M} & M\ar[r] &0
}
\]
%$$0 \to M \otimes_{\calR}
%\calR[t] \stackrel{t-f_{M}}\rightarrowtail M
%\otimes_{\calR} \calR[t] \stackrel{\pi}\twoheadrightarrow
%M \to 0$$
(with $t-f_M\colon x\otimes t^i\mapsto
x\otimes t^{i+1}
-f_M(x) \otimes t^i $)
is exact (see e.g. \cite[p. 630]{BassBook}).
Let $K_{1}(\eta), K_{1}(j)$ denote the corresponding maps on
K-theory.
Because
$\eta \stackrel{\mathcal{F}}\rightarrowtail \eta
 \stackrel{\mathcal{G}}\twoheadrightarrow j$ is a short exact sequence
of exact functors of exact categories,
it follows from the Additivity Theorem
\cite[V.1.2]{WeibelBook} that $K_{1}(\eta) = K_{1}(\eta) + K_{1}(j)$. Thus $K_{1}(j)$ is the zero map. This concludes the proof of Theorem \ref{fredholmk1injective} in the case $\calR$ is commutative.
\begin{remark}
In the commutative case, the injectivity of $K_{1}(\calR[t]) \to K_{1}(S_{+}^{-1}\calR[t])$ may also be deduced using an argument of Grayson, found in \cite[Corollary 6]{Grayson1977}. As described in \cite[Corollary 6]{Grayson1977}, one constructs a Mayer-Vietoris sequence that splits up, analogous to the proof of the Fundamental Theorem concerning $K_{1}(\calR[t,t^{-1}])$ as found in \cite[p.20]{Grayson1976}.
\end{remark}

%\subsection{The General Case} \label{subsec:generalcaseproof}
\begin{center} {\bf The General Case} \end{center}

From here on,
 we do not assume the ring $\calR$ is commutative.
Before proving the general case of Theorem \ref{fredholmk1injective},
we present the
necessary material from
\cite{Neeman2,RanickiNeeman}.
%Recall that an $\calR$-module $M$ is finitely presented if there exists a short exact sequence
%$$0 \to K \to \calR^{n} \to M \to 0$$
%for some finitely generated $\calR$-module $K$.

\begin{definition} \label{Edefinition}
Let $\Sigma = \{\sigma_{i}\}$ be a collection of monomorphisms between
finitely generated projective $\Rcal$-modules.
The exact category $\mathcal{E}=\mathcal{E}(\Sigma )$ is defined to be the full subcategory of
$\mathcal{H}_{1}(\calR)$ determined by the following conditions:
\begin{enumerate}
\item For every $\sigma \in \Sigma$, \coker($\sigma$) lies in $\mathcal{E}$.
\item
If $0 \to M_{1} \to M_{2} \to M_{3} \to 0$  is a short exact
  sequence
of objects in $\mathcal{H}_{1}(\calR)$
such that  two of the objects $M_{1}, M_{2}, M_{3}$ lie in
$\mathcal{E}$, then so does the third.
\item $\mathcal{E}$ contains all direct summands of its objects.
\item $\mathcal{E}$ is minimal, subject to (1),(2) and (3).
%(i), (ii), (iii).
\end{enumerate}
\end{definition}
Following \cite{Neeman2}, we refer to the objects in the category
$\mathcal{E}(\Sigma)$ as $(\calR,\Sigma)$-torsion modules. When the
collection $\Sigma$ is clear, we may simply refer to $\mathcal{E}$
instead of $\mathcal{E}(\Sigma)$.
Note that in
 Definition \ref{Edefinition} we have used
$\mathcal{H}_{1}(\calR)$
in place of the
 category of all finitely presented $\calR$-modules of projective
 dimension $\le 1$ in
\cite{Neeman2}.
The two definitions are equivalent, because
the category $\mathcal{H}_{1}(\calR)$ and the category of finitely
presented modules of projective dimension $\le 1$ coincide: given a
finitely presented module $M$ of projective dimension less than or
equal to one, one may always construct a resolution of length one or
less by finitely generated projective modules \cite[4.1.6]{WeibelBook}.
 \\

The next theorem will not be used directly, but helps provide context for the torsion category
$\mathcal{E}$ defined above, so we include it.
%A proof can be found  \cite[Corollary 3.3]{Neeman2}.}
%\annotation{\red{For Theorem \ref{explained}, it seems we can just
%    quote
%the explicit statement in Neeman, yes?
%He has the proof reference right by that.}}
\begin{theorem}\label{explained}\cite[Proposition 0.7]{Neeman2}
  Assume for all $\sigma \in \Sigma$ that $\sigma$ is a monomorphism, and let $\mathcal{E} = \mathcal{E}(\Sigma)$ be as in
Definition \ref{Edefinition}.
  Then an $\Rcal$-module $M$ belongs to $\mathcal{E}$ iff\\
(i) $M$ is finitely presented with projective dimension $\le 1$, and \\
(ii) $\{\Sigma^{-1}\Rcal\} \otimes_{\calR} M$ and $Tor_{1}^{\calR}(\Sigma^{-1}\calR,M)$ both vanish.
\end{theorem}
%\annotation{I changed without coloring a few trivialities in Theorem
%  \ref{explained} to math Neeman more closely.}

When $\calR$ is commutative and $S \subset \calR$ is a
multiplicatively closed set of non-zero-divisors, we let $\Omega_{S}$
denote the collection of all homomorphisms $f_{s} \colon \calR \to
\calR$ given by
$f_{s} \colon
x \mapsto xs$, with $s \in S$. In this
case the
%noncommutative
Cohn
localization $\Omega_{S}^{-1}\calR$ coincides with the standard
commutative localization $S^{-1}\calR$, and $\mathcal{E}(\Omega_{S})$
agrees with $\mathcal{H}_{1,S}(\calR)$. Indeed, in the commutative
case $S^{-1}\calR$ is flat, so we always have
$Tor_{1}^{\calR}(S^{-1}\calR,M) = 0$, and for a
nontrivial finitely generated $\calR$-module $M$, $S^{-1}\calR \otimes_{\calR} M = 0$ iff there exists $s \in S$ such that $Ms = 0$.\\

The following theorem is the main tool we use to prove the injectivity
of the map $K_{1}(\calR[t]) \to K_{1}(\Omega_{+}^{-1}\calR[t])$. The
sequence \ref{mainsequence}, without the leftmost map, was established by Schofield in
\cite{Schofield}. The extension to include the term $K_{1}(\mathcal{E}) \to
K_{1}(\calR)$, which is critical for our application,
is due to
Neeman and Ranicki;
Theorem \ref{NeemanRanicki6term}
is a combination of
\cite[Theorem 0.5]{Neeman2}
and the result stated as Theorem \ref{sixR} below.
%\annotation{Since Theorem \ref{NeemanRanicki6term} is the main
%theorem we appeal to, I tried  to give as precise a reference as
%possible. }
\begin{theorem}
\label{NeemanRanicki6term}
\cite[p. 789]{Neeman2}
Let $\calR$ be a ring, and $\Sigma$ be a collection of monomorphisms
between finitely-generated projective $\calR$-modules. Let
$\mathcal{E} = \mathcal{E}(\Sigma)$ denote the torsion category of
Definition 3.2. Then there is an exact sequence
%\annotation{\red{I changed some * references to numbered labels.}}
\begin{equation} \label{mainsequence}
K_{1}(\mathcal{E}) \to K_{1}(\Rcal) \to K_{1}(\Sigma^{-1}\Rcal) \to
K_{0}(\mathcal{E}) \to K_{0}(\Rcal) \to
K_{0}(\Sigma^{-1}\Rcal)
%\tag{$*$}
\end{equation}
\end{theorem}
\begin{remark}
\normalfont{
Neeman and Ranicki \cite{RanickiNeeman} extended
\eqref{mainsequence} to
\[
\cdots \to K_{n}(\mathcal{E}) \to K_{n}(\Rcal) \to K_{n}(\Sigma^{-1}\Rcal) \to
K_{n-1}(\mathcal{E})\to \cdots
\]
 for all $n>1$
%infinitely far to the left
under the
  hypothesis that the localization $\Sigma^{-1}\calR$ is \emph{stably
    flat}:
for all $n\geq 1$ the group
$\text{Tor}_n^{\calR}(\Sigma^{-1}\calR  ,\Sigma^{-1}\calR  )$
vanishes.  The six term version
\eqref{mainsequence}
has no stably flat requirement. We have no need of the full long exact in the present paper.}
\end{remark}

By Theorem \ref{NeemanRanicki6term}, to prove the injectivity of
$K_{1}(\calR[t]) \to K_{1}(\Omega_{+}^{-1}\calR[t])$ it is sufficient
to show the map $K_{1}(\mathcal{E}) \to K_{1}(\calR[t])$ in
\eqref{mainsequence} is zero. For
this, we will need a more detailed examination of the original
sequence from
\cite[Corollary 4.9]{RanickiNeeman}.
Definitions of maps
in $\eqref{mainsequence}$ involve identifications of various
groups, and we take care to track through these
identifications. We do this for general $\Sigma$ at first,
specializing to the case of interest
($\Sigma = \Omega_{+}$, the Fredholms) at a later point. \\

\indent Recall that a Waldhausen category consists of a category with
a subcategory of morphisms called cofibrations, along with a
distinguished family of morphisms called weak equivalences, satisfying
some axioms, which may be found in
\cite[Definition II.9.1.1]{WeibelBook}. We let $C_{b}(\textnormal{Proj} \calR)$ denote the following Waldhausen category:
\begin{enumerate}
\item The objects are bounded chain complexes of finitely generated projective $\calR$-modules
\item The morphisms are chain maps
\item The cofibrations are degree-wise split monomorphisms
\item The weak equivalences are the quasi-isomorphisms,
 i.e. the chain maps inducing an isomorphism on homology in every degree.
\end{enumerate}
The only Waldhausen categories which will be considered in this article are full subcategories of the category of chain complexes over some exact category, where the morphisms are chain maps, the cofibrations are degree-wise split monomorphisms, and the weak equivalences are quasi-isomorphisms. \\
%Note that in \cite{NeemanRanicki}, \cite{Neeman2}, the term perfect complex is used to refer to the objects of $C_{b}(Proj \calR)$.
\indent For an exact category $\mathcal{A}$ or Waldhausen category
$\mathcal{B}$, we let $K(\mathcal{A})$ and $K(\mathcal{B})$ denote the
corresponding K-theory spaces, as in \cite[IV.6.3 and IV.8.4]{WeibelBook}.
For a topological space $X$, let $\pi_n(X)$ denote the $n$th homotopy
group.
 By definition, $K_{n}(\mathcal{A}) =
 \pi_{n}(K(\mathcal{A}))$, and $K_{n}(\mathcal{B}) =
 \pi_{n}(K(\mathcal{B}))$. Since the definitions agree in the
 case $\mathcal{B}$ is exact
\cite[IV.8.6]{WeibelBook},  we do
 not distinguish, and use the same $K(\mathcal{A})$ and
 $K(\mathcal{B})$ for both.
%\\ \indent

We will make use of the following theorem.
\begin{theorem}[Gillet-Waldhausen] \label{gwtheorem}
Let $\mathcal{A}$ be an exact category, closed under taking kernels of
surjections. Then the exact monomorphism $\mathcal{A} \hookrightarrow
C_{b}(\mathcal{A})$,
%\annotation{I replaced inclusion with monomorphism to be formally
%  correct--yes??? }
taking an object $M$ to the chain complex
%\annotation{\red{Is this ``containing'' standard terminology?}}
which is $M$ in degree 0 and is zero elsewhere, induces a homotopy equivalence $K(\mathcal{A}) \stackrel{\sim}\to K(C_{b}(\mathcal{A}))$, and hence isomorphisms $K_{n}(\mathcal{A}) \stackrel{\cong}\to K_{n}(C_{b}(\mathcal{A}))$. %For $n=0$, the inverse isomorphism $K_{0}(C^{perf}(\mathcal{A})) \to K_{0}(\mathcal{A})$ is given by the euler characlalteristic $\chi$.
\end{theorem}
A proof of Theorem \ref{gwtheorem}
%3.6
may be found in \cite[V.2.2, II.9.2.2]{WeibelBook}. \\
%\red{Scott: to get `` Theorem \ref{gwtheorem}'' I replaced
%  text 3.6 with a reference to a label I added -- please check that I
%guessed right.}

Let $\Sigma = \{\sigma_{i}\}$ denote a collection of morphisms between
finitely generated projective $\Rcal$-modules. Note that each $\sigma
\in \Sigma$ may be considered in
  $C_{b}(\textnormal{Proj}\Rcal)$ as the complex
\begin{equation}\label{sigmachains}
%\tag{$**$}
\cdots \to 0 \to P \stackrel{\sigma} \to Q \to 0 \cdots
\end{equation}
with $P,Q$ in degrees $0,1$ and modules in all other degrees zero.\\
\indent By a Waldhausen subcategory $\mathcal{A} \subset \mathcal{B}$ of a Waldhausen category $\mathcal{B}$ we mean a subcategory $\mathcal{A} \subset \mathcal{B}$ which is also a Waldhausen category, satisfying:
\begin{enumerate}
\item
the inclusion functor $\mathcal{A} \to \mathcal{B}$ is exact,
%\annotation{\red{In ``i.e. preserves'', do you really mean exactness
%is equivalent to all of those preservations? - yes}}
i.e. preserves all of the following: zero, cofibrations, weak equivalences, and pushouts along cofibrations,
\item
the cofibrations in $\mathcal{A}$ are the maps in $\mathcal{A}$ which are cofibrations in $\mathcal{B}$ and whose cokernels lie in $\mathcal{A}$,
\item
the weak equivalences in $\mathcal{A}$ are the weak equivalences of $\mathcal{B}$ which lie in $\mathcal{A}$.
\end{enumerate}
%supported in degrees 0 and 1 by
%\indent
Define a Waldhausen category as follows:
\begin{definition} The category \textbf{R} is the smallest subcategory of $C_{b}(\textnormal{Proj} \calR)$ which:\\
(i) contains the complex \eqref{sigmachains} as defined above, for all $\sigma \in \Sigma$,\\
(ii) contains all acyclic complexes,\\
(iii) is closed under the formation of mapping cones and suspensions,\\
(iv) contains any direct summand of any of its objects.\\
\end{definition}
The following theorem
%version of \ref{NeemanRanicki6term}
is a combination of \cite[Corollary 4.9]{RanickiNeeman} and   \cite[Theorem 0.10]{Neeman2}.
\begin{theorem} \label{sixR}
\cite[p.789]{Neeman2} Let $\calR$ be a ring, and $\Sigma$ a collection of homomorphisms between finitely generated projective $\calR$-modules. There is an exact sequence
\begin{equation}
K_{1}(\textbf{R}) \to K_{1}(C_{b}(\textnormal{Proj} \calR)) \to K_{1}(\Sigma^{-1}\Rcal) \to K_{0}(\textbf{R}) \to K_{0}(C_{b}(\textnormal{Proj} \calR)) \to K_{0}(\Sigma^{-1}\Rcal)
\end{equation}
\end{theorem}
%We remark that
In Theorem \ref{sixR}, $\calR$ is general and there is no
requirement that $\Sigma$ consists of monomorphisms. The maps
$K_{i}(\textbf{R}) \to K_{i}(C_{b}(\textnormal{Proj} \calR))$ are
induced by the inclusion $\textbf{R} \to C_{b}(\textnormal{Proj}
\calR)$. Upon replacing $C_{b}(\textnormal{Proj} \calR)$ in
Theorem \ref{sixR} with $\calR$ using Gillet-Waldhausen, the maps $K_{i}(\calR) \to K_{i}(\Sigma^{-1}\calR)$ coincide with the maps $K_{i}(\calR) \to K_{i}(\Sigma^{-1}\calR)$ induced by the ring homomorphism $\calR \to \Sigma^{-1}\calR$ (see the discussion following Theorem 0.10 in \cite{Neeman2}).

Let $C_{b}(\mathcal{H}_{1}(\Rcal))$ denote the Waldhausen category of
bounded chain complexes of finitely presented $\Rcal$-modules of
projective dimension $\le 1$. Given $\Sigma$ a collection of
monomorphisms and $\mathcal{E} = \mathcal{E}(\Sigma)$ as in Definition
\ref{Edefinition},
we let $C_{b}(\mathcal{E})$ denote the Waldhausen category of bounded
chain complexes of objects of
$\mathcal{E}$. For both $C_{b}(\mathcal{H}_{1}(\calR))$ and $C_{b}(\mathcal{E})$, the cofibrations consist of the chain maps which are degree-wise split monomorphisms, and the weak equivalences are the quasi-isomorphisms. \\
\begin{lemma}\label{subcategorylemma}\cite[Theorem 2.7]{Neeman2}
There is a Waldhausen subcategory $\textbf{R}^{\prime} \subset C_{b}(\mathcal{H}_{1}(\Rcal))$ and inclusions $\textbf{R} \to \textbf{R}^{\prime}$, $C_{b}(\mathcal{E}) \to \textbf{R}^{\prime}$ that induce homotopy equivalences
\[ \xymatrix{
 & K(\textbf{R}) \ar[dr]^{\simeq} &  \\
  & & K(\textbf{R}^{\prime}) \\
  & K(C_{b}(\mathcal{E})) \ar[ur]_{\simeq} &  }  \\
\]
\end{lemma}
%\red{Below, I dropped your  \textbf{Remark: } and put the paragraph into the latex remark
%  environemnt. \\}
\begin{remark}
The subcategory $\textbf{R}^{\prime}$ of
Lemma
%3.9
\ref{subcategorylemma}
%\red{Scott: please check I guessed right
%  when replacing ``Lemma 3.9'' with ``Lemma \ref{subcategorylemma}''}
defined in \cite[Theorem 2.7]{Neeman2} is the full Waldhausen
subcategory
of $C_{b}(\mathcal{H}_{1}(\Rcal))$ consisting of all objects which
become isomorphic in $D(C_{b}(\mathcal{H}_{1}(\Rcal)))$ to objects in
the image of $D(\textbf{R})$, the derived category of $\calR$. Details
regarding $\textbf{R}^{\prime}$ are not important for the present
article, and may be found in the proof of Theorem 2.7 in
\cite{Neeman2}.
\end{remark}
 %(Recall the derived category of a category of bounded chain complexes consists of the same objects, but one identifies chain maps which are chain homotopic, and then localizes at the set of quasi-isomorphisms. In other words, given a category $\mathcal{A}$ whose objects are some class of bounded chain complexes, the objects of $D(\mathcal{A})$ are again this class of chain complexes, but the morphisms now are so called 'roofs': a morphism from $C \to D$ consists of $C \leftarrow C^{\prime} \to D$, where $C^{\prime} \to C$ is a quasi-isomorphism, and $C^{\prime} \to D$ is a any morphisms of chain complexes.))

%\red{Scott: please replace the hardwired 3.10,3.4,3.8 of the next
%  paragraph (probably some
%  or all are wrong now) with appropriate refs to labels}
One consequence of \ref{subcategorylemma} is that, by the Gillet-Waldhausen theorem, we have $K(\textbf{R}) \simeq K(\mathcal{E})$, which gives one of the identifications made when passing between \ref{NeemanRanicki6term} and \ref{sixR}. \\

We now specialize to the case of interest, in order to prove the main
result of the section.
%\\
\emph{For the remainder of the section, we let $\Sigma=\Omega_{+}$
  denote the collection of Fredholm homomorphisms of finitely
  generated projective
$\calR[t]$-modules.}
% over $\calR[t]$.}

\begin{proposition} \label{vanishingmaps}
Consider a polynomial ring $\calR[t]$, with $\Omega_{+}$ the collection of Fredholm homomorphisms, and \textbf{R} as defined in Definition 3.8. Then the maps
$$K_{n}(i) \colon K_{n}(\textbf{R}) \to K_{n}(C_{b}(\textnormal{Proj} \calR[t]))$$
are zero, for all $n$, where $K_{n}(i)$ is the map induced by the inclusion \textbf{R} $\to C_{b}(\textnormal{Proj} \calR[t])$.  \\
\end{proposition}

Since the maps $K_{n}(\textbf{R}) \to K_{n}(C_{b}(\textnormal{Proj} \calR[t]))$ in Theorem \ref{sixR} are induced by the inclusion $\textbf{R} \to C_{b}(\textnormal{Proj} \calR[t])$, Theorem \ref{fredholmk1injective} will follow from Proposition \ref{vanishingmaps}.\\

\emph{Proof of Proposition \ref{vanishingmaps}:} Consider the diagram of inclusions
\[ \xymatrix{
 &  & C_{b}(\mathcal{H}_{1}(\Rcal[t])) \\
  & \textbf{R} \ar[ur] \ar[r] & C_{b}(\textnormal{Proj} \Rcal[t]) \ar[u] }  \\
\]
%\red{The next bit needs more detail; and if you use the resolution
%  theorem, give a statement. Also, watch the grammar--do you really
%  want a functor using a theorem?}
By the Resolution Theorem (see \cite[V.3.1]{WeibelBook}) we have $K(\textnormal{Proj}\calR[t]) \simeq K(\mathcal{H}_{1}(\calR[t]))$, so combined with Gillet-Waldhausen, the vertical functor on the right induces a homotopy equivalence
$$K(C_{b}(\textnormal{Proj} \Rcal[t])) \stackrel{\simeq}\to K(C_{b}(\mathcal{H}_{1}(\Rcal[t])))$$
and therefore isomorphisms $K_{n}(C_{b}(\textnormal{Proj} \calR[t]))
\to K_{n}(C_{b}(\mathcal{H}_{1}(\calR[t])))$ for all $n$. Furthermore,
Lemma \eqref{subcategorylemma}
shows that the images of the homomorphisms
\begin{align*}
K_{n}(\textbf{R}) &\to K_{n}(C_{b}(\mathcal{H}_{1}(\Rcal[t])))  \\
K_{n}(C_{b}(\mathcal{E})) &\to K_{n}(C_{b}(\mathcal{H}_{1}(\Rcal[t])))
\end{align*}
coincide. We claim that the map $K_{n}(C_{b}(\mathcal{E})) \to K_{n}(C_{b}(\mathcal{H}_{1}(\Rcal[t])))$ is zero for all $n$. This will prove that $K_{1}(\calR[t]) \to K_{1}(\Omega_{+}^{-1}\calR[t])$ is injective, in light of Theorem 3.8. We have a diagram

\[ \xymatrix{
 & C_{b}(\mathcal{E}) \ar[r] & C_{b}(\mathcal{H}_{1}(\Rcal[t])) \\
  & \mathcal{E} \ar[u] \ar[r] & \mathcal{H}_{1}(\Rcal[t]) \ar[u] }  \\
\]
in which by Gillet-Waldhausen the vertical arrows
induce homotopy equivalences in $K$,

\[ \xymatrix{
 & K(C_{b}(\mathcal{E})) \ar[r] & K(C_{b}(\mathcal{H}_{1}(\Rcal[t]))) \\
  & K(\mathcal{E}) \ar[u]^{\simeq} \ar[r] & K(\mathcal{H}_{1}(\Rcal[t])) \ar[u]_{\simeq} }  \\
\]
Thus it suffices to show that the maps $K_{n}(\mathcal{E}) \to
K_{n}(\mathcal{H}_{1}(\Rcal[t]))$, induced by the inclusion functor
$j \colon \mathcal{E} \to \mathcal{H}_{1}(\Rcal[t]))$, are zero for all $n$.

Let $X$ be the full subcategory of $\mathcal{H}_{1}(\calR[t])$
whose objects are
the modules $M$ in (i.e. the
objects $M$ of) $\mathcal{H}_{1}(\calR[t])$
such that $\eta(M) := M \otimes_{\calR} \calR[t]$ is in
$\mathcal{H}_{1}(\calR[t])$. (For example, $\calR[t]$ is in
$\mathcal{H}_{1}(\calR[t])$ but is not in $X$, because
  $\calR[t]$
 is not
  finitely generated as an $\calR$-module.) We claim that $\mathcal{E}$ is contained in $X$. Consider each of the following:
\begin{enumerate}
\item
If $\sigma \in \Omega_{+}$, then $\coker(\sigma)$
is finitely generated projective
as an $\calR$-module, since $\Omega_{+}$ consists of Fredholm morphisms. It follows that $\coker(\sigma) \otimes_{\calR} \calR[t]$ lies in $\textnormal{Proj} \calR[t] \subset \mathcal{H}_{1}(\calR[t])$, so $X$ contains the cokernels of all morphisms $\sigma \in \Omega_{+}$.
\item
Now suppose
$$0 \to M_{1} \to M_{2} \to M_{3} \to 0$$
is exact in $\mathcal{H}_{1}(\calR[t])$.

Tensoring this sequence with $\calR[t]$ gives
\begin{equation}\label{tensorsequence}
0 \to M_{1} \otimes_{\calR}\calR[t] \to M_{2} \otimes_{\calR}\calR[t]
\to M_{3} \otimes_{\calR}\calR[t] \to 0
%\tag{$***$}
\end{equation}
which is
an exact sequence of $\calR[t]$-modules,
 since $\calR[t]$ is free over $\calR$.
 We claim that if two of $M_{1}, M_{2}, M_{3}$ lie in
 $X$, then so does the third.
\begin{enumerate}
\item
Suppose $M_{2}$ and $M_{3}$ lie in $X$.
Then $\eta(M_{2})$ and $\eta(M_{3})$ lie in
$\mathcal{H}_{1}(\calR[t])$. Since $\mathcal{H}_{1}(\calR[t])$ is
closed under kernels of surjections (see \cite[II.7.7.1]{WeibelBook}),
the exactness of \eqref{tensorsequence} shows that  $M_{1}$ lies in $X$.
\item Suppose $M_{1}$ and $ M_{3}$ lie in $X$. Then
the exactness of \eqref{tensorsequence} along with
%\annotation{\red{Is it feasible to give a short indication of what
%that fact 2.2.8 is?}}
the fact that $\mathcal{H}_{1}(\calR[t])$ is closed under extensions (see \cite[2.2.8]{WeibelHomological}) implies $M_{2}$ lies in $X$ as
  well.
\item
Suppose $M_{1}$ and $M_{2}$ lie in $X$. Then
the exactness of
\eqref{tensorsequence} shows that  $\eta(M_{3})$
is finitely presented, being a quotient of two finitely presented
modules. But it is clear that $\eta(M_{3})$ is also of homological
dimension $\le 1$, so $M_{3}$ is in $X$ as well.
\end{enumerate}
\item
$X$ contains all direct summands of its objects, since $\mathcal{H}_{1}(\calR[t])$ is closed under direct summands.
\end{enumerate}
Since $\mathcal{E}$ is the minimal subcategory of
$\mathcal{H}_{1}(\calR[t])$ satisfying
the corresponding properties (1,2,3) in
Definition
\ref{Edefinition}, we have $\mathcal{E} \subset X$, as desired.\\

%(note to self: $M$ of projective dimension le 1 over $\mathcal{S}$ implies $M \otimes_{\mathcal{S}}\Rcal$ is also of pd le 1, since given a projective resolution $P \to Q \to M$, since $\mathcal{S}[t]$ is a free $\mathcal{S}$-module, tensoring that resolution with $\mathcal{S}[t]$ gives a projective resolution of $M \otimes_{\mathcal{S}} \mathcal{S}[t]$.)
\indent The remainder of the proof closely follows that of the commutative case given earlier. Given $M \in \mathcal{E}$, let $f_{M}$ denote the endomorphism of $M$ induced by the
$\calR[t]$-module structure of $M$ (so $f_{M}(x) = t \cdot x$). From the discussion above we have the exact functor $\eta \colon \mathcal{E}(\Omega_{+}) \to \mathcal{H}_{1}(\calR[t])$, and we denote by $\mathcal{F}$ the natural transformation $\mathcal{F} \colon \eta  \mapsto  \eta$ defined by $\mathcal{F}(M) \colon \eta(M) \stackrel{t-f_{M}}\to \eta(M)$. Recall $j \colon \mathcal{E} \to \mathcal{H}_{1}(\calR[t])$ denotes the inclusion functor. Define the natural transformation $\mathcal{G} \colon \eta \mapsto j$ by $\mathcal{G} \colon \eta(M) \stackrel{\pi}\to M$, where $\pi(p(t) \otimes x) = p(f_{M})(x)$. Then $\eta \stackrel{\mathcal{F}}\rightarrowtail \eta \stackrel{\mathcal{G}}\twoheadrightarrow j$ is an exact sequence of functors, since for any $M \in \mathcal{E}$, the sequence
$$0 \to M \otimes_{\calR} \calR[t] \stackrel{t-f_{m}}\rightarrowtail M \otimes_{\calR} \calR[t] \stackrel{\pi}\twoheadrightarrow M \to 0$$
is exact (see \cite[p. 630]{BassBook}). Letting $K_{n}(\eta),
K_{n}(j)$ denote the corresponding maps on K-theory, the Additivity
Theorem (V.1.2 in \cite{WeibelBook}) now implies that, for all $n$,
$K_{n}(\eta) = K_{n}(\eta) + K_{n}(j)$. Thus $K_{n}(j)$ is the zero
map, for all $n$. This finishes the proof of Theorem \ref{fredholmk1injective}.
\section{The elementary stabilizer} \label{secstabilizer}

Recall our notational conventions (\ref{abuse}, \ref{remarkableabuse}).
In particular, $\mathfrak M(\calR) $ is the set of $\N\times \N$ matrices
over the ring $\calR $ equal to the identity except in
finitely many entries, with
$\el(\calR) \subset \GL (\calR) \subset \mathfrak M(\calR) $.
  Given $\calR$ and
$M$ in $\mathfrak M(\calR)$,  the elementary stabilizer of
$M$ is defined to be
\begin{equation}\label{elstabdefn}
\seor (M)= \{ U\in \GL (\calR): U\eo (M) \subset \eo (M)\} \  .
\end{equation}
Because $\EL (\calR )$ is a subgroup of $\seor (M)$,
$\{ [U]\in K_1(\calR)\colon U \in \seor (M)\}$ is
a subgroup of $K_1(\calR)$, which by abuse of notation
we also denote by $\seor (M)$. We give a shorter notation
for the elementary stabilizer which is our main interest.
Given an $n\times n$
matrix $A$ over $\calR$,
let $I-tA$ denote
$(I_n-tA)_{\stabone} = I -tA_{{\stabzero}}$
(as in \ref{remarkableabuse})  and define
\begin{equation} \label{eardefn}
\ear := \seort (I-tA) \ .
\end{equation}
    If $U\in \ear$, then there are $E,F$ from
  $\EL(\calR)$ such that $U(I-tA) = E(I-tA)F$.
  Evaluating at $t=0$, we see
    that $\ear$, considered as a subgroup of $K_1(\calR[t])$,
    satisfies
    \begin{equation} \label{elstabisinnk1}
      \ear \subset NK_1(\calR) \ .
    \end{equation}

\begin{proposition}\label{nopun}
Suppose $\calR$ is a ring
and $ A\in \mathfrak M (\calR)$.
%Define
%$\ear :=\{[U]\in K_1(\calR ): U\in \seor (I-A)\}$.
%Then $\ear$
%is a subgroup of $K_1(\calR)$,  and
Then there is a bijection
\begin{align*}
  K_1(\calR)/\seor (A) &\to \{\eo (B): B \in \go (A)\} \\
  [U]\ & \mapsto \ U\eo (A) \ .
\end{align*}
If $B\in \go (A)$, then $\seor(B) = \seor(A)$. \\
Also,
\begin{equation}\label{ww}
\bigcup_{A\in \mathcal M(\calR )}\ear  = \bigcup_{C\in \GL
  (\calR[t])\colon\, \\ C_0\in \GL (\calR)} \seort (C) \ .
\end{equation}
\end{proposition}

\begin{proof}
For $B\in
    \orbglr (A)$, let
%  we use abbreviations  $\mathcal E_B =\seor (I-B)$ and
    $\mathcal O_B=\eo (B)$.
    Then
      $U\mathcal O_B = \mathcal O_{UB}= \mathcal O_{BU}=\mathcal O_BU $,
    for all $U$ in $\GL (\calR )$ and  $B\in \orbglr (A)$.
    Therefore
    the rule $U\colon \mathcal O \mapsto U\mathcal O $ gives a well defined
    action of
    $\GL (\calR)$ on
  $\{ \mathcal O_B\colon B\in
    \orbglr (A)\} $.
    The isotropy group of an element $\mathcal O_B$ under this
    action is
    $    \seor (B)$, which contains
    $\el(\calR)$. Therefore given
    $B\in
    \orbglr (A)$ we have well defined bijections
    \begin{alignat*}{5}
            K_1(\calR)  /\seor (B) \ \
      & \to
      & \ \       \GL(\calR)/\seor (B) \ \
      & \to \ \
      & \{ \mathcal O_C\colon C \in \orbglr (A)\}
      \\
       [U] \qquad
      & \mapsto
      & [U] \qquad \qquad
      & \mapsto
      & U\mathcal O_C \qquad \ .\qquad \qquad \qquad
          \end{alignat*}
    For $B \in \orbglr (A)$, the isotropy groups
    $\seor (A)$ and $\seor (B)$ are conjugate in
    $\GL (\calR)$, and therefore equal, as $\seor (A)$ contains
    $\el(\calR)$, the
    commutator subgroup of $\GL (\calR)$.

To prove  ``Also'',  it now suffices, given a
 square matrix $C$ over $\calR[t]$ with
$C(0)$ in $\GL(\calR)$, to
note that the $\GL (\calR [t])$ orbit of $C$ contains
a matrix of the form $I-tA$ with $A$ over $\calR$.
This holds by application of Higman's trick to
$C_0^{-1} C$.
\end{proof}

The next result, a key fact for us, follows directly from Theorem \ref{fredholmk1injective}. %in the case of $\Omega_{M}(\calR[t])$%
For its statement, recall that by
our notational convention, the elementary stabilizer of a finite
matrix $I-A$ means the elementary stabilizer of
$(I-A)_{\stabone}$.
%Recall that the map
%$j \colon \calR[t] \to \Omega_{+}^{-1}\calR[t]$ from
%Corollary} \ref{diagramlemma} denotes the map $j = i_{2}
% \circ i_{1}$, where $i_{1} \colon \calR[t] \to \calR[t,t^{-1}]$ with $i_{1}$ sending $t \to t^{-1}$, and $i_{2} \colon \calR[t,t^{-1}] \to \Omega_{+}^{-1}\calR[t]$ the standard map for the localization.
Recall that the map $i:\calR[t] \to \Omega_{+}^{-1}\calR[t]$ denotes the standard map coming from the definition of the localization.

%\red{For a finite square matrix $I-tA$, we define

\begin{theorem} \label{trivialkseo}
  Let $\calR[t]$ be a polynomial ring, with coefficient ring $\calR$.
  If $A$ is a square matrix over $\calR [t]$
  such that $I-A \in \Omega_{+}(\calR[t])$, then
  $\seort (I-A) \textnormal{  is trivial in } K_1(\calR[t])$.
% such that
%$A \in \Omega_{+}(\calR[t])$ or
% is Fredholm over $\calR$ or
%$A\in \Omega_{M}(\calR[t])$.
%Then the elementary stabilizer $\ea=\eart$ is trivial in $K_1(\calR[t])$.
\end{theorem}
\begin{proof}

\iffalse
EDIT OUT PORTION BEGINS
We first consider the case where $A$ is over $t\calR[t]$. Let $U$ be a matrix in $\seo (I-A)$. Suppose $A = \sum_{i=1}^{n}t^{i}A_{i}$, with $A_{n} \ne 0$.
We have  matrices $E,F$ in $\text{El}(\calR[t])$ such that
$U(I-A)=E(I-A)F$.
%Suppose  $A\in \Omega$ where $\Omega$ is
%$\Omega_{+}(\calR[t])$ or
%$\Omega_{M}(\calR[t])$.
Then the embedding $j = i_{2} \circ i_{1} \colon R[t]\to \Omega^{-1} R[t]$ of
Corollary \ref{diagramlemma} satisfies $$t^{n} \cdot j(I-A) = t^{n}(I - \sum_{i=1}^{n}t^{-i}A_{i}) = t^{n}I-\sum_{i=1}^{n}t^{n-i}A_{i}$$
which is monic, and hence invertible in
$\Omega_{+}^{-1}\calR[t]$. Since $t$ is invertible in
$\Omega^{-1}\calR[t]$, $j$ thus sends $I-A $ to an invertible matrix. It follows that $[U]$
is in the kernel of the map
$j \colon K_1(R[t])\to K_1(\Omega_{+}^{-1}R[t] )$,
which by Corollary
\ref{diagramlemma} is injective.
Therefore  $[U]$ is trivial in
$K_1(R[t])$.\\
EDIT OUT SECTION ENDS
\fi

If $I-A \in \Omega_{+}(\calR[t])$ and
$U \in \seort(I-A)$, then $[U]$ is in the kernel of the map $i_{*}\colon K_{1}(\calR[t]) \to K_{1}(\Omega_{+}^{-1}\calR[t])$ of Theorem \ref{fredholmk1injective}. By Theorem \ref{fredholmk1injective}, this implies $[U]=0$.
\end{proof}
%\red{Do we want a more general class in place of $\widetilde{\Omega}_+ $?}

\begin{theorem} \label{classesinOmega}
  Let $A$ be a square matrix over $t\calR[t]$ such that $A = \sum_{i=1}^{d}t^{i}A_{i}$. Suppose $A$ satisfies any of the following:
\begin{enumerate}
  \item
$A_{d}$ is nilpotent and $A_{i} = 0$ for $1 \le i < d$.
\item
$A_{d}$ is invertible over $\calR$.
\item
$A_{d}$ is idempotent and $A_{i}=0$ for $1 \le i < d$.
\end{enumerate}
Then $\seort (I-A)$ is trivial in $K_{1}(\calR[t])$.
\end{theorem}
In the statement of Theorem \ref{classesinOmega},
if $d=1$ then $\seort (I-A)=E(A_1,\calR)$.

\begin {proof} [Proof of Theorem \ref{classesinOmega}]
%\begin{proof}
  The claim for case (1) is clear, because $I-A \in \GL (\calR[t])$.
    For the remaining cases,
    by Theorem \ref{trivialkseo} it suffices to show that the matrix $I-A$ is invertible over $\Omega_{+}(\calR[t])$.  For case
  $(2)$, the matrix $(I-A)A_{d}^{-1}$ is monic, and hence invertible over $\Omega_{+}^{-1}\calR[t]$.\\
\indent For case $(3)$, we first note that if $P$ is an $n \times n$ idempotent matrix over $\calR$, then $cok(I-tP)$ is a finitely generated projective $\calR$-module. Let $J$ denote $P(\calR^{n})$,
  the image of the $\calR$-module endomorphism $\calR^{n} \stackrel{P}\to \calR^{n}$ given by multiplication by $P$.
  The finitely generated $\calR$-module $J$
  is  projective, since $P$ is idempotent.
  Letting $x_0, \dots , x_{d-1}$ denote elements of $\calR^n$,
  we have an isomorphism of $\calR$-modules
  $cok(I-tP) \to J^d$ given by
  $[\sum_{i=0}^{d-1} t^{i+d}Px_i] \mapsto (Px_0, \dots , Px_{d-1})$.
It follows that the matrix $I-t^dP$ belongs to $\Omega_{+}^{mat}(\calR[t])$, i.e. is Fredholm. Thus, for the map $i\colon \calR[t] \to \Omega_{+}^{-1}\calR[t]$ given by the localization, the matrix $i(I-t^dP)$ is invertible over $\Omega_{+}^{-1}\calR[t]$.
\end{proof}
%Given $A$ square over $\calR$, we would like to know the
%  structure of $\ear$. We cannot answer the first question about this:
%  \begin{question}  \label{question:trivialear}
%    Must $\ear$ be trivial, for every square matrix $A$ over $\calR$?
%    \end{question}

%\red{BEGIN CUT}
%
%Corollary.
%Suppose $A = \sum_{i=1}^{d}t^{i}A_{i}$ is a square matrix over $t\calR[t]$, and there exists a nilpotent matrix $N$ over $\calR$ such that $A_{d}N = 0$ and $A_{d-1}N-A_{d}$ is invertible over $\calR$. Then $\seort (I-A)$ is trivial in $K_{1}(\calR[t])$.
%\begin{proof}
%This follows from Theorem \ref{classesinOmega}, along with the observation that, given $N$ as described, the top degree non-zero term of $(I-A)(I-tN)$ is $A_{d-1}N-A_{d}$.
%\end{proof}
%\red{END CUT}

In the case $\calR$ is commutative, localization techniques allow us
to make further statements regarding $\seort(I-A)$ for certain
matrices $A$ which satisfy none of the sufficent conditions
(1)-(3) of Theorem \ref{classesinOmega}.Our main tool for this will be the following result of
Vorst (see \cite[1.7, Remark 1.12]{Vorst1979}). For
a an element $r \in \calR$ we let $\calR_{r}$ denote the localization of $\calR$ at the multiplicative subset $\{r^{i}\}$. Recall that a collection of elements $\{f_{1},\ldots,f_{k}\} \subset \calR$ is called a unimodular row if the ideal $(f_{1},\ldots,f_{k})$ generated by the collection is $\calR$ itself.
\begin{theorem}\label{VorstTheorem}\cite[Corollary 1.7]{Vorst1979}
Let $\calR$ be a commutative ring, and let $f_{1},\ldots,f_{k} \in \calR$ be a unimodular row over $\calR$. Then the map $NK_{1}(\calR) \to \prod_{i=1}^{r}NK_{1}(\calR_{f_{i}})$ is injective.
\end{theorem}

\begin{proposition}\label{localizationprop}
  Let $\calR$ be a commutative ring, and let $A$ be a square matrix over $\calR$ such that
$0 \ne \det(A)$
  and $\det(A)$ is not a zero-divisor. Suppose
there exists $j$ unimodular rows
$$\{\det(A),f_{1,1},\ldots,f_{k_{1},1}\}, \ldots, \{\det(A),f_{1,j},\ldots,f_{k_{j},j}\}$$
 each containing $\det(A)$ such that
$$\bigcup_{i=1}^{j} \bigcap_{n=1}^{k_{i}}\textnormal{ker}(NK_{1}(\calR) \to NK_{1}(\calR_{f_{n,i}})) = NK_{1}(\calR)$$
Then $\seort(I-A) = 0$.
\end{proposition}
\begin{proof}
Suppose $G \in \seort(I-A)$, and let $[G]$ denote its class in $NK_{1}(\calR)$. The assumptions give an $i$ such that $[G] \in \bigcap_{n=1}^{k_{i}}\textnormal{ker}(NK_{1}(\calR) \to NK_{1}(\calR_{f_{n,i}}))$. Since $A$ is invertible over $\calR_{\det(A)}$, Theorem \ref{classesinOmega} implies $[G] \in \textnormal{ker}(NK_{1}(\calR) \to NK_{1}(\calR_{\det(A)}))$ as well, and hence by Theorem \ref{VorstTheorem} we must have $[G] = 0$.
\end{proof}
Proposition \ref{localizationprop} can be used to show
that for certain matrices $A$ over $\ZG$,
the elementary stabilizer $\ear:=\seort(I-tA)$ must vanish, as
follows.
\begin{corollary}\label{localizationexample}
  Let $G$ be a finite abelian group of order $|G|$, and let $\mathbb{Z}G$ denote the integral group ring. Let $A$ be a square matrix over $\mathbb{Z}G$ such
  that
  $0\neq \det(A) = a \in \mathbb{Z}$ and $(a,|G|) = 1$ (so $a$ and the order of $G$ are relatively prime). Then $\ = 0$.
\end{corollary}
\begin{proof}
The collection $\{a,|G|\}$ forms a unimodular row over $\mathbb{Z}G$. However, by \cite[6.5, pg. 490]{Weibel1981}, $\textnormal{ker}(NK_{1}(\mathbb{Z}G) \to NK_{1}((\mathbb{Z}G)_{a})) = NK_{1}(\mathbb{Z}G)$, so Proposition \ref{localizationprop} implies the claim.
\end{proof}
\begin{remark} \label{limitstolocal}
    The technique of using localization to
prove  $\seort(I-A)$ is trivial, as
in the proof of Corollary \ref{localizationexample}, has its limits.
For example, if
$G$ is a finite group $G$, then
the map $NK_{1}(\mathbb{Z}G) \to NK_{1}((\mathbb{Z}G)_{|G|})$ is the zero map.
\end{remark}

\begin{remark} \label{festabBS05remark}
For $G$ a finite group  and
$A$ a matrix over $\Z G$,  the group
$K_1(\Z G)/ \seozg (I-A)$
%(defined in Defn. \ref{stabilizerdefn4} )
appeared  in \cite{BS05}
as the primary invariant
for the classification up to equivariant flow equivalence
of certain symbolic
dynamical systems: irreducible shifts of finite type
with a free continuous shift-commuting $G$-action.
%(which can be presented by a matrix $A$ over $\Z_+G$).
%\st{In contrast to } \eqref{trivalelrt}\st{,}
\end{remark}

\begin{center}
{\bf Limits to generalizations}
\end{center}

Theorem \ref{classesinOmega} applies to a rather special class
of matrices and its proof appeals to the
sophisticated algebraic K-theory of Neeman and Ranicki
\cite{Neeman2, RanickiNeeman}.
It is natural to ask if there is an easier proof.
It is also natural to hope the conclusion  of Theorem \ref{classesinOmega}
 might hold for
a more general class of matrices.
We'll note next that some candidate improvements cannot
work.

\begin{remark}\label{failedproofs}
With an eye to an easier proof
of Theorem \ref{classesinOmega}, one might note
for $A$ over $t\calR [t]$ that
$(I-A)$ also inverts over the familiar ring of formal power series
$\calR[[t]]$, and ask if  $\calR[[t]]$ could play the role of
$\Omega^{-1}_{+}\calR[t]$ in
Theorem \ref{fredholmk1injective}. We thank Wolfgang Steimle for showing us
this fails: the natural map
$i^*\colon  K_1(\calR[t])\to K_1(\calR[[t]])$
induced by the inclusion
$i\colon \calR[t] \to \calR[[t]]$ need not be injective.
For example, if $\calR$ is commutative,
then there is a straightforward decomposition of $K_{1}(\calR[[t]])$ given by
$$0 \to K_{1}(\calR) \to K_{1}(\calR[[t]]) \stackrel{d}\to \hat{W}(\calR) \to 0$$
where $\hat{W}(\calR) = \{1 + \sum_{i=1}^{\infty}a_{i}t^{i}\} \in
\calR[[t]]$ is the group of Witt vectors. The map $d$ is given by
$d(M) =\det(M_{0}^{-1}M)$, where $M = \sum_{i=0}^{\infty}M_{i}t^{i}$
(as in e.g. \cite[14.6]{Ranickibook}). Thus,
if $\calR$ is a commutative ring (for example, an integral domain) such that $\det(I-tN) = 1$ for all
nilpotent matrices $N$, then
the kernel of the map
$K_{1}(\calR[t])\to K_{1}(\calR[[t]])$ induced by the inclusion
$\calR[t] \to \calR[[t]]$ will always
contain $NK_{1}(\calR)$.
Indeed, $d(I-tN) =\det(I-tN) = 1$, so $NK_{1}(\calR)$ maps into
the kernel of $d$, which is generated by the image of
$K_{1}(\calR)$;
but the only class of the form $[I-tN]$ which lies in the image of
$K_{1}(\calR)$ is the class $[1]$. Since there are integral domains
$\calR$ with $NK_{1}(\calR) \ne 0$
(e.g.  \cite[Example 3.5]{BoSc3})
%(as in {e.g. \cite{BoSc3})}
%\mbred{(
%see Example 3.5 in [Cite Spec paper])},
the map $NK_1(\calR[t])\to K_1(\calR[[t]])$ need not be injective.

Similarly,
one could hope to
 prove in place of Theorem \ref{fredholmk1injective} that the
 map $i_{*} \colon K_{1}(\calR[t]) \to K_{1}(S_{RM}^{-1}\calR[t])$ is
injective, where
$\Sigma_{RM}$ is the set of reverse monic matrices (those of the
form $A = I + \sum_{i=1}^{n}A_{i}t^{i}$). But this map need not
be injective.
In the case $\calR$ is
commutative,
 localizing at
$\Sigma_{RM}$ is equivalent to localizing at $S_{RMP} = \{p(t) = 1 +
\sum_{i=1}^{n}a_{i}t^{i}\}$, the set of reverse monic polynomials.
There is an exact sequence
\[
0 \to K_{1}(\calR) \to K_{1}(S_{RMP}^{-1}\calR[t]) \stackrel{d}\to
1+tS_{RMP}^{-1}\calR[t]\to 0
\]
which can be found by examining \cite[Corollary 3]{Grayson1977}, or \cite[III.2.4(2)]{WeibelBook}.
As in the previous paragraph, the map
$i_{*} \colon NK_{1}(\calR) \to K_{1}(S_{RMP}^{-1}\calR[t])$
will fail to be injective for an integral domain $\calR$
with $NK_{1}(\calR)$  nontrivial.
\\
\end{remark}

With regard to generalizing the result,
Corollary \ref{toobad} of
Proposition \ref{nontrivialkseoexamples} below
shows Theorem \ref{classesinOmega} already
fails badly for the more general class
of matrices
$I-A$ which are injective (in the statement
of Corollary \ref{toobad}, $\calR$ could be a polynomial
ring).
The rest of this section is devoted to establishing
that corollary.
We thank David Handelman for showing us the
 embedding  argument which produces
 the nonderogatory matrix $V=UE$
in the reduction step of Prop. \ref{nontrivialkseoexamples} below.

\begin{proposition} \label{nontrivialkseoexamples}
Suppose $\calR$ is an  integral domain of characteristic zero which does not
embed into
$\mathbb Z[i]$ or  $\mathbb Z[e^{i2\pi /3}]$,
and $U$ is in
$\SL (n,\calR)$. Then there is an $n\times n$
 matrix $A$ over $\calR$
such that $I-A$ is
injective and
$U$
is in the elementary stabilizer $\seo_{\calR}(I-A)$.
%$\ear$.
\end{proposition}
%\annotation{I reworked Proposition \ref{nontrivialkseoexamples} and
%  its proof, without coloring.}
\begin{proof} Case I: For this case, we assume
there is a matrix $B$ over the field of fractions
 $\mathbb F$ of $\calR$ such that $B^{-1}UB=C$,
with $C$ a companion matrix.
Without loss of generality,  we then assume $B$ has all entries in $\calR$.
Because $C$ must be the companion matrix of the characteristic polynomial of $U$, the entries of $C$
must lie in $\calR$. From the companion matrix form and $\det C =1$,
we have $C\in \EL (n,\calR)$.
Now $UB=BC$; defining $A=I-B$, we have
that $U$ is in $\seo (I-A)$.
Clearly $I-A$ is injective.

For the reduction to Case I,
it suffices to show that there is a matrix
$E\in \EL (n,\calR)$ such that the matrix $V=UE$ has no
repeated eigenvalue
%nonderogatory (i.e. its characteristic polynomial equals
%its minimal polynomial). Then $V$ is
(and therefore is similar over $\mathbb F$ to its companion matrix).
%the Case I argument applies to $V$ and $[V]=[U]$.
After passing if needed to a subring containing the entries of $U$
and still satisfying the nonembeddability hypothesis,
we may assume $\calR$ is finitely generated.
Then $\mathbb F$
 is isomorphic to an algebraic extension
of a subfield of $\mathbb R$
(generated by $\mathbb Q$ and a set of
algebraically independent elements). Thus
after embedding $\mathbb F$ into
$\mathbb R$
or $\mathbb C$,  we have the
closure $\overline{\mathbb F}$ equal to
$\mathbb R$ or $\mathbb C$.
In either case, except under
the very  special conditions which are
excluded in the hypotheses (and are
not of interest to us now),
the ring $\calR$ will likewise be dense in
$\overline{\mathbb F}$, and consequently
$\text{El}(n,\calR)$ will be dense in
$\text{El}(n,\overline{\mathbb F})=\text{Sl}(n,\overline{\mathbb F})
$.
%\mbred{Scott, I put a proof of this for you just before the
%References.}
Let $W$ be a matrix in $\text{SL}(n,\mathbb Z)$
without repeated eigenvalues.
The matrices over $\mathbb F$ without repeated eigenvalues
 form a dense open set.
Consequently the matrix $U^{-1}W$ in
$\text{SL}(n, \mathbb F)$
can be perturbed to a matrix $E$
in $\text{El}(n,\calR)$
such that $UE$ has no repeated eigenvalues.
\end{proof}

In the next statement, $E_{A,\calR[t]}$ denotes
$\{[U]\in K_1(\calR[t] ): U\in \seo (I-A)\}$.

\begin{corollary} \label{toobad}
Suppose $\calR$ is a characteristic zero integral domain
  which is not generated by three elements as an additive group,
  and
$NK_1(\calR)$ nontrivial. (Such domains
exist.) In the class of injective matrices $(I-A)$
over $\calR[t]$,
the elementary stabilizer $E_{A,\calR[t]}$ is not independent of
$A$. If $H$ is a finitely generated subgroup
of $NK_1(\calR )$, then there exists an injective
$(I-A)$ such that
$E_{A,\calR[t]}$
   contains $H$.
\end{corollary}
\begin{proof}
  For $I-A$ invertible over $\calR[t]$,
$E_{A,\calR[t]}$ is trivial in $NK_1(\calR)$.
Now choose $U_k$ in $\gl (\calR[t] )$ for
$1\leq k \leq K$ with $[U_{k}] \in NK_{1}(\calR)$.
Because $\calR$ is an integral domain, the
$U_{k}$ lie in $SL(\calR[t])$. Proposition \ref{nontrivialkseoexamples} then gives
finite matrices $I-A_k$ over $\calR[t]$ with
$I-A_k$ injective such that for $A=A_k$,
$E_{A,\calR[t]}$ contains
$U_k$. If $A = \oplus_{k=1}^K A_k $, then
$E_{A,\calR[t]}$ contains all of the $U_k$.
For an explicit example of an integral domain $\calR$
which embeds into $\R$ and has $NK_{1}(\calR) \ne
0$ see \cite[Example 3.5]{BoSc3}.

%\red{Full previous text for proof:}
%For $A=0$, $\ear =NK_1(\calR)$
%\red{is first sentence to show it is possible to have $\ear = NK_{1}(\calR)?$: I think it should be for $A=I$?}.
%Given \red{$U_k$ in $\gl (\calR[t] )$ with $[U_{k}] \in NK_{1}(\calR)$},
%$1\leq k \leq K$, \red{we have $U_{k} \in SL(\calR[t])$, by Higman's trick and the fact that $\calR$ is an integral domain.}
%Proposition \ref{nontrivialkseoexamples} then gives
%finite matrices $I-A_k$ \red{over $\calR[t]$} with
%$I-A_k$ injective such that for $A=A_k$, $\ear $ contains
%$U_k$. If $A = \oplus_{k=1}^K A_k $, then
%$\ear $ contains all of the $U_k$.
%For an example of an integral domain $\calR$ with $NK_{1}(\calR) \ne
%0$ see e.g.  \cite[Example 3.5]{BoSc3}).
\end{proof}

\section{The union of the elementary stabilizers}
The purpose of this section is to prove Theorem
\ref{nonvanishingstabilizers}. Throughout, we will use the following
notational conventions
\begin{itemize}
\item
$\mathcal{R}$  denotes a commutative ring
(except in Conjecture \ref{unionofesconjecture}).
\item
$NSK_{1}(\mathcal{R})$ denotes the subgroup of $NK_{1}(\mathcal{R})$ defined by
$$\{[G] \in NK_{1}(\mathcal{R}) \hspace{.02in} | \hspace{.02in} \det(G)=1\}.$$
\item
$S_{RMP}$ denotes the collection of polynomials in $\mathcal{R}[t]$
whose constant term is 1, i.e. $S_{RMP} = \{1 + a_{1}t + \cdots
+a_{n}t^{n}\}$ (the ``reverse monic''  polynomials).
\item Given a matrix $C$ in $\mathcal M (\mathcal{R})[t])$, $C_0,
  \dots  C_n$ are the matrices over $\mathcal R $ such that
$C= \sum_{i=0}^nC_it^i$.
(For $\mathcal M(\mathcal R[t])$, recall the notation \eqref{abuse}.)
\item
$\ear = \seort (I-tA)$
% denotes the elementary stabilizer of $A$
(recall  Definitions \ref{elstabdefn}, \ref{eardefn}).
\end{itemize}
\begin{theorem}\label{nonvanishingstabilizers}
For a commutative ring $\mathcal{R}$, we have
\[
 \bigcup_{A\in \mathcal M(\mathcal R )} \ear = NSK_{1}(\mathcal{R})
\]
%where the union is taken over all finite square matrices over $\mathcal{R}$. \\
\indent If $\mathcal{R}$ is reduced, then $NSK_{1}(\mathcal{R})=NK_{1}(\mathcal{R})$.
\end{theorem}

\begin{remark}\label{meaculpa}
Version 1 of our arXiv post \cite{BoSc1} claimed
that
for every $\calR$ and every $A$,  $\ear$ is trivial.
Theorem \ref{nonvanishingstabilizers} corrects
that statement. The error in the proof in \cite{BoSc1}
is that
 \cite[Corollary 3.20 ]{BoSc1} is not true.
The error in the proof  of  \cite[Corollary 3.20 ]{BoSc1}
 is the claim of existence of the
map $f_1$. Under $j$, a monic matrix is carried to a
reverse monic matrix, which need not be invertible
in $\Omega_{+}^{-1}\calR [t]$; so we cannot apply the
universal property of the Cohn localization to produce
$f_1$. \end{remark}

\begin{proof}[Proof of Theorem \ref{nonvanishingstabilizers}]
The last statement of Theorem \ref{nonvanishingstabilizers}
recapitulates for reference a well known fact. (A ring is reduced if it
contains no nontrivial nilpotent element;  for  a nilpotent matrix
$N$ over a commutative
and reduced ring
$\mathcal R$, the unit $\det (I-tN)$ in $\calR[t]$ must equal 1.) So, we
only need to prove the claim in Theorem \ref{nonvanishingstabilizers}
for $NSK_1(\mathcal R)$.
\iffalse Also, if $C$ is a matrix in $\mathcal
(\calR[t])$ with $C_0$ in $\GL (\calR)$, then
(by Proposition \ref{xyz}) there is a matrix $A$ over $\calR$ such
\begin{equation}\ref{ww}
\bigcup_{A\in \mathcal M(\calR )}\ear  = \bigcup_{C\in \GL
  (R[t])\colon\, \\ C_0\in \GL (\calR)} \seort (C)
\end{equation}
 that  $\seort (C)=\ear $. Therefore
\fi

Let $j\colon \mathcal R[t] \to S_{RMP}^{-1}\calR[t]$ be the localization map,
with induced map $j_*\colon K_1(\calR [t]) \to K_1(S_{RMP}^{-1}\calR[t])$.
If $A\in \mathcal M(\calR) $ and
$U\in \ear$, then $[U]\in \ker(j_*)$, because $I-tA$ is invertible
over $S_{RMP}^{-1}\calR[t]$. Therefore
$ \ear \subset NK_1(\calR) $ follows from
$\ker(j_{*}) \subset NSK_{1}(\mathcal{R})$, which is part of the
next proposition.

\begin{prop}\label{nk1iskernel}
$NSK_{1}(\mathcal{R}) = \ker(j_{*})$.
\begin{proof}
Let $\tau_1\colon \calR[t] \to \calR$ and $\tau_2\colon S_{RMP}^{-1}\calR[t]\to \calR$
be the maps induced by $t\mapsto 0$. Then $\tau_1 = \tau_2j$ (because
$t\mapsto 0$ sends elements of $S_{RMP}$ to 1), so $j_*(x)=0$
implies $x\in \text{ker} ((\tau_1)_*)=NK_1(\calR)$.
Also, if $j_*(x)=0$  and $U$ is a matrix such that
$x=[U]$, then $\det(U) =1$. Therefore
$\ker (j_*) \subset NSK_1(\calR )$.

Now suppose $x\in NSK_1(\calR )$.
 By Higman's trick \cite[III.3.5.1]{WeibelBook},
$x = [I-tN]$ for some nilpotent matrix $N$ over $\mathcal{R}$. The
matrix $I-tN$ has the property that all diagonal entries are units in
$S_{RMP}^{-1}\mathcal{R}[t]$, and all off-diagonal entries lie in
$t\mathcal{R}[t]$. It follows that $I-tN$ is elementary equivalent
over $S_{RMP}^{-1}\mathcal{R}[t]$ to a $1\times 1$ matrix.
Thus $j_{*}([I-tN]) = [ (\det(I-tN)) ]= [(1)]=[I]$, and
$NSK_1(\calR) \subset \ker (j_*)$.
\end{proof}
\end{prop}
Before continuing the proof of Theorem
  \ref{nonvanishingstabilizers},
we pause to note there are rings for which
$\nk_1$ is nontrivial, but the elementary
stabilizer is always trivial.
\begin{remark}\label{goodrings}
Suppose $\calR$ is a commutative
    ring for which
the embedding $\sk_1(\calR) \to \sk_1(\calR[t])$
      induced by the inclusion $\calR  \to \calR[t]$ is
      surjective.
    (For example,
    $\calR=\mathcal S [x]/(x^N)$, with $N>1$ and
    $\mathcal S$ a commutative regular ring
    \cite[Example III.3.8.1]{WeibelBook}.)
    Then
$NSK_1(\calR)=\{0\}\neq \nk_1$,
 and $\ear$ is trivial
    for every matrix $A$ over $\calR$.
\end{remark}

To finish the proof of Theorem \ref{nonvanishingstabilizers}, it
suffices given $x\in NK_1(\mathcal R)$ to find $A$ in $\mathcal
M(\calR )$ such that $\ear$ contains $x$. Moreover, by
\eqref{ww}, its suffices to find $A$ in $\mathcal M(\calR)$
such that $\seort(A)$ contains $x$ and
\begin{equation}\label{007}
 A_0 \in \GL (\calR).
\end{equation}

Recall $\mathcal{H}_{1}(\mathcal{R}[t])$ denotes the
category of finitely generated $\calR[t]$-modules of projective
dimension $\le 1$. For a multiplicatively closed set of
non-zero-divisors $S \subset \mathcal{R}[t]$ we let
$\mathcal{H}_{1}(\mathcal{R}[t],S)$ denote the full subcategory of
$\mathcal{H}_{1}(\mathcal{R}[t])$ whose objects are also $S$-torsion
modules, i.e. modules $M$ such that $S^{-1}M = S^{-1}\mathcal{R}[t]
\otimes M = 0$. The categories $\mathcal{H}_{1}(\mathcal{R}[t])$ and
$\mathcal{H}_{1}(\mathcal{R}[t],S)$ are exact categories. The elements
of the multiplicative set $S_{RMP}$ are non-zero-divisors in $\calR
[t]$.  \\

Let $\mathcal{A}$ be an exact category. By a \emph{double short exact
  sequence
(d.s.e.s.)}  we mean a pair of short exact sequences in $\mathcal{A}$ on the same objects
\begin{gather*}
%\begin{split}
\label{doubleshortexact}
0 \to A \stackrel{f_{1}}\longrightarrow B \stackrel{g_{1}}\longrightarrow C \to 0\\
0 \to A \stackrel{f_{2}}\longrightarrow B \stackrel{g_{2}}\longrightarrow C \to 0
%\end{split}
\end{gather*}
which we denote by $A \doublearrow{f_{1}}{f_{2}} B \doublearrow{g_{1}}{g_{2}} C$. In \cite{Nenashev}, Nenashev defined the following group.
\begin{definition} \label{DAdefn}
The group $D(\mathcal{A})$ is defined to be the abelian group with
generators
$\langle\ell \rangle$ for all double short exact sequences $\ell$ subject to the following relations:
 \\
\begin{enumerate}[(i)]
\item
The class of any double short exact sequence of the form $A
\doublearrow{f}{f} B \doublearrow{g}{g} C$ is zero.
(A d.s.e.s. of this form is called {\it diagonal}.)
\item
Suppose we have a diagram consisting of six double short exact sequences of the form
\[
\xymatrix{
 A_{0} \xydr \xydd & A_{1}  \xydd \xydr & A_{2} \xydd\\
 B_{0} \xydr \xydd & B_{1} \xydd \xydr & B_{2} \xydd \\
 C_{0} \xydr & C_{1} \xydr & C_{2} \\
}
\]
satisfying the following commutativity conditions: all arrows on top
commute with all arrows on the left, and all arrows on bottom commute
with all arrows on the right.
(We will a diagram of this form a {\it Nenashev diagram}.) Then, letting $r_{i}$ denote the $i$th row in the diagram, $c_{i}$ the $i$th column in the diagram, we have the relation
$$[c_{0}] - [c_{1}] + [c_{2}] = [r_{0}] - [r_{1}] + [r_{2}]$$

\end{enumerate}
\end{definition}
Nenashev  proved in \cite{Nenashev} the following. Here $K_{1}(\mathcal{A})$ refers to Quillen's $K_{1}$ group.
\begin{theorem}[{{\cite[Nenashev]{Nenashev}}}]\label{NenashevsTheorem}
$K_{1}(\mathcal{A}) \cong D(\mathcal{A})$.
%Every element of $K_{1}(\mathcal{A})$ is represented by a loop corresponding to a double short exact sequence $A \doublearrow{f_{1}}{f_{2}} B \doublearrow{g_{1}}{g_{2}} C$. Moreover, the group $K_{1}(\mathcal{A})$ has the following presentation:
%\begin{enumerate}
%\item
%$K_{1}(\mathcal{A})$ is the abelian group generated by double short exact sequences $$A \doublearrow{f_{1}}{f_{2}} B \doublearrow{g_{1}}{g_{2}} C$$ subject to the following relations:
%\begin{enumerate}
%\end{enumerate}
\end{theorem}
\indent Nenashev's Theorem (Theorem \ref{NenashevsTheorem}) is based
on the Gillet-Grayson construction $G.\mathcal{A}$ associated to
$\mathcal{A}$,
whereby $K_n(\mathcal A)$ is presented as $\pi_n(G.\mathcal{A})$. To a
d.s.e.s. $A \doublearrow{f_{1}}{f_{2}} B \doublearrow{g_{1}}{g_{2}} C$ one may associate a loop in $\pi_{1}(G.\mathcal{A}) = K_{1}(\mathcal{A})$, and this produces a map $m\colon D(\mathcal{A}) \to \pi_{1}(G.\mathcal{A})$ enacting the isomorphism in Nenashev's Theorem.\\
\indent The Bass $K_{1}$ group of $\mathcal{A}$, denoted
$K_{1}^{\det}(\mathcal{A})$, is the group
$K_{0}(\textbf{Aut}\mathcal{A})/R$, where $R$ is the subgroup
generated by elements of the form $[(M,\alpha)]+[(M,\beta)] -
[(M,\alpha \beta)]$. There is a homomorphism
$\eta_{\mathcal A}\colon K_{1}^{\det}(\mathcal{A}) \to K_{1}(\mathcal{A})$, often called the Gersten-Sherman map, defined in \cite[Section 5]{Gersten}, \cite[Section 3]{Sherman1}. In fact, $\eta$ defines a natural transformation between the functors $K_{1}^{\det}$ and $K_{1}$. In general, $\eta_{\mathcal{A}}$ is neither surjective \cite[Prop. 5.1]{Gersten} nor injective \cite[Prop. 5.2]{Gersten}. In the case $\mathcal{A} = \textbf{Proj}\mathcal{R}$ is the category of finitely generated projective $\mathcal{R}$-modules, the map $\eta_{\mathcal{A}}$ is an isomorphism \cite[Section 3]{Sherman1}. \\
\indent The map $\eta_{\mathcal{A}}\colon K_{1}^{\det}(\mathcal{A}) \to
K_{1}(\mathcal{A})$ factors as $K_{1}^{\det}(\mathcal{A})
\stackrel{\gamma}\to D(\mathcal{A})
 \stackrel{m}\to K_{1}(\mathcal{A})$, where the map $\gamma$
is determined as follows \cite[pg.198]{Nenashev}:
for an automorphism $\alpha\colon M \to M$ in $\mathcal{A}$ and its
corresponding class $[M,\alpha]$ in $K_{1}^{\det}(\mathcal{A})$,
% $\gamma $ maps via
$\gamma \colon [M,\alpha] \mapsto \big\langle 0 \doublearrow{}{} M
\doublearrow{1}{\alpha} M \big\rangle$.
%$$\gamma:[M,\alpha] \mapsto \big\langle 0 \doublearrow{}{} M \doublearrow{1}{\alpha} M \big\rangle$$
%\\

\begin{remark}\label{RemarkNenashev} Using the relations for
  $D(\mathcal A)$, one may show the following (as in \cite{Nenashev}):
\begin{enumerate}
\item
If $\alpha \colon M \to M, \beta \colon M \to M$ are automorphisms in $\mathcal{A}$, then
$$\big\langle M \doublearrow{1}{\beta \alpha} M \doublearrow{}{} 0
\big\rangle
 = \big\langle M \doublearrow{1}{\alpha} M \doublearrow{}{} 0 \big\rangle + \big\langle M \doublearrow{1}{\beta} M \doublearrow{}{} 0 \big\rangle$$
\item
Given two d.s.e.s.'s
$$\big\langle M_{1} \doublearrow{f_{1}}{f_{2}} M_{2} \doublearrow{g_{1}}{g_{2}} M_{3} \big\rangle, \hspace{.04in} \big\langle M_{1}^{\prime} \doublearrow{f_{1}^{\prime}}{f_{2}^{\prime}} M_{2}^{\prime} \doublearrow{g_{1}^{\prime}}{g_{2}^{\prime}} M_{3}^{\prime} \big\rangle$$
we have
$$\big\langle M_{1} \doublearrow{f_{1}}{f_{2}} M_{2} \doublearrow{g_{1}}{g_{2}} M_{3} \big\rangle + \big\langle M_{1}^{\prime} \doublearrow{f_{1}^{\prime}}{f_{2}^{\prime}} M_{2}^{\prime} \doublearrow{g_{1}^{\prime}}{g_{2}^{\prime}} M_{3}^{\prime} \big\rangle$$
$$= \big\langle M_{1} \oplus M_{1}^{\prime} \doublearrow{f_{1} \oplus f_{1}^{\prime}}{f_{2} \oplus f_{2}^{\prime}} M_{2} \oplus M_{2}^{\prime} \doublearrow{g_{1} \oplus g_{1}^{\prime}}{g_{2} \oplus g_{2}^{\prime}} M_{3} \oplus M_{3}^{\prime}\big\rangle$$
\end{enumerate}
\end{remark}

\indent Quillen's localization sequence (see \cite[V.7.1]{WeibelBook})
in K-theory for the
localization map $j\colon R[t] \to S_{RMP}^{-1}R[t]$ has the form
\begin{equation}
\label{longexactlocal}
 \cdots \to K_{1}(H_{1}(\mathcal{R}[t],S_{RMP})) \stackrel{i_{*}}\longrightarrow K_{1}(\mathcal{R}[t]) \stackrel{j_{*}}\longrightarrow K_{1}(S_{RMP}^{-1}\mathcal{R}[t]) \to \cdots
 \end{equation}
in which the map $i_{*}$ factors as
$$K_{1}(\mathcal{H}_{1}(\mathcal{R}[t],S)) \stackrel{I_{1,*}}\longrightarrow K_{1}(\mathcal{H}_{1}(\mathcal{R}[t])) \stackrel{I^{-1}_{2,*}}\longrightarrow K_{1}(\textbf{Proj}\mathcal{R}[t])$$
where $I_{1,*}$ is induced by the inclusion functor $I_{1}\colon \mathcal{H}_{1}(\mathcal{R}[t],S) \to \mathcal{H}_{1}(\mathcal{R}[t])$, and $I_{2,*}^{-1}$ is the inverse of the isomorphism $I_{2,*}\colon K_{1}(\textbf{Proj}\mathcal{R}[t]) \to K_{1}(\mathcal{H}_{1}(\mathcal{R}[t]))$ induced by the inclusion $I_{2}\colon \textbf{Proj}\mathcal{R}[t] \to \mathcal{H}_{1}(\mathcal{R}[t])$ (see the proof of \cite[V.7.1]{WeibelBook}). That the map $I_{2,*}$ is an isomorphism follows from Quillen's Resolution Theorem.

\iffalse
Bass constructed \cite[pg. 494]{BassBook}
a localization sequence of the form
$$K_{1}^{\det}(\mathcal{R}[t]) \to K_{1}^{\det}(S^{-1}\mathcal{R}[t]) \to K_{0}(\mathcal{H}_{1}(\mathcal{R}[t],S)) \to K_{0}(\mathcal{R}[t]) \to K_{0}(S^{-1}\mathcal{R}[t])$$
The Gersten-Sherman transformation then gives a diagram
\[ \xymatrix{
& & K_{1}^{\det}(\mathcal{R}[t]) \ar[r] \ar[d]_{\cong}^{\eta_{\mathcal{R}[t]}} & K_{1}^{\det}(S^{-1}\mathcal{R}[t]) \ar[d]_{\cong}^{\eta_{S^{-1}\mathcal{R}[t]}} \\
& K_{1}(H_{1}(\mathcal{R}[t],S)) \ar[r]^{i_{*}} & K_{1}(\mathcal{R}[t]) \ar[r]^{j_{*}} & K_{1}(S^{-1}\mathcal{R}[t])  }\\
\]
where the vertical maps are isomorphisms.

\fi

\begin{lemma}\label{quickhomologicallemma}
Given $M \in \mathcal{H}_{1}(\mathcal{R}[t],S_{RMP})$ there exists $n$,
$M^{\prime} \in \mathcal{H}_{1}(\mathcal{R}[t],S_{RMP})$, and a map $h\colon \mathcal{R}[t]^{n} \to \mathcal{R}[t]^{n}$ such that $M \oplus M^{\prime} \cong \coker(h\colon \mathcal{R}[t]^{n} \to \mathcal{R}[t]^{n})$.
\begin{proof}
Let $0 \to P_{1} \stackrel{f}\longrightarrow P_{0} \to M \to 0$ be a
projective resolution for $M$. Since $f$ is invertible over
$S^{-1}_{RMP}\mathcal{R}[t]$, we may find $s \in S_{RMP}$ such that
$sf^{-1}$ is isomorphic to a map $g\colon P_{0} \to P_{1}$. Since
$P_{1} \oplus P_{0}$ is projective we may choose $Q$ such that $P_{1}
\oplus P_{0} \oplus Q \cong \mathcal{R}[t]^{n}$ for some $n$. Letting
$M^{\prime} =\coker(g)$ we have that
\[
0 \to P_{1} \oplus P_{0} \oplus Q \stackrel{f \oplus g \oplus
  1}\longrightarrow P_{0} \oplus P_{1} \oplus Q \to M \oplus
M^{\prime} \to 0
\]
is exact.
\end{proof}
\end{lemma}

\textbf{Notation: } The notation $\langle \ell \rangle$ refers to the class of a double short exact sequence $\ell$ in $D(\mathcal{A})$. From here on, we will abuse notation and also use $\langle \ell \rangle$ to refer to the image of $\langle \ell \rangle$ under Nenashev's isomorphism $m\colon D(\mathcal{A}) \to K_{1}(\mathcal{A})$. \\

%We are now ready to prove Theorem \ref{nonvanishingstabilizers}.
%\begin{theorem}\label{nonvanishingstabilizers}
%Let $\mathcal{R}$ be a commutative reduced ring. For any $x \in NK_{1}(\mathcal{R})$ there exists a square matrix $A_{G}$ over $\mathcal{R}$ such that $x \in \seort (A_{G})$.
%\end{theorem}

Now suppose
$x\in NSK_1(\calR)$. We will construct the required matrix $A$
($A$ is $\alpha_2\oplus 1$ below) such
that $\ear$ contains $x$.
By Proposition \ref{nk1iskernel} and the
 exactness of \eqref{longexactlocal}, we may fix $y \in K_{1}(\mathcal{H}_{1}(\mathcal{R}[t],S_{RMP}))$ such that $i_{*}(y) = x$, and by Nenashev's Theorem we may represent $y$ in $K_{1}(\mathcal{H}_{1}(\mathcal{R}[t],S_{RMP}))$ by a d.s.e.s. in $\mathcal{H}_{1}(\mathcal{R}[t],S_{RMP})$
$$y = \bigl\langle N_{1} \doublearrow{k_{1}}{k_{2}} N_{2} \doublearrow{l_{1}}{l_{2}} N_{3} \bigr\rangle$$
By Lemma \ref{quickhomologicallemma} we may choose
$N_{1}^{\prime}, N_{3}^{\prime}$ and endomorphisms
$ \alpha_{1}\colon \calR [t]^n \to \calR [t]^n $,
$ \alpha_{3}\colon \calR [t]^m \to \calR [t]^m $
such that $N_{1} \oplus N_{1}^{\prime} \cong \coker(\alpha_{1})$, $N_{3} \oplus N_{3}^{\prime} \cong \coker(\alpha_{3})$. Let $M_{1} = N_{1} \oplus N_{1}^{\prime}$, $M_{2} = N_{2} \oplus N_{1}^{\prime} \oplus N_{3}^{\prime}, M_{3} = N_{3} \oplus N_{3}^{\prime}$, and define
$$f_{1}= \begin{pmatrix} k_{1} & 0 \\ 0 &  1 \\ 0 & 0 \end{pmatrix},\ \
 f_{2} = \begin{pmatrix} k_{2} & 0 \\ 0 &  1 \\ 0 & 0 \end{pmatrix},\ \
 g_{1} = \begin{pmatrix} l_{1} & 0 & 0 \\ 0 & 0 & 1 \end{pmatrix}, \ \
g_{2} = \begin{pmatrix} l_{2} & 0 & 0 \\ 0 & 0 & 1 \end{pmatrix}$$
It follows from part 2 of Remark \ref{RemarkNenashev} that
$$y = \bigl\langle M_{1} \doublearrow{f_{1}}{f_{2}} M_{2} \doublearrow{g_{1}}{g_{2}} M_{3} \bigr\rangle$$
By construction, we have free resolutions for $M_{1}$ and $M_{3}$
\begin{gather*}
%\begin{split}
%\label{doubleshortexact}
0 \to \mathcal{R}[t]^{n} \stackrel{\alpha_{1}}\longrightarrow \mathcal{R}[t]^{n} \stackrel{\pi_{1}}\longrightarrow M_{1} \to 0\\
0 \to \mathcal{R}[t]^{m} \stackrel{\alpha_{3}}\longrightarrow \mathcal{R}[t]^{m} \stackrel{\pi_{3}}\longrightarrow M_{3} \to 0
%\end{split}
\end{gather*}
giving the following diagram in $\mathcal{H}_{1}(\mathcal{R}[t])$
\[
\xymatrix{
 \mathcal{R}[t]^{n} \ar@<-.5ex>[d]_{\alpha_{1}} \ar@<.5ex>[d]^{\alpha_{1}} &  & \mathcal{R}[t]^{m} \ar@<-.5ex>[d]_{\alpha_{3}} \ar@<.5ex>[d]^{\alpha_{3}}\\
 \mathcal{R}[t]^{n} \ar@<-.5ex>[d]_{\pi_{1}} \ar@<.5ex>[d]^{\pi_{1}} & & \mathcal{R}[t]^{m} \ar@<-.5ex>[d]_{\pi_{3}} \ar@<.5ex>[d]^{\pi_{3}} \\
 M_{1} \ar@<-.5ex>[r]_{f_{2}} \ar@<.5ex>[r]^{f_{1}} & M_{2} \ar@<-.5ex>[r]_{g_{2}} \ar@<.5ex>[r]^{g_{1}} & M_{3} \\
}
\]
Using the Horseshoe Lemma (\cite[2.2.8]{WeibelHomological}) we may
fill this in to get
the Nenashev diagram
 \[
\xymatrix{
 \mathcal{R}[t]^{n} \ar@<-.5ex>[d]_{\alpha_{1}} \ar@<.5ex>[d]^{\alpha_{1}} \ar@<-.5ex>[r] \ar@<.5ex>[r]  & \mathcal{R}[t]^{n} \oplus \mathcal{R}[t]^{m} \ar@<-.5ex>[d]_{\alpha_{2}} \ar@<.5ex>[d]^{\alpha_{2}} \ar@<-.5ex>[r] \ar@<.5ex>[r]  & \mathcal{R}[t]^{m} \ar@<-.5ex>[d]_{\alpha_{3}} \ar@<.5ex>[d]^{\alpha_{3}}\\
 \mathcal{R}[t]^{n} \ar@<-.5ex>[d]_{\pi_{1}} \ar@<.5ex>[d]^{\pi_{1}} \ar@<-.5ex>[r] \ar@<.5ex>[r] &  \mathcal{R}[t]^{n} \oplus \mathcal{R}[t]^{m} \ar@<-.5ex>[d]_{\pi_{2}^{\prime}} \ar@<.5ex>[d]^{\pi_{2}} \ar@<-.5ex>[r] \ar@<.5ex>[r]  & \mathcal{R}[t]^{m} \ar@<-.5ex>[d]_{\pi_{3}} \ar@<.5ex>[d]^{\pi_{3}} \\
 M_{1} \ar@<-.5ex>[r]_{f_{2}} \ar@<.5ex>[r]^{f_{1}} & M_{2} \ar@<-.5ex>[r]_{g_{2}} \ar@<.5ex>[r]^{g_{1}} & M_{3} \\
}
\]
Here $\pi_{2}^{\prime} = f_{1}\pi_{1} \oplus \pi^{(g_{1})}_{3}$ and $\pi_{2} = f_{2}\pi_{1} \oplus \pi^{(g_{2})}_{3}$, where $\pi^{(g_{1})}_{3}\colon \mathcal{R}[t]^{m} \to M_{2}$ is a lift of $\pi_{3}$ along $g_{1}$, and $\pi^{(g_{2})}_{3}\colon \mathcal{R}[t]^{m} \to M_{2}$ is a lift of $\pi_{3}$ along $g_{2}$. The horizontal maps in the first and second row are the canonical inclusions and projections.\\
\indent It follows from Nenashev's Theorem that in $K_{1}(\mathcal{H}_{1}(\mathcal{R}[t]))$ we have
\begin{gather*}
I_{1,*}(y) = \bigl\langle M_{1} \doublearrow{f_{1}}{f_{2}} M_{2} \doublearrow{g_{1}}{g_{2}} M_{3} \bigr\rangle = - \bigl\langle \mathcal{R}[t]^{n} \oplus \mathcal{R}[t]^{m} \doublearrow{\alpha_{2}}{\alpha_{2}} \mathcal{R}[t]^{n} \oplus \mathcal{R}[t]^{m} \doublearrow{\pi_{2}^{\prime}}{\pi_{2}} M_{2} \bigr\rangle\\
= \bigl\langle \mathcal{R}[t]^{n} \oplus \mathcal{R}[t]^{m} \doublearrow{\alpha_{2}}{\alpha_{2}} \mathcal{R}[t]^{n} \oplus \mathcal{R}[t]^{m} \doublearrow{\pi_{2}}{\pi_{2}^{\prime}} M_{2} \bigr\rangle
\end{gather*}
A proof for the second equality
follows that of \cite[Lemma 3.4]{Nenashev}.\\

We will use the following result, extracted from
Warfield's presentation
\cite[pp. 1816-1817]{Warfield1978}
of some results from a 1936 paper of Fitting \cite{Fitting1936}
(which were generalized by Warfield).
For the convenience of the reader we include a proof.
Below and later, we identify $M$ and $M\oplus \{0\}$.
\begin{theorem}\label{WarfieldLemma}\cite{Fitting1936,Warfield1978}
Let $\mathcal{S}$ be a ring, and suppose $\pi_{1}\colon P \to M$,
$\pi_{2}
\colon Q
\to M$ are cokernels for injective $\mathcal{S}$-module maps
$h_{1}\colon P
\to P$, $h_{2}\colon Q \to Q$ respectively, with $P,Q$ f.g. projective
$\mathcal{S}$-modules. Then there exist
$S$-module isomorphisms $\phi_{1}$, $\phi_{2}$ making the following diagram commute
\[
\xymatrix{
 P \oplus Q \ar[r]^{h_{1} \oplus 1} \ar[d]_{\phi_{2}} & P \oplus Q \ar[r]^{\pi_{1} \oplus 0} \ar[d]_{\phi_{1}} & M \ar[d]^{1} \\
 P \oplus Q \ar[r]^{1 \oplus h_{2}} & P \oplus Q \ar[r]^{0 \oplus \pi_{2}} & M \\
}
\]
Moreover, $\phi_{1}$ can be chosen such that
%$[P \oplus Q,\phi_{1}] = 0$ in $K^{\det}_{1}(\mathcal{S})$,
$[P \oplus Q,\phi_{1}] = 0$ in $K^{\det}_{1}(\mathcal{S})$,
and hence
$\eta_{\textbf{Proj}\mathcal{S}}([P \oplus Q,\phi_{1}]) = 0$ in
$K_{1}(\textbf{Proj}\mathcal{A})$. Equivalently, upon choosing a
projective module $Q^{\prime}$ such that $P \oplus Q \oplus
Q^{\prime}$ is free and a matrix $M_{\phi_{1}}$ to represent the map
$\phi_{1} \oplus 1\colon P \oplus Q \oplus Q^{\prime} \to P \oplus Q \oplus Q^{\prime}$, $M_{\phi_{1}}$ is elementary.
%, then $A$ is an $n \times n$ matrix over $\mathcal{R}[t]$ and $f_{A}\colon \mathcal{R}[t]^{n} \to \mathcal{R}[t]^{n}$ is the map induced by $A$. Suppose $M \cong \coker(f_{A}\colon \mathcal{R}[t]^{n} \to \mathcal{R}[t]^{n})$ and $\varphi\colon M \to M$ is an $\mathcal{R}[t]$-module automorphism of $M$. Then there exists a matrix $U \in GL(\mathcal{R}[t])$ and an elementary matrix $E \in El(\mathcal{R}[t])$ such that $EA = AU$.
\end{theorem}

\begin{proof}
In the proof, we will use a matrix formalism to denote maps from
$P \oplus Q$.
For example, $(\pi_1, \pi_2)$ denotes the map $P \oplus Q \to M$
which sends a pair $(p,q)$ (viewed as a column vector)
to $\pi_1(p) + \pi_2(q)$. Because $P$ is projective
and $\pi_2$
%(as a cokernel)
is surjective,
we may choose $\rho\colon P\to Q$ such that
$\pi_1=\pi_2 \rho$; likewise we have
$\psi\colon Q\to P$ such that
$\pi_2 = \pi_1 \psi$.
 Then
\[
(\pi_1,0)
\begin{pmatrix} 1& \psi \\ 0 &1
\end{pmatrix}
= (\pi_1, \pi_2) =
(0, \pi_2)
\begin{pmatrix} 1& 0\\ \rho  &1
\end{pmatrix}  \  .
\]
Define $\phi_1 \colon P\oplus Q \to P\oplus Q$ to be
$
\left( \begin{smallmatrix} 1& 0 \\ \rho &1
\end{smallmatrix} \right)
\left( \begin{smallmatrix} 1& -\psi \\ 0 &1
\end{smallmatrix} \right) $.
Because we identify $M$ and $M \oplus \{0\}$,
$(\pi_1,0) = (0,\pi_2) \phi_1$ means
$\pi_1 \oplus 0  = (0\oplus \pi_2) \phi_1$, as required.
From the
triangular forms, we see $\phi_1$ is trivial as an element of
$K_1(S)$.

We have
\begin{align*}
  \text{image}(h_{1} \oplus 1) &
  =
\text{image}(h_{1}) \oplus Q =
\ker( (\pi_1,0)) \quad \text{and} \\
\text{image}(1 \oplus h_{2}) & =
P \oplus \text{image}(h_{2}) =
\ker( (0,\pi_2) )\ .
\end{align*}
Because $\phi_1 $ takes $\ker ((0,\pi_1))$ onto
$\ker ((\pi_2,0))$, it follows that
$\phi_1$ maps
$  \text{image}(h_{1} \oplus 1) $
onto
$  \text{image}(1 \oplus h_{2}) $. Because $\pi_1$ and
$\pi_2$ are injective, the maps
$h_{1}\oplus 1$ and $1 \oplus h_{2}$ are isomorphisms onto
their images. Therefore, for the given
isomorphism $\phi_1$
 there is a unique isomorphism
$\phi_2\colon P\oplus Q \to P \oplus Q$ such that
$(h_{1}\oplus 1)\phi_1 = \phi_2 (1\oplus h_{2})$.
\end{proof}

Resuming the proof
of Theorem \ref{nonvanishingstabilizers}, we have
\begin{equation} \label{resuming}
x= I_{1,*}(y)= \big\langle
\mathcal{R}[t]^{n} \oplus \mathcal{R}[t]^{m}
\doublearrow{\alpha_{2}}{\alpha_{2}} \mathcal{R}[t]^{n} \oplus
\mathcal{R}[t]^{m} \doublearrow{\pi_{2}}{\pi_{2}^{\prime}} M_{2}
\big\rangle
\end{equation}
\iffalse
we consider the top and bottom short exact sequences individually
\begin{gather*}
0 \to \mathcal{R}[t]^{n} \oplus \mathcal{R}[t]^{m} \overset{\alpha_{2}}\longrightarrow \mathcal{R}[t]^{n} \oplus \mathcal{R}[t]^{m} \overset{\pi_{2}}\longrightarrow M_{2} \to 0\\
0 \to \mathcal{R}[t]^{n} \oplus \mathcal{R}[t]^{m} \overset{\alpha_{2}}\longrightarrow \mathcal{R}[t]^{n} \oplus \mathcal{R}[t]^{m} \overset{\pi_{2}^{\prime}}\longrightarrow M_{2} \to 0
\end{gather*}
\fi
We apply Theorem \ref{WarfieldLemma} to the top and bottom short exact sequence
of \eqref{resuming} to get
\begin{equation}\label{sespair1}
\begin{gathered}
\xymatrix{
0 \ar[r] &  \mathcal{R}[t]^{n+m} \oplus \mathcal{R}[t]^{n+m}\ar[r]^{\hspace{.1in}\alpha_{2} \oplus 1} \ar[d]_{\phi_{2}} & \mathcal{R}[t]^{n+m} \oplus \mathcal{R}[t]^{n+m} \ar[d]^{\phi_{1}}\ar[r]^{\hspace{.5in}\pi_{2} \oplus 0} & M_{2} \ar[r] \ar[d]^{1} & 0 \\
0 \ar[r] & \mathcal{R}[t]^{n+m} \oplus \mathcal{R}[t]^{n+m}\ar[r]^{\hspace{.1in}1 \oplus \alpha_{2}}  & \mathcal{R}[t]^{n+m} \oplus \mathcal{R}[t]^{n+m} \ar[r]^{\hspace{.5in}0 \oplus \pi_{2}^{\prime}} & M_{2} \ar[r] & 0 \\
}
\end{gathered}
\end{equation}
with $[\phi_1]$ trivial in $K_1(\calR [t])$.

We may regard the morphisms of \eqref{sespair1} as matrices.
Because
$\phi_1 \in \El (\calR [t])$ and $\alpha_2\oplus 1$ and $1 \oplus
\alpha_2$ are $\El (\calR [t])$ equivalent, we have
$\phi_2 \in \seort (\alpha_2 \oplus 1)$.
Because
$M_{2} \cong \coker(\alpha_{2})$ is a  $S_{RMP}$-torsion module,
there is some $q$ in $S_{RMP}$ such that
$q\calR[t]^{n+m} \subset \text{image}(\alpha_2 \oplus 1)$; therefore the injective map
$\alpha_2\oplus 1$ defines an
automorphism of $S_{RMP}^{-1}\calR [t]^{n+m}$, and the image
of $\alpha_2\oplus 1$
under $t\mapsto 0$ lies in $\GL (\calR)$, as required (recall
the
condition \eqref{007}).
So, to finish the
proof of Theorem \ref{nonvanishingstabilizers} it suffices to  show
 $[\phi_2] = x$.

Using
the pair (\ref{sespair1}) of short exact sequences as column one, we
get the following
Nenashev diagram in $\mathcal{H}_{1}(\mathcal{R}[t])$
\begin{equation}\label{diagram1}
\begin{gathered}
\xymatrix{
 \mathcal{R}[t]^{n+m} \oplus \mathcal{R}[t]^{n+m} \ar@<-.5ex>[d]_{\alpha_{2} \oplus 1} \ar@<.5ex>[d]^{1 \oplus \alpha_{2}} \ar@<-.5ex>[r]_{1} \ar@<.5ex>[r]^{\phi_{2}}  & \mathcal{R}[t]^{n+m} \oplus \mathcal{R}[t]^{n+m} \ar@<-.5ex>[d]_{1 \oplus \alpha_{2}} \ar@<.5ex>[d]^{1 \oplus \alpha_{2}} \ar@<-.5ex>[r] \ar@<.5ex>[r]  & 0 \ar@<-.5ex>[d] \ar@<.5ex>[d]\\
 \mathcal{R}[t]^{n+m} \oplus \mathcal{R}[t]^{n+m} \ar@<-.5ex>[d]_{\pi_{2} \oplus 0} \ar@<.5ex>[d]^{0 \oplus \pi_{2}^{\prime}} \ar@<-.5ex>[r]_{1} \ar@<.5ex>[r]^{\phi_{1}} &  \mathcal{R}[t]^{n+m} \oplus \mathcal{R}[t]^{n+m} \ar@<-.5ex>[d]_{0 \oplus \pi_{2}^{\prime}} \ar@<.5ex>[d]^{0 \oplus \pi_{2}^{\prime}} \ar@<-.5ex>[r] \ar@<.5ex>[r]  & 0 \ar@<-.5ex>[d] \ar@<.5ex>[d]\\
 M_{2} \ar@<-.5ex>[r]_{1} \ar@<.5ex>[r]^{1} & M_{2} \ar@<-.5ex>[r] \ar@<.5ex>[r] & 0 \\
}
\end{gathered}
\end{equation}
Letting $c_{i}, r_{i}$ denote the $i$th column, row, respectively, of
diagram
\eqref{diagram1}, we have
$$\bigl\langle c_{2} \bigr\rangle = \bigl\langle c_{3} \bigr\rangle = \bigl\langle r_{3} \bigr\rangle = 0$$
Define $\bigl\langle l \bigr\rangle = \bigl\langle c_{1} \bigr\rangle$. The diagram \eqref{diagram1}, together with Nenashev's relations, implies
\begin{equation}\label{arelation}
\bigl\langle l \bigr\rangle = \bigl\langle r_{1} \bigr\rangle - \bigl\langle r_{2} \bigr\rangle
\end{equation}
We claim that $$\bigl\langle l \bigr\rangle = \bigl\langle \mathcal{R}[t]^{n} \oplus \mathcal{R}[t]^{m} \doublearrow{\alpha_{2}}{\alpha_{2}} \mathcal{R}[t]^{n} \oplus \mathcal{R}[t]^{m} \doublearrow{\pi_{2}}{\pi_{2}^{\prime}} M_{2} \bigr\rangle = I_{1,*}(y) $$
To see this, let $E = \begin{pmatrix}0 & -1 \\ 1 & 0 \end{pmatrix}$
and consider the
Nenashev diagram
\begin{equation} \label{ttt}
\xymatrix{
 \mathcal{R}[t]^{n+m} \oplus \mathcal{R}[t]^{n+m} \ar@<-.5ex>[d]_{\alpha_{2} \oplus 1} \ar@<.5ex>[d]^{\alpha_{2} \oplus 1} \ar@<-.5ex>[r]_{E} \ar@<.5ex>[r]^{1}  & \mathcal{R}[t]^{n+m} \oplus \mathcal{R}[t]^{n+m} \ar@<-.5ex>[d]_{\alpha_{2} \oplus 1} \ar@<.5ex>[d]^{1 \oplus \alpha_{2}} \ar@<-.5ex>[r] \ar@<.5ex>[r]   & 0 \ar@<-.5ex>[d] \ar@<.5ex>[d]\\
 \mathcal{R}[t]^{n+m} \oplus \mathcal{R}[t]^{n+m} \ar@<-.5ex>[d]_{\pi_{2} \oplus 0} \ar@<.5ex>[d]^{\pi_{2}^{\prime} \oplus 0} \ar@<-.5ex>[r]_{E} \ar@<.5ex>[r]^{1} & \mathcal{R}[t]^{n+m} \oplus \mathcal{R}[t]^{n+m} \ar@<-.5ex>[d]_{\pi_{2} \oplus 0} \ar@<.5ex>[d]^{0 \oplus \pi_{2}^{\prime}} \ar@<-.5ex>[r] \ar@<.5ex>[r]  & 0 \ar@<-.5ex>[d] \ar@<.5ex>[d]\\
 M_{2} \ar@<-.5ex>[r]_{1} \ar@<.5ex>[r]^{1} & M_{2} \ar@<-.5ex>[r] \ar@<.5ex>[r] & 0 \\
}
\end{equation}

Using Nenashev's relations on \eqref{ttt} to justify
\eqref{forttt}, we have
%Since $E$ is elementary, we have
%$$0 =\big[\mathcal{R}[t]^{n+m} \oplus \mathcal{R}[t]^{n+m}  \doublearrow{1}{E} \mathcal{R}[t]^{n+m} \oplus \mathcal{R}[t]^{n+m} \doublearrow{}{} 0 \big]$$
%(see the remark following Nenashev's Theorem), and
\begin{align}
\bigl\langle l \bigr\rangle := \ &\bigl\langle\mathcal{R}[t]^{n+m} \oplus
\mathcal{R}[t]^{n+m} \doublearrow{\alpha_{2} \oplus 1}{1 \oplus
  \alpha_{2}} \mathcal{R}[t]^{n+m} \oplus \mathcal{R}[t]^{n+m}
\doublearrow{\pi_{2} \oplus 0}{0 \oplus \pi_{2}^{\prime} }
M_{2}\bigr\rangle\\ \label{forttt}
=\ &
\bigl\langle\mathcal{R}[t]^{n+m} \oplus \mathcal{R}[t]^{n+m} \doublearrow{\alpha_{2} \oplus 1}{\alpha_{2} \oplus 1} \mathcal{R}[t]^{n+m} \oplus \mathcal{R}[t]^{n+m} \doublearrow{\pi_{2} \oplus 0}{\pi_{2}^{\prime} \oplus 0} M_{2}\bigr\rangle
\\
=\ &
\bigl\langle\mathcal{R}[t]^{n} \oplus \mathcal{R}[t]^{m} \doublearrow{\alpha_{2}}{\alpha_{2}} \mathcal{R}[t]^{n} \oplus \mathcal{R}[t]^{m} \doublearrow{\pi_{2}}{\pi_{2}^{\prime}} M_{2}\bigr\rangle \nonumber\\
\  &
+ \bigl\langle\mathcal{R}[t]^{n+m} \oplus \mathcal{R}[t]^{n+m}
\doublearrow{1}{1} \mathcal{R}[t]^{n+m} \oplus \mathcal{R}[t]^{n+m}
\doublearrow{0}{0} M_{2}\bigr\rangle\nonumber
\ , \quad \text{by \ref{DAdefn}(ii)},
\\
=\ &
\bigl\langle\mathcal{R}[t]^{n} \oplus \mathcal{R}[t]^{m}
\doublearrow{\alpha_{2}}{\alpha_{2}} \mathcal{R}[t]^{n} \oplus
\mathcal{R}[t]^{m} \doublearrow{\pi_{2}}{\pi_{2}^{\prime}}
M_{2}\bigr\rangle + 0 \nonumber
\ , \quad \text{by } \ref{RemarkNenashev}(1),
\\
=\ &
\bigl\langle\mathcal{R}[t]^{n} \oplus \mathcal{R}[t]^{m}
\doublearrow{\alpha_{2}}{\alpha_{2}} \mathcal{R}[t]^{n} \oplus
\mathcal{R}[t]^{m} \doublearrow{\pi_{2}}{\pi_{2}^{\prime}}
M_{2}\bigr\rangle \nonumber \ , \quad
\text{by \eqref{resuming}}.
\end{align}
%proving the claim that $\bigl\langle l \bigr\rangle = I_{1,*}(y) = \bigl\langle \mathcal{R}[t]^{n} \oplus \mathcal{R}[t]^{m} \doublearrow{\alpha_{2}}{\alpha_{2}} \mathcal{R}[t]^{n} \oplus \mathcal{R}[t]^{m} \doublearrow{\pi_{2}}{\pi_{2}^{\prime}} M_{2} \bigr\rangle$\\
%\indent Together with \ref{arelation}, we now have that
Therefore
\begin{align*}
I_{1,*}(y)
&:=\
\bigl\langle \ell \bigr\rangle =
\bigl\langle r_{1} \bigr\rangle - \bigl\langle r_{2} \bigr\rangle \ ,
 \quad \text{by \eqref{arelation}},
\\
=\ & \bigl\langle\mathcal{R}[t]^{n+m} \oplus \mathcal{R}[t]^{n+m}  \doublearrow{\phi_{2}}{1} \mathcal{R}[t]^{n+m} \oplus \mathcal{R}[t]^{n+m} \doublearrow{}{} 0 \bigr\rangle\\
 &- \bigl\langle\mathcal{R}[t]^{n+m} \oplus \mathcal{R}[t]^{n+m}  \doublearrow{\phi_{1}}{1} \mathcal{R}[t]^{n+m} \oplus \mathcal{R}[t]^{n+m} \doublearrow{}{} 0 \bigr\rangle  \\
=\ & \bigl\langle\mathcal{R}[t]^{n+m} \oplus \mathcal{R}[t]^{n+m}
\doublearrow{\phi_{2} \phi_{1}^{-1}}{1} \mathcal{R}[t]^{n+m} \oplus
\mathcal{R}[t]^{n+m} \doublearrow{}{} 0 \bigr\rangle \ ,
\quad \text{by  } \ref{RemarkNenashev}(i).
\end{align*}
Applying the map $I_{2,*}^{-1}\colon
K_{1}(\mathcal{H}_{1}(\mathcal{R}[t])) \to K_{1}(\mathcal{R}[t])$, we
get
$$I_{2,*}^{-1}\colon \bigl\langle \mathcal{R}[t]^{n+m} \oplus \mathcal{R}[t]^{n+m}  \doublearrow{\phi_{2} \phi_{1}^{-1}}{1} \mathcal{R}[t]^{n+m} \oplus \mathcal{R}[t]^{n+m} \doublearrow{}{} 0 \bigr\rangle \mapsto [\phi_{2} \phi_{1}^{-1}] = [\phi_{2}] \in K_{1}(\mathcal{R}[t])$$
where the last equality comes from the fact that $\phi_{1}$ is
elementary. Altogether,
\[
x = i_{*}(y) = I_{2,*}^{-1}I_{1,*}(y) = [\phi_{2}]
\]
This finishes the proof of Theorem \ref{nonvanishingstabilizers}.
\end{proof}
\iffalse

\indent It is easy to check from \ref{diagram1} that we have $\phi_{2}\begin{pmatrix}\alpha_{2} & 0 \\ 0 & 1 \end{pmatrix} = \begin{pmatrix} 1 & 0 \\ 0 & \alpha_{2} \end{pmatrix} \phi_{1}$. Furthermore, $\begin{pmatrix} 1 & 0 \\ 0 & \alpha_{2} \end{pmatrix} \phi_{1}$ is elementary equivalent over $\mathcal{R}[t]$ to $\begin{pmatrix} \alpha_{2} & 0 \\ 0 & 1 \end{pmatrix}$, since $\phi_{1}$ is elementary. Thus $\phi_{2} \in \seort{(\alpha_{2} \oplus 1)}$. \\

Let $\alpha_{2}(0)$ denote the image of $\alpha_{2}$ under the map $\mathcal{R}[t] \stackrel{t \mapsto 0}\longrightarrow \mathcal{R}$. It follows that $\alpha_{2}(0)$ is invertible, and we get $\alpha_{2}(0)^{-1}\phi_{2}\alpha_{2}(0) \in \seort(\alpha_{2}(0)^{-1}\alpha_{2})$, since
$$(\alpha_{2}(0)^{-1}\phi_{2}\alpha_{2}(0))\alpha_{2}(0)^{-1}\alpha_{2} = \alpha_{2}(0)^{-1}\phi_{2}\alpha_{2}$$
and $\alpha_{2}(0)^{-1}\phi_{2}\alpha_{2}$ is elementary equivalent over $\mathcal{R}[t]$ to $\alpha_{2}(0)^{-1}\alpha_{2}$. But in $K_{1}(\mathcal{R}[t])$ we have
$$[\alpha_{2}(0)^{-1}\phi_{2}\alpha_{2}(0)] = [\phi_{2}]=[x]$$
and the constant term of $\alpha_{2}(0)^{-1}\alpha_{2}$ is $I$.
\fi
\begin{corollary}
For any finitely generated subgroup $H \subset NSK_{1}(\mathcal{R})$, there exists a matrix $A_{H}$ over $\mathcal{R}$ such that $H \subset \seort(A_{H})$.
\begin{proof}
This follows from Theorem \ref{nonvanishingstabilizers}, along with the fact that if $x \in \seort(A)$ and $y \in \seort(B)$, then $\{x,y\} \subset \seort(A \oplus B)$.
\end{proof}
\end{corollary}

Naturally, one asks what statement would replace Theorem
\ref{nonvanishingstabilizers}
if $\calR$ is not assumed to be commutative.
\begin{conjecture} \label{unionofesconjecture}
For a  ring $\calR$,
\[
\bigcup_{A\in \mathcal R} \ear =
\ker \Big( K_{1}(\mathcal{R}[t]) \stackrel{j_{*}}\longrightarrow
K_{1}(\Sigma_{RMP}^{-1}\mathcal{R}[t]) \Big)
\]
where $\Sigma_{RM}$ is the set of reverse monic matrices (those of the
form $A = I + \sum_{i=1}^{n}A_{i}t^{i}$) and $j_{*}$ is the map on $K_1$ induced by the localization map
$j$ from $\calR [t]$ into its Cohn localization with respect
to $\Sigma_{RM}$.
\end{conjecture}
The conjecture is true when $\calR$ is commutative, where
the Cohn localization with respect to $\Sigma_{RMP}$
can be identified with the standard
localization.\\

Finally, we leave the following problem regarding the structure of the elementary stabilizer groups in general. 

\begin{elstabprob} \label{elstabprob} 
For a square matrix $A$ over a ring $\calR$, find a satisfactory description of the elementary stabilizer $\ear$. In particular, when is $\ear$ trivial?
\end{elstabprob}

\section{SSE/SE($A,\calR$) =
  $\textnormal{NK}_1(\calR)/\ear$}
\label{secfitting}
%\begin{theorem} \label{central}
%Let $\calR$ be a ring. Suppose
%$A,B$ are square matrices over $\calR$. TFAE.
%\begin{enumerate}
%\item  $A$ and $B$ are SSE over $\calR$.
%\item
%$I-tA$ and $I-tB$ are
%$\textnormal{El} (\calR[t])$ equivalent.
%\end{enumerate}
%\end{theorem}
%
%Theorem \ref{central} is an immediate
%corollary of Theorem \ref{finecentral},
%proved later.
%\red{From here to the next red is a rewriting. }

In this section we prove one of our main results,
  Theorem \ref{aplusn}, assuming the main result of the next section,
  Theorem \ref{finecentral}.

%\red{In this section, assuming Theorem \ref{finecentral},
%we prove our  main result, Theorem \ref{aplusn}.}

To begin we state a
matrix version of Theorem \ref{WarfieldLemma}.
A slightly different formulation
of Theorem \ref{fittingtheorem} is given in
\cite[Lemma 9.1]{BS05}, with further commentary.
We say a $k\times k$ matrix $A$ over $\calR$ is injective
if matrix multiplication  $x\mapsto Ax$ defines an injective map
$  \calR^k \to \calR^k$.

\begin{theorem} \label{fittingtheorem}\cite{Fitting1936}
Suppose $A$ and
$B$ are square injective matrices over a ring $\calR$ and the
$\calR$-modules
$\coker (A)$ and $\coker (B)$ are isomorphic.
Then there are identity matrices $I_m,I_n$;
$k\in \N$;  $U$ in $\GL (k,\calR)$;
and  $V$ in $\El (k,\calR)$ such that
$U(A\oplus I_m)V = B\oplus I_n$.
\end{theorem}

Next, we compile some characterizations of shift equivalence
as a theorem.
The equivalence of (1), (2) and (3) below is well known.
The equivalence of (1) and (4) is what we need for Theorem
\ref{aplusn}.
For an $n\times n$ matrix $A$ over a ring $\calR$,
%$ \varinjlim\{\Rcal^{n},A\}$
  the $\calR[t]$ module
$\overline{\Rcal_{A}}$ is
 direct limit $\calR$-module
$\Rcal^{n} \stackrel{A}\to \Rcal^{n} \stackrel{A}\to \Rcal^{n}
\stackrel{A}\to
\cdots$,
  with $t$ acting by $[v,i]\mapsto [v,i+1]$ (inverse to $[v,i]\mapsto [Av,i]$).

For $n\in \N$, $0_n$ and $I_n$ denote the $n\times n$ zero and
identity matrices.
For a square matrix $A$ over $\calR$,
$\ear$ denotes $\seort (I-tA)$,
as in \eqref{eardefn}.

\begin{theorem} \label{selisttheorem}
Suppose $A$ and $B$ are square matrices over a ring $\Rcal$. Then
the following are equivalent.
\begin{enumerate}
\item $A$ and $B$ are shift equivalent over $\Rcal$.
\item $\overline{\Rcal_{A}}$ and
$\overline{\Rcal_{B}} $ are isomorphic $\calR[t]$ modules.
\item  $\coker (I-tA)$ and $\coker (I-tB)$
are isomorphic $\calR[t]$ modules.
\item  There are
$k,m,n\in \N$ and $U,V$ in $\GL (k,\calR[t])$ such that \\
$U\big(( I-t A)\oplus I_m \big)V = \big(  (I-t B)\oplus I_n\big)$ ,
i.e., \\
$U\big(( I-t (A\oplus 0_m) \big)V = \big(  I-t (B\oplus 0_n)\big)$ .
\end{enumerate}
If $A$ and $B$ are shift equivalent over $\calR$, then
$\ear = \ebr$.
\end{theorem}
\begin{proof}
$(1)\iff (2)$ See \cite[p.122]{BH93}.
This connection is due to Krieger;
the result for $\Rcal = \Z$ was a piece of
his introduction of dimension groups to symbolic
dynamics  \cite{KriegerDimMark1980}.
Another proof for the
case $\Rcal = \mathbb{Z}$ can
be found in
\cite[7.5.6--7.5.7]{LindMarcus1995}.

$(2)\iff (3)$ The map
$[v,i] \mapsto [t^{i}v]$
defines an
$\Rcal[t]$-module isomorphism
$\overline{\Rcal_{A}}\to
\coker(I-tA)$.
This connection was introduced
by Kim, Roush and
Wagoner \cite{S7},  for $\Rcal = \Z$.

$(4)\implies (3)$ Clear.

$(3)\implies (4)$
$I-tA$ and  $I-tB$ are injective matrices over $\calR[t]$,
so (4) follows by Theorem \ref{fittingtheorem}.

Because $(1)$ implies $(4)$, the final
claim of the theorem
follows from the final claim of
Proposition \ref{nopun}.
\end{proof}

We let $\sim$ denote $\el (\calR[t])$ equivalence.

%\red{``the next red''}

% Theorem \ref{fittingtheorem} is not true if the
%injectivity hypothesis is removed;
%\red{State Fitting for rectangulars? The identity matrices are added
%such that the two matrices acquire the same number of rows,
%but they don't have to have the same number of columns. So U
%is square but V might not be. The original action is on column
%vectors.}

\begin{theorem}\label{aplusn} Let $\calR$ be a ring, and
$A$ a square matrix over $\calR$. The following hold.
\begin{enumerate}
\item
If $B$ is shift equivalent over $\Rcal$ to $A$,
then there is a nilpotent matrix
$N$ over $\calR$ such that $B$ is SSE over $\calR$
to the matrix $A\oplus N$.
\item
%\label{nilpotentdsum}
For nilpotent matrices $N_1,N_2$
over $\calR$, the matrices $A\oplus N_1$
and  $A\oplus N_2$ are SSE over $\calR$ iff
$I-tN_1$ and $I-tN_2$ are the same
element in $NK_{1}(\calR)/\ear $.
\item
If $A$ is shift equivalent over $\Rcal$ to
    a matrix which is nilpotent, invertible or idempotent,
    then $\ear $ is the trivial group.
\end{enumerate}
\end{theorem}

\begin{proof}
  For the proof of  (1), suppose  $B$ is   shift equivalent over
  $\calR$ to $A$.
Let $k,m,n,U,V$ be as in (4) of Theorem \ref{selisttheorem}.
After replacing $A$ with $A\oplus 0_m $ and
$B$ with $B\oplus 0_n$ (which is harmless),
we have
$(I-tB) = U(I-tA)V$ .
%Consequently we have $U,V$ in $ GL(\calR[t])$
%such that $(I-tB)_{\infty}= U((I-tA)_{\infty})V$.
%We use e.g. $P\sim Q$ to indicate
%$P\oplus I_{\infty}$ and  $Q\oplus I_{\infty}$
%%(or just $P,Q$ if they are already infinite)
%are
%$\text{El}(R[t])$ equivalent, i.e., there are $U,V$ in
%$\text{El}(R[t])$ such that $U(P\oplus I_{\infty})V=Q\oplus I_{\infty}$.
%From here we will drop the $_{\infty}$ subscripts and
%infer their presence or absence by context.

Because
$\left(\begin{smallmatrix} VU & 0 \\ 0 & I
\end{smallmatrix}\right)
\left(\begin{smallmatrix} I-tA & 0 \\ 0 & I
\end{smallmatrix}\right)
=
\left(\begin{smallmatrix} V & 0 \\ 0 & V^{-1}
\end{smallmatrix}\right)
\left(\begin{smallmatrix} U(I-tA)V & 0 \\ 0 & I
\end{smallmatrix}\right)
\left(\begin{smallmatrix} V^{-1} & 0 \\ 0 & V
\end{smallmatrix}\right)
$ ,
%Since $U(I-tA)V \sim VU(I-tA)$ (via
%$\left(\begin{smallmatrix} V & 0 \\ 0 & V^{-1}
%\end{smallmatrix}\right)$ on the left
%and
%$\left(\begin{smallmatrix} V^{-1} & 0 \\ 0 & V
%\end{smallmatrix}\right)$ on the right),
we have $I-tB \sim W(I-tA)$, where
$W=VU$. Setting $t = 0$, we see
$W$ represents an element of $ NK_{1}
(\calR)$. So, for some $j$,
after replacing $W$ with $W\oplus I_j$
 there exists $N$ nilpotent over $\calR$
and $E$ and $F$ elementary over $\calR[t]$
such
that $EWF = I-tN$.
After replacing $A$ with $A\oplus 0_j$, we have
\begin{align*}
I-tB \sim
W(I-tA) &\sim  ( I-tA ) \oplus W \sim
\begin{pmatrix} I & 0 \\ 0 & E \end{pmatrix}
\begin{pmatrix} I-tA & 0 \\ 0 & W\end{pmatrix}
\begin{pmatrix} I & 0  \\ 0 & F\end{pmatrix} \\
& = (I-tA) \oplus ( I-tN )
= I -t(A\oplus N) \ .
\end{align*}
Now Theorem \ref{finecentral} implies $B$
is strong shift equivalent over $\calR$  to $A \oplus N$.
This proves (1).

%\label{nilpotentdsum}
For (2), suppose
$N_1,N_2$ are  nilpotent matrices
over $\calR$. By Theorem \ref{finecentral},
 the matrices $A\oplus N_1$
and  $A\oplus N_2$ are SSE over $\calR$ iff
$(I-t(A\oplus N_1)) \sim (I-t(A\oplus N_2))$.
For $N$ nilpotent,
$(I-t(A\oplus N) \sim (I-tN)(I-tA)$. Therefore
$A\oplus N_1$  and
$A\oplus N_2$ are SSE over $\calR$ iff
$(I-tN_1)(I-tA) \sim (I-tN_2)(I-tA)$.
%\[
%\begin{pmatrix}
%I-tA & 0 \\ 0 & I-tN_2
%\end{pmatrix}
%=
%\begin{pmatrix}
%I & 0 \\ 0 & I-tN_1
%\end{pmatrix}^{-1}
%\begin{pmatrix}
%I-tA & 0 \\ 0 & I-tN_1
%\end{pmatrix}
%\begin{pmatrix}
%I & 0 \\ 0 & I-tN_2
%\end{pmatrix}  \ ,
%\]
%it follows from
By Proposition \ref{nopun}, this holds
%\[
%(I-tN_1)(I-tA) \sim (I-tN_2)(I-tA)
%(I-t(A\oplus N_1)) \sim (I-t(A\oplus N_2))
iff
$(I-tN_1)^{-1}
(I-tN_2)
\in
\seo (I-tA)
$ .
By Theorem \ref{classesinOmega},
this inclusion   holds iff $I-tN_1=I-tN_2$ in
$K_1(\calR[t])/ \ear$
  (equivalently,  in $NK_1(\calR) / \ear $). This proves (2).

(3) holds by Theorem \ref{classesinOmega} and
  the final claim of Theorem \ref{selisttheorem}. Note, the
  nilpotent matrices form the shift equivalence class of
the zero matrices.
\end{proof}

\begin{corollary}\label{nilcor} Suppose $\nk_1(\calR)$ is trivial
  (for example, when $\calR$ is a Noetherian regular ring).
Then SE-$\calR$ implies SSE-$\calR$.
\end{corollary}
Corollary \ref{nilcor} answers in the affirmative
a question
of Wagoner \cite[Sec. 9, Problem Number 3]{Wagoner99}: does
 SE-$\calR$ implies SSE-$\calR$ when $\calR$ is
a commutative regular ring?

Given $\calR$ and a square matrix $B$ over
  $\calR$, let
$[B]_{SSE}$ denote the SSE-$\calR$ class of $B$ and
let $[B]_{SE}$ denote the SE-$\calR$ class. For
a square matrix $A$ over $\calR$, define
\begin{equation}
  \text{SSE}/\text{SE}(A,\calR)
  = \{[B]_{SSE} \hspace{.05in} \colon  \hspace{.05in}
  [A]_{SE} = [B]_{SE}\}\ .
  \end{equation}
We can now give a short summary
of the correspondence provided by Theorem \ref{aplusn}.

\begin{theorem}\label{aplusn2}
Let  $N$ range over
nilpotent matrices over $\calR$. Then
for any  square matrix $A$ over $\calR$,
the map $[I-tN] \to [A \oplus N]_{SSE}$
 is a well-defined bijection
   \[
   NK_{1}(\calR) / \ear \to
  \text{SSE}/\text{SE}(A,\calR)
\]
Equivalently, the map $[N]\to [A \oplus N]_{SSE}$
is a well defined bijection
\[
\nil (\calR)/\enilar \to
 \text{SSE}/\text{SE}(A,\calR)
\]
with $\enilar = \{ [N] \in \nil (\calR):
[I-tN] \in \ear \}$ .
\end{theorem}

%
%Examining the proof
%of Theorem \ref{aplusn}, we can extract a description of
%the refinement of $\GL (R[t])$ equivalence by
%$\EL (R[t])$ equivalence for the stable classes of
%matrices which are of the form $I-tA$, with $A$ over $\calR$, with respect to the
%stabilization identifying $A$ and $A\oplus I_n$ for all $n$.

Using Theorems \ref{selisttheorem} and \ref{finecentral},
we record a restatement of Theorem \ref{aplusn}.

\begin{theorem}\label{aplusn3}  Let $\calR$ be a ring.
%, and let $P\sim Q$ mean that there
%are identity matrices $I_m, I_n$ such that
%$(P\oplus I_m)$ and $(Q\oplus I_n)$ are $\EL (R[t])$ equivalent.
 Then the following hold.
\begin{enumerate}
%\item
%If $A, B$ are square matrices over $\calR[t]$ such that $A, B\in \Omega_{+}$, and the $R[t]$-modules
%$\coker (A)$, $\coker(B)$ are isomorphic,
%then there is a matrix $M$ in $NK_1({\calR})$
%such that $B \sim (A\oplus M)$.
\item
If $A, B$ are square matrices over $\calR$ such that the $\calR[t]$-modules
$\coker(I-tA)$, $\coker(I-tB)$ are isomorphic,
then there is a nilpotent matrix $N$ over $\calR$ such that $I-tB \sim I-t(A\oplus N)$.
\item
%\label{nilpotentdsum}
Suppose $N_1,N_2$ are nilpotent matrices over $\calR$. Then
$$I-t(A\oplus N_1) \sim
I-t(A\oplus N_2)$$ iff
$[I-tN_1]$ and $[I-tN_2]$ are the same
element in $NK_{1}(\calR) / \ear $.
\end{enumerate}
\end{theorem}

\section{SSE as elementary equivalence} \label{sseaselsec}
%section{Strong shift equivalence as elementary equivalence}

The purpose of this section is to prove
Theorem \ref{finecentral},
  our central result for connecting
  strong shift equivalence and algebraic $K$-theory.
%Theorem \ref{central},
%on which all our results depend.
%, is an immediate corollary.
To prepare for its statement,
%the statement of Theorem \ref{finecentral},
 we give some definitions.

\begin{definition} \label{asharpdefinition}
Given $A\in t\rr [t]$, choose $n\in \N$ and $k\in \N$
such that $A_1, \dots A_k$ are $n\times n$ matrices over
$\rr$ such that
\[
A = \sum_{i=1}^k t^iA_i
\]
and define a finite matrix $\As= \mathcal A^{\Box (k,n)}$ over $\rr$ by the
following block form,
in which  every block is $n\times n$:
\[
 \As =
\begin{pmatrix}
A_1 & A_2 &A_3& \dots &A_{k-2}&A_{k-1} & A_k \\
I   & 0   &0  & \dots & 0     & 0      & 0  \\
0   & I   &0  & \dots & 0     & 0      & 0  \\
0   & 0   & I & \dots & 0     & 0      & 0  \\
\dots &\dots &\dots &\dots &\dots &\dots &\dots  \\
0   & 0   & 0 & \dots &I      & 0     & 0  \\
0   & 0   & 0 & \dots &0      & I     & 0
\end{pmatrix} \ .
\]
\end{definition}
In the definition, there is some freedom in
the choice of $\As$: $k$ can be increased by
using zero matrices, and $n$ can be increased
by filling additional entries of the $A_i$
with zero. These choices do not affect the
SSE-$\rr$ class of $\As$.\\

%\begin{remark}
%The $\As$ defined here differs from that defined in \cite{BW04}.
%\end{remark}

With $\sim$ denoting $\el (\calR[t])$ equivalence,
recall that for finite matrices $I-A$ and $I-B$,
$I-A \sim I-B$ by definition means
 $(I-A)_{\stabone} \sim (I-B)_{\stabone}$.

\begin{theorem}\label{finecentral}
 Let $\rr$ be a ring. Then there is a
 bijection between the following sets:
\begin{itemize}
\item the set of
$\textnormal{El}(\rr [t])$ equivalence classes of
%$\N \times \N$
square matrices $I-A$ with $A$
over $t \rr [t]$
\item
the set of SSE-$\rr$ classes of square matrices
 over $\rr$.
\end{itemize}
The  map  to SSE-$\rr$ classes is induced by
the map
 $ I-A \mapsto \As$. The inverse map (from
the set of SSE-$\rr$ classes)
 is induced by the map sending $A$ over $\rr$ to
the
%$\N \times \N$
matrix
$ I-tA$.
\end{theorem}

\begin{proof} We will first show that
when $A$ and $B$ are SSE over $\rr$, it follows that the matrices
$I-tA$ and $I-tB$ are $\text{El}(\rr [t] )$ equivalent.
It suffices to do this for  an elementary strong shift equivalence.
Suppose $U,V$ are matrices over $\rr$ such that
$A=UV$ and $B=VU$. Then
(as pointed out by  Maller and Shub
\cite{MallerShub1985}),
\[
\begin{pmatrix} I & 0 \\ V & I
\end{pmatrix}
\begin{pmatrix} A & U \\ 0 & 0
\end{pmatrix}
\ =\
\begin{pmatrix} 0 & U \\ 0 & B
\end{pmatrix}
\begin{pmatrix} I & 0 \\ V & I
\end{pmatrix}
\]
and therefore
\[
\begin{pmatrix} I & 0 \\ V & I
\end{pmatrix}
\begin{pmatrix} I-tA & -tU \\ 0 & I
\end{pmatrix}
\ =\
\begin{pmatrix} I & -tU \\ 0 & I-tB
\end{pmatrix}
\begin{pmatrix} I & 0 \\ V & I
\end{pmatrix} \ .
\]
Also,
\begin{align*}
\begin{pmatrix} I-tA & -tU \\ 0 & I
\end{pmatrix}
& \ =\
\begin{pmatrix} I & -tU \\ 0 & I
\end{pmatrix}
\begin{pmatrix} I-tA & 0\\ 0 & I
\end{pmatrix}
 \quad \text{ and } \quad \\
\begin{pmatrix} I & -tU \\ 0 & I-tB
\end{pmatrix}
&\ = \
\begin{pmatrix} I & 0 \\ 0 & I-tB
\end{pmatrix}
\begin{pmatrix} I & -tU \\ 0 & I
\end{pmatrix} \ .
\end{align*}
Therefore $I-tA$ and $I-tB$ are
$\textnormal{El}(\rr [t])$ equivalent.

Now suppose that $A$ and $B$ are matrices over $t\rr [t]$ such that
$I-A$ and $I-B$ are $\textnormal{El}(\rr [t])$ equivalent.
We will show that
$ \mathcal A^{\Box}$
and $ \mathcal B^{\Box}$
are SSE
over $\rr$.

There are basic elementary matrices $E_1, \dots , E_j$ and
$F_1,\dots ,F_k$ ,
in each of which the single  nonzero offdiagonal term
has the form $rt^{\ell}$, with $r\in \rr$ and $\ell \geq 0$,
such that
\[
E_j\cdots E_2E_1 (I-A)\ = \ (I-B)F_1F_2\cdots F_k \ .
\]
 Choose the block size $n$ for $A^{\Box}$ and
$B^{\Box}$ large enough that each
 $E_i$ and $F_j$ equals $I$
outside the principal submatrix on indices
$\{1,\dots , n\} \times \{1,\dots , n\}$.
Let $G_i$ denote the image of $E_i$ in $\textnormal{El}(\rr )$
under the map induced by $t\mapsto 0$.
Recursively,
for $0<i \leq j$,
  given $A_{i-1}$ we
will define
 $A_{i}$ over $\rr [t]$ such that
  $(A_{i})^{\Box}$ is
SSE over $\rr$ to $(A_{i-1})^{\Box}$ and also
\begin{align} \label{same}
E_j\cdots E_{i+1} (I-A_{i})\ &= \ (I-B)(F_1F_2\cdots F_k) (G_1)^{-1}\cdots (G_i)^{-1} \quad
\textnormal{if } i<j \\ \notag
 (I-A_{i})\ &= \ (I-B)(F_1F_2\cdots F_k) (G_1)^{-1}\cdots (G_i)^{-1} \quad
\textnormal{if } i=j  \ .
\end{align}
There are two cases.

Case 1: The offdiagonal entry of $E_i$ has the form
$rt^{\ell}$ with $\ell > 0$. In this case,
define $A_i$ by the equation $I-A_i=E_i(I-A_{i-1})$.
By Lemma \ref{centrallemma},  $(A_i)^{\Box}$ is
SSE over $\rr$ to $\As$. Equation (\ref{same}) holds because $G_i=I$.

Case 2:
$E_i$ has all  entries in $\rr$. Then
define $A_i $ over  $t\rr [t]$ by the equation
$I-A_i=E_i (I-A_{i-1})(E_i)^{-1}$.
Equation (\ref{same}) holds because $G_i=E_i$,
so for this case it remains
to check the strong shift equivalence.
Let $E_i$ also denote the restriction of
$E_i$ to the  finite
principal submatrix on indices
$\{1, \dots , n\}
\times
\{1, \dots , n\}$, define
$D$ to be the block diagonal matrix
with $k$ diagonal blocks, each equal to
$E_i$.  Then
$A_i^{\Box}=D^{-1}
A_{i-1}^{\Box} D$, and therefore
$A_i^{\Box}$ is SSE over $\rr$ to
$A_{i-1}$.

Define $G=G_j\cdots G_2G_1\in \text{El}(\rr )$.
From the preceding we have
$\As$ SSE over
$\rr$ to $(A_j)^{\Box}$,
 with
$
I-A_j = (I-B) (F_1F_2\cdots F_k )G^{-1}
$
and therefore
\[
(I-A_j)G = (I-B) (F_1F_2\cdots F_k ) \ .
\]
Let $H_i$ denote the evaluation
of $F_i$ at $t=0$. Repeating the previous
procedure, with the role of left and right
interchanged, we find $B_k$ with $(B_k)^{\Box} $
and $B^{\Box} $ SSE over $\rr$, and
\[
(H_k)^{-1}\cdots (H_2)^{-1}(H_1)^{-1}(I-A_j)G =
(I-B_k) \ .
\]
Define $H=
H_1H_2\cdots H_k$.  Then
$H^{-1}(I-A_j)G = I-B_k$. Evaluating at $t=0$,
we see $H=G$. Then $B_k= G^{-1}A_jG$; as
 in Case 2,
$(A_j)^{\Box }$ is SSE over $\rr$ to $(B_k)^{\Box}$.
This finishes the proof (given Lemma \ref{centrallemma}).

\end{proof}

\begin{lemma}\label{centrallemma}
Let $\rr$ be a ring.
Suppose $A$ and $B$ are matrices over $t\calR[t]$;
$\ell$ is a positive integer;
$E$ is a basic elementary matrix
whose nonzero offdiagonal entry is $E(i_0,j_0) = rt^{\ell} $,
with $r\in \rr$;
and $E(I-A) = I-B$ or  $(I-A)E = I-B$.

Then the matrices $\As$ and $\Bs$ are SSE over $\rr$.
\end{lemma}

\begin{proof}
Without loss of generality, suppose
for notational simplicity that $(i_0,j_0) = (1,2)$.

We first give a proof assuming that
$E(I-A) = I-B$.
Let $A= tA_1 + \cdots + t^k A_k$,
with the $A_i$ over $\rr$,  and
for later notational convenience set
$A_i =0$ if $i>k$.
Since
$E(A-I) = B-I$, we have
$B= EA -E+I = EA -(E-I)I
$.
Therefore $B = tB_1 + \dots + t^{k+l}B_{k+\ell}$,  with
$
B_{\ell}(1,2)\  =\ A_{\ell}(1,2) -r $, and
%
%\[
%B_{i+\ell }(1,j)\  =\ A_{i+\ell }(1,j)+ rA_i{\ell}(2,j)  \ , \quad
%1\leq i \leq k\ ,
%\]
\[
B_{i+\ell }(1,j)\  =\ A_{i+\ell }(1,j)+ rA_i(2,j)  \ , \quad
1\leq i \leq k\ ,
\]
and in all other entries $B=A$.

We  first consider the case $\ell =1$.  Let $X$ be the $n\times n $ matrix
such that $X(1,2)=1$ and other entries of $X$ are zero.
Let $u_i $ be the row vector which is the second row of $A_i$. Let
$U_i$ be the $n\times n$ matrix whose first row is $u_i$ and whose
other rows are zero.
Then the matrix $\Bs$, in block form with
$n\times n$ blocks, is
\[
\Bs =
\begin{pmatrix}
A_1-rX&A_2+rU_1 &A_3+rU_2 & \dots &A_{k-1}+rU_{k-2}& A_k+rU_{k-1}& rU_k    \\
I     & 0       &0        & \dots  & 0            & 0           & 0      \\
0     & I       &0        & \dots  & 0            & 0           & 0      \\
0     & 0       & I       & \dots  & 0            & 0           & 0     \\
\dots &\dots    &\dots    &\dots   &\dots         &\dots        &\dots    \\
0     & 0       & 0       & \dots  & I            & 0           & 0        \\
0     & 0       & 0       & \dots  & 0            & I           & 0      \\
\end{pmatrix} \ .
\]
We we will perform a string of elementary SSEs over $\rr$ which will
transform $\Bs$ into $\As$. We use lines within matrices to emphasize
block patterns, especially for blocking compatible with a multiplication.

First we perform the column splitting which splits off columns which
isolate all entries with coefficient $r$.
Letting  $e_1$ denote the size $n$ column vector
with first entry 1 and other entries zero,
we define the
$n(k+1)\times n(2k+1)+1$ matrix
\[
W=
\left(
\begin{array}{cccccc|ccc|c}
A_1&A_2 & \dots &A_{k-1}& A_k& 0  &rU_1& \cdots &rU_k & -re_1  \\
I  & 0  & \dots  & 0    & 0  & 0  & 0  & \cdots & 0   &0         \\
0  & I  & \dots  & 0    & 0  & 0  & 0  & \cdots & 0   &0         \\
0  & 0  & \dots  & 0    & 0  & 0  & 0  & \cdots & 0   &0        \\
\dots &\dots   &\dots &\dots &\dots  &\dots &\dots &\dots &\dots &\dots     \\
0  & 0  & \dots  & I    & 0  & 0  & 0  & \cdots & 0   &0       \\
0  & 0  & \dots  & 0    & I  & 0  & 0  & \cdots & 0   &0     \\
\end{array}
\right)\ .
\]
and the
$(n(2k+1)+1)\times n(k+1)$ matrix
\[
M=
\begin{pmatrix}
I_n  & 0 \\
0  & I_{nk} \\ \hline
0  & I_{nk} \\ \hline
e_2& 0
\end{pmatrix}
\]
in which $I_j$ as usual  means a $j\times j$ identity matrix and
$e_2$ is the row vector $ \left(
\begin{smallmatrix} 0&1&0&\cdots& 0
\end{smallmatrix}\right)$. Then $\Bs = WM$ and we define
$B^{(1)}= MW$,
 SSE over $\rr$ to $\Bs$. In block form,
\[
B^{(1)}=
\left(
\begin{array}{cccccc|ccc|c}
A_1&A_2 & \dots &A_{k-1}& A_k& 0  &rU_1& \cdots &rU_k & -re_1  \\
I  & 0  & \dots  & 0    & 0  & 0  & 0  & \cdots & 0   &0         \\
0  & I  & \dots  & 0    & 0  & 0  & 0  & \cdots & 0   &0         \\
0  & 0  & \dots  & 0    & 0  & 0  & 0  & \cdots & 0   &0        \\
\dots &\dots   &\dots &\dots &\dots  &\dots &\dots &\dots &\dots &\dots     \\
0  & 0 & \dots  & I    & 0  & 0  & 0  & \cdots & 0   &0       \\
0  & 0 & \dots  & 0    & I  & 0  & 0  & \cdots & 0   &0     \\ \hline
I  & 0 & \dots  & 0    & 0  & 0  & 0  & \cdots & 0   &0         \\
0  & I & \dots  & 0    & 0  & 0  & 0  & \cdots & 0   &0         \\
0  & 0 & \dots  & 0    & 0  & 0  & 0  & \cdots & 0   &0        \\
\dots  &\dots   &\dots &\dots &\dots  &\dots &\dots &\dots &\dots &\dots     \\
0  & 0 & \dots  & I    & 0  & 0  & 0  & \cdots & 0   &0       \\
0  & 0 & \dots  & 0    & I  & 0  & 0  & \cdots & 0   &0     \\ \hline
u_1&u_2& \dots  &u_{k-1}&u_k& 0  & 0  & \cdots & 0   &0
\end{array}
\right)
 \ .
\]
Next we perform a diagonal refactorization of $B^{(1)}$.
Define the diagonal matrix $D$ by setting
\begin{align*}
D(t,t) &\ =\ 1 \quad \ \ \text{if } 1\leq t \leq (k+1)n \\
D\big((k+i)n+t,\, (k+i)n+t)\big)\  & =\ u_i(t)\ \text{ if }
1\leq i \leq k \ \text{ and }\ 1\leq t \leq n \\
       &\ =\ 1 \quad \ \ \text{if } t=(2k+1)n +1 \ .
\end{align*}
Define a matrix $X$ which is equal to $B^{(1)}$ except that
$X(1,t) = r$ if $(k+1)n + 1 \leq t \leq (2k+1)n$. Then
$B^{(1)}=XD$. Define $B^{(2)}=DX$. In block form,
\[
B^{(2)}=
\left(
\begin{array} {cccccc|ccc|c}
A_1&A_2 & \dots &A_{k-1}& A_k& 0& R  & \cdots & R & -re_1  \\
I  & 0  & \dots  & 0    & 0  & 0& 0  & \cdots & 0   &0         \\
0  & I  & \dots  & 0    & 0  & 0& 0  & \cdots & 0   &0         \\
0  & 0  & \dots  & 0    & 0  & 0& 0  & \cdots & 0   &0        \\
\dots   &\dots   &\dots &\dots  &\dots &\dots &\dots&\dots &\dots &\dots     \\
0  & 0  & \dots  & I    & 0  & 0& 0  & \cdots & 0   &0       \\
0  & 0  & \dots  & 0    & I  & 0& 0  & \cdots & 0   &0     \\ \hline
U'_1   &0   & \dots  & 0    & 0 & 0  & 0  & \cdots & 0   &0         \\
0  & U'_2 & \dots  & 0    & 0   & 0  & 0  & \cdots & 0   &0         \\
0  & 0    & \dots  & 0    & 0   & 0  & 0  & \cdots & 0   &0        \\
\dots     &\dots   &\dots &\dots&\dots  &\dots &\dots &\dots &\dots &\dots     \\
0  & 0    & \dots  &U'_{k-1}& 0 & 0  & 0  & \cdots & 0   &0       \\
0  & 0    & \dots  & 0    & U'_k& 0  & 0  & \cdots & 0   &0     \\ \hline
u_1&u_2   & \dots  &u_{k-1}&u_k & 0  & 0  &\cdots  & 0   &0
\end{array}
\right)
\]
in which every entry of the top row of $R$ is $r$ and the other entries
of $R$ are zero, and $U'_i$ denotes the diagonal matrix with
$U'_i(t,t) =u_i(t)$, for $1\leq t \leq n$.

Next, amalgamate the columns $(k+1)n+1, \dots , (2k+1)n$ (the columns
through the $R$ blocks) to a single column to form $B^{(3)}$.
For this define
\begin{align*}
Y&=
\left(
\begin{array}{ccccccc|c|c}
A_1&A_2&A_3 & \cdots  &A_{k-1}& A_k& 0 & re_1  & -re_1  \\
I  & 0 &0   & \cdots  & 0    & 0  & 0  & 0     &0         \\
0  & I &0   & \cdots  & 0    & 0  & 0  & 0     &0         \\
0  & 0 &I   & \cdots  & 0    & 0  & 0  & 0     &0   \\
\cdots  &\cdots&\cdots  &\cdots &\cdots &\cdots &\cdots &\cdots &\cdots     \\
0  & 0 & 0  & \cdots  & I    & 0  & 0 & 0 &0       \\
0  & 0 & 0  & \cdots  & 0    & I  & 0 & 0 &0     \\ \hline
U'_1& 0&0   & \cdots  & 0    & 0  & 0 & 0 &0         \\
0  & U'_2 &0& \cdots  & 0    & 0  & 0 & 0 &0         \\
0  & 0 &U'_3& \cdots  & 0    & 0  & 0 & 0 &0        \\
\cdots &\cdots&\cdots   &\cdots &\cdots&0 &\cdots &\cdots &\cdots     \\
0  & 0 & 0  & \cdots  &U'_{k-1}&0 & 0 & 0 &0       \\
0  & 0 & 0  & \cdots  & 0    & U'_k&0 & 0 &0     \\ \hline
u_1&u_2&u_3 & \cdots  &u_{k-1}&u_k& 0 & 0 & 0
\end{array}
\right)\quad \text{ and }\\
Z&=
\left(
\begin{array}{c|c|c}
I_{(k+1)n} & 0 & 0 \\ \hline
0 & 1 \cdots 1 & 0 \\ \hline
0 & 0 & 1
\end{array}
\right)
\end{align*}
in which the central block of $Z$ is a row vector of size
$kn$ with every entry 1. Then
$B^{(2)}= YZ$ and we define $B^{(3)}= ZY$. In block form,
\[
B^{(3)}=
\begin{pmatrix}
A_1&A_2&A_3 & \cdots &A_{k-1}& A_k    & 0      & re_1  & -re_1  \\
I  & 0 &0   & \cdots  & 0    & 0      & 0      & 0     &0         \\
0  & I &0   & \cdots  & 0    & 0      & 0      & 0     &0         \\
0  & 0 &I   & \cdots  & 0    & 0      & 0      & 0     & 0   \\
\cdots &\cdots&\cdots   &\cdots &\cdots   &\cdots   &\cdots  &\cdots &\cdots     \\
0  & 0 & 0  & \cdots  & I    & 0      & 0      & 0     &0       \\
0  & 0 & 0  & \cdots  & 0    & I      & 0      & 0     &0     \\
u_1&u_2&u_3 & \cdots  &u_{k-1}&u_k    & 0      & 0     &0    \\
u_1&u_2&u_3 & \cdots  &u_{k-1}&u_k    & 0      & 0     &0
\end{pmatrix}  \ .
\]
Next we similarly amalgamate the last two rows, to obtain the
matrix
\[
B^{(4)}=
\left(
\begin{array}{cccccc|c|c}
A_1&A_2&A_3 & \cdots &A_{k-1}& A_k    & 0      & 0       \\
I  & 0 &0   & \cdots  & 0    & 0      & 0      & 0              \\
0  & I &0   & \cdots  & 0    & 0      & 0      & 0              \\
0  & 0 &I   & \cdots  & 0    & 0      & 0      & 0        \\
I  & 0 &0   & \cdots  & 0    & 0      & 0      & 0              \\
0  & I &0   & \cdots  & 0    & 0      & 0      & 0              \\
0  & 0 &I   & \cdots  & 0    & 0      & 0      & 0             \\
\cdots &\cdots&\cdots   &\cdots &\cdots   &\cdots   &\cdots  &\cdots     \\
0  & 0 & 0  & \cdots  & I    & 0      & 0      & 0            \\ \hline
0  & 0 & 0  & \cdots  & 0    & I      & 0      & 0          \\ \hline
u_1&u_2&u_3 & \cdots  &u_{k-1}&u_k    & 0      & 0
\end{array}
\right)
\ .
\]
This matrix is a zero extension of $\As$ and therefore is
SSE over $\rr$ to $\As$  (see Proposition 6.5). This finishes the
proof in the case
$\ell=1$ that the matrices $\As$ and $\Bs$ are SSE over $\rr$.

The proof for the case $\ell > 1$ is very similar. We will discuss it for the case
$\ell =3$, from which the general argument should be clear.
For $\ell =3$, with the same notation as in the case $\ell =1$,
and recalling $A_i=0$ if $i>k$, we have
\[
\Bs =
\begin{pmatrix}
A_1 &A_2 &A_3-rX&A_4+rU_1 & \cdots &A_{k+1}+rU_{k-2}&A_{k+2}+rU_{k-1} &A_{k+3} +rU_k    \\
I   & 0  &0     &0        & \cdots & 0            & 0           & 0      \\
0   & I  &0     &0        & \cdots & 0            & 0           & 0      \\
0   & 0  &I     &0        & \cdots & 0            & 0           & 0     \\
0   & 0  &0     &I        & \cdots & 0            & 0           & 0     \\
\cdots &\cdots &\cdots&\cdots &\cdots  &\cdots    &\cdots       &\cdots    \\
0   & 0  & 0    &0        & \cdots & I            & 0           & 0        \\
0   & 0  & 0    &0        & \cdots & 0            & I           & 0
\end{pmatrix} \ .
\]
As in the case $\ell =1$, we split columns to isolate the terms involving $r$.
The resulting matrix $B^{(1)}$ here has a  form involving
a shift of the $\ell=1$ form in the new rows:
\[
B^{(1)}=
\left(
\begin{array}{ccccccc|cc|c}
A_1&A_2&A_3 &A_4\ \  \cdots   &A_{k+1}& A_{k+2}& 0 &rU_1 \ \ \cdots   &rU_k & -re_1  \\
I  & 0 &0   &0 \ \ \    \cdots \ \   & 0    & 0     & 0  & 0   \ \ \ \cdots\   & 0   &0         \\
0  & I &0   &0 \ \  \  \cdots \ \   & 0    & 0     & 0  & 0   \ \ \  \cdots  \  & 0   &0         \\
0  & 0 &I   &0 \ \ \  \cdots \ \   & 0    & 0     & 0  & 0   \ \  \ \cdots  \  & 0   &0        \\
\cdots &\cdots&\cdots &  \cdots \ \ \  \cdots &\cdots &\cdots  &\cdots & \ \cdots \ \ \cdots\ \ &\cdots &\cdots     \\
0  & 0 & 0  &0  \ \ \   \cdots \ \  & I    & 0     & 0  & 0   \ \ \cdots \   & 0   &0       \\
0  & 0 & 0  &0  \ \ \   \cdots \ \  & 0    & I     & 0  & 0  \  \ \ \cdots \   & 0   &0     \\ \hline
0  & 0 &I   &0  \ \ \   \cdots \ \  & 0    & 0     & 0  & 0    \ \ \ \cdots \   & 0   &0         \\
0  & 0 &0   &I  \ \ \   \cdots \ \  & 0    & 0     & 0  & 0   \ \ \  \cdots \   & 0   &0         \\
0  & 0 &0   &0  \ \ \   \cdots \ \  & 0    & 0     & 0  & 0   \ \  \ \cdots \   & 0   &0        \\
\cdots  &\cdots&  \cdots &   \cdots \ \ \ \cdots &\cdots&\cdots&\cdots &\cdots \  \cdots \ \ &\cdots &\cdots     \\
0  & 0 & 0  &0  \ \   \cdots \ \  & I    & 0     & 0  & 0   \ \  \ \cdots \   & 0   &0       \\
0  & 0 & 0  &0  \ \   \cdots \ \  & 0    & I     & 0  & 0   \ \  \ \cdots \   & 0   &0     \\ \hline
0  & 0 &u_1 &u_2\ \  \cdots \ \   & u_{k-1}&u_k    & 0  & 0  \ \  \ \cdots \   & 0   &0
\end{array}
\right)  \ .
\]
From here the argument proceeds  as in the case $\ell =1$,
through slightly different matrices,
\[
B^{(2)}=
\left(
\begin{array}{ccccccc|cc|c}
A_1&A_2&A_3 &A_4\ \  \cdots   &A_{k+1}& A_{k+2}& 0 &R \ \ \cdots   &R & -re_1  \\
I  & 0 &0   &0 \ \ \    \cdots \ \   & 0    & 0     & 0  & 0   \ \ \ \cdots\   & 0   &0         \\
0  & I &0   &0 \ \  \  \cdots \ \   & 0    & 0     & 0  & 0   \ \ \  \cdots  \  & 0   &0         \\
0  & 0 &I   &0 \ \ \  \cdots \ \   & 0    & 0     & 0  & 0   \ \  \ \cdots  \  & 0   &0        \\
\cdots &\cdots&\cdots &  \cdots \ \ \  \cdots &\cdots &\cdots  &\cdots & \ \cdots \ \ \cdots\ \ &\cdots &\cdots     \\
0  & 0 & 0  &0  \ \ \   \cdots \ \  & I    & 0     & 0  & 0   \ \ \cdots \   & 0   &0       \\
0  & 0 & 0  &0  \ \ \   \cdots \ \  & 0    & I     & 0  & 0  \  \ \ \cdots \   & 0   &0     \\
\hline
0  & 0 &U'_1 &0  \ \ \   \cdots \ \  & 0    & 0     & 0  & 0    \ \ \ \cdots \   & 0   &0         \\
0  & 0 &0   &U'_2  \ \ \   \cdots \ \  & 0    & 0     & 0  & 0   \ \ \  \cdots \   & 0   &0         \\
0  & 0 &0   &0  \ \ \   \cdots \ \  & 0    & 0     & 0  & 0   \ \  \ \cdots \   & 0   &0        \\
\cdots  &\cdots&  \cdots &   \cdots \ \ \ \cdots &\cdots&\cdots&\cdots &\cdots \  \cdots \ \ &\cdots &\cdots     \\
0  & 0 & 0  &0  \ \   \cdots \ \  & U'_{k-1}    & 0     & 0  & 0   \ \  \ \cdots \   & 0   &0       \\
0  & 0 & 0  &0  \ \   \cdots \ \  & 0    & U'_k     & 0  & 0   \ \  \ \cdots \   & 0   &0     \\
\hline
0  & 0 &u_1 &u_2\ \  \cdots \ \   & u_{k-1}&u_k    & 0  & 0  \ \  \ \cdots \   & 0   &0
\end{array}
\right)
\]
and
\[
B^{(3)}=
\left(
\begin{array}{ccccccccc}
A_1&A_2&A_3 &A_4\ \  \cdots   &A_{k+1}& A_{k+2}& 0   &re_1 & -re_1  \\
I  & 0 &0   &0 \ \ \ \cdots \ \ & 0   & 0     & 0    & 0   &0         \\
0  & I &0   &0 \ \  \  \cdots \ \ & 0 & 0     & 0    & 0   &0         \\
0  & 0 &I   &0 \ \ \  \cdots \ \   & 0 & 0    & 0    & 0   &0        \\
\cdots &\cdots&\cdots &  \cdots \ \ \  \cdots &\cdots &\cdots  &\cdots & \cdots &\cdots     \\
0  & 0 & 0  &0  \ \ \   \cdots \ \  & I & 0   & 0    & 0   &0       \\
0  & 0 & 0  &0  \ \ \   \cdots \ \  & 0 & I   & 0    & 0   &0     \\
0  & 0 &u_1 &u_2\ \  \cdots \ \   & u_{k-1}&u_k    & 0    & 0   &0 \\
0  & 0 &u_1 &u_2\ \  \cdots \ \   & u_{k-1}&u_k    & 0    & 0   &0
\end{array}
\right)  \ .
\]
 This completes our proof that
$\As$ and $\Bs$ are SSE over $\rr$ in the case $E(I-A)=I-B$.

Now suppose $(I-A)E=I-B$. In place of $\As$, we consider a
matrix form corresponding to a role reversal for rows
and columns:
\[
 A^{\text{col}} \ = \
\begin{pmatrix}
A_1     & I   &0      &\cdots   &0      &0       & 0 \\
A_2     & 0   &I      &\cdots   & 0     & 0      & 0  \\
A_3     & 0   &0      &\cdots   & 0     & 0      & 0  \\
\cdots &\cdots& \cdots &\cdots  &\cdots &\cdots  &\cdots  \\
A_{k-2} & 0   &0      & \cdots  & 0     & I      & 0  \\
A_{k-1} & 0   & 0     & \cdots  &0      & 0      & I  \\
A_k     & 0   & 0     & \cdots  &0      & 0      & 0
\end{pmatrix} \ .
\]
With the roles of row and column reversed, the arguments
we've given show that $ A^{\text{col}}$ and $ B^{\text{col}}$
are SSE over $\rr$. What remains is to see that
$ A^{\text{col}} $ and $\As$ are SSE over $\rr$.
For this we define a matrix $A'$ with the
block form

\[
A'=
\left(
\begin{array}{cc|c|c|c|c|c}
A_1 & A_2 &A_3\ \ 0&A_4\  0& \cdots       &A_{k-1}\ \ 0&A_k\ \ 0 \\
I   & 0   &0\ \ \ 0  & 0 \ \ 0& \cdots   & 0\ \ \ 0 & 0\ \ 0    \\ \hline
0   & 0  &\ 0\ \ \ I_{n}& 0 \ \ 0& \cdots & 0\ \ \ 0 & 0\ \ 0    \\
I   & 0   &0\ \ \  0  & 0 \ \ 0& \cdots    & 0\ \ \ 0 & 0\ \ 0 \\  \hline
0   & 0  &\, 0\ \ \ 0& \ 0 \ \ I_{2n}& \cdots& 0\ \ \ 0 & 0\ \ 0    \\
I   & 0   &0\ \ \  0  & 0 \ \ 0& \cdots    & 0\ \ \ 0 & 0\ \ 0    \\  \hline
\cdots    &\cdots  &\cdots  &\cdots        &\cdots  &\cdots &\cdots\\ \hline
0   & 0   &0\ \ \ 0&0\ \ \ 0  & \cdots &\ \ \ \  \ 0\ \ I_{(k-3)n}   & 0\ \ 0      \\
I   & 0   &0\ \ \ 0&0\ \ \ 0  & \cdots &  0\ \ \ \ 0                 & 0\ \ 0 \\  \hline
0   & 0   &0\ \ \ 0&0\ \ \ 0  & \cdots &  0\ \ \ \ 0                 & \ \ \ \ \ \ 0\ \ I_{(k-2)n} \\
I   & 0   &0\ \ \ 0&0\ \ \ 0  & \cdots &  0\ \ \ \ 0                 & 0\ \ 0
\end{array}
\right) \ .
\]
In the display of
$A'$ above and next,  a block $I$ without subscript  is $I_n$.
%$\left(\begin{smallmatrix}A_j &0\end{smallmatrix}\right)$
%in which $0$ is $n\times (j-2)n$. $I'_j$ denotes the

%$\left(\begin{smallmatrix} 0\\I_n\end{smallmatrix}\right)$
For example, if $k=4 $ then
\[
A'=
\left(
\begin{array}{cc|cc|ccc}
A_1 & A_2 &A_3 & 0      & A_4 & 0      & 0 \\
I   & 0   &0   & 0      & 0   & 0      & 0  \\ \hline
0   & 0   &0   & I      & 0   & 0      & 0  \\
I   & 0   &0   & 0      & 0   & 0      & 0  \\  \hline
0   & 0   & 0  &0       & 0   & I      & 0  \\
0   & 0   & 0  &0       & 0   & 0      & I  \\
I   & 0   & 0  &0       & 0   & 0      & 0
\end{array}
\right) \ .
\]
Three steps remain.

First, the matrix $A'$ is SSE over $\rr$ to $\As$ by a string
of $k-2$ block row amalgamations. Beginning with $A'=A'_0$: amalgamate
to block row 2 the
block rows with $I$ in block column 1 to form $A'_1$.
From the resulting matrix,
amalgamate to block row 3 the rows  with $I$ in column 2,
to form $A'_2$.  Etc.
The last block row amalgamation produces
$\As$. For example, with $A'$ above for $k=4$ and
\[
X=
\left(
\begin{array}{ccccccc}
I   & 0   &0   & 0      & 0   \\
0   & I   &0   & 0      & 0    \\
0   & 0   &I   & 0      & 0    \\
0   & I   &0   & 0      & 0    \\
0   & 0   & 0  & I      & 0    \\
0   & 0   & 0  &0       & I    \\
0   & I   & 0  &0       & 0
\end{array}
\right) \ , \quad
Y=
\left(
\begin{array}{ccccccc}
A_1 & A_2 &A_3 & 0      & A_4 & 0      & 0 \\
I   & 0   &0   & 0      & 0   & 0      & 0  \\
0   & 0   &0   & I      & 0   & 0      & 0  \\
0   & 0   & 0  &0       & 0   & I      & 0  \\
0   & 0   & 0  &0       & 0   & 0      & I
\end{array}
\right)
\]
we have $A'=A'_0= XY$ and
\[
A'_1= YX =
\left(
\begin{array}{ccccc}
A_1 & A_2 &A_3      & A_4 & 0     \\
I   & 0   &0        & 0   & 0      \\
0   & I   &0        & 0   & 0      \\
0   & 0   & 0       & 0   & I      \\
0   & I   & 0       & 0   & 0
\end{array} \ .
\right)
\]
The next step produces $A'_2=\As$.

Second, the matrix $A'$ is conjugate to the matrix $A^*$
obtained from $A'$ by (i)  replacing
in block row 1
the blocks $A_j$,
 $2\leq j\leq k$, with the identity
block $I_n$ and (ii) replacing the  $I$ blocks in
block column 1 with $ A_2, \dots , A_k$ (with
$A_j$ appearing above $A_{j+1}$, $1\leq i < k$).
 An SSE from $A'$ to $A^*$ is achieved
by a string of diagonal refactorizations of the blocks
$A_j$. For example, in the display for  $k=4$,
let $X$ be the matrix obtained from
$A'$ by replacing the $A_2$ block with $I$. Let $D$
be the block
diagonal matrix with block indices  matching those of $A'$,
and with $D=A_2$ in the second diagonal block and $D=I$ otherwise.
Then $XD=A'$ and $DX$ has $A_2$ occupying the $2,1$ block as
desired. To move $A_j$ to its target position
in the  first  block column
takes $j-1$
  moves of this type.

Third and last, the matrix $A^*$ is SSE over $\rr$ to the
matrix $A^{\text{col}}$ by a string of block column amalgamations,
just as $A'$ is SSE over $\rr$ to $\As$ by a string of block
row amalgamations.

This finishes the proof of the lemma.
\end{proof}

We record a corollary of Theorem \ref{finecentral}.
%The following corollary follows easily from the above, and is useful enough to state on its own.
\begin{corollary} \label{sseandel}
Suppose $\Rcal$ is a ring, and suppose
$P$ and $Q$ are square
matrices over $\Rcal[t]$. Suppose $A^{\prime}$ and $B^{\prime}$ are
matrices over $\Rcal$ such that $P$ and $Q$ are $\textnormal{El}(\Rcal[t])$ equivalent
  (respectively) to
$I-tA^{\prime}$ and $I-tB^{\prime}$. Then the following are equivalent:
\begin{enumerate}
\item
$A^{\prime}$ and $B^{\prime}$ are SSE over $\Rcal$.
\item
$P$ and $Q$ are $\textnormal{El} (\Rcal)[t]$ equivalent.
\end{enumerate}
\end{corollary}

%\red{With Corollary \ref{sseandel} and
%Theorem \ref{selisttheorem}, we can interpret
%Theorem \ref{aplusn2} as saying that
%$\text{NK}_1(\Rcal )$ is in bijective correspondence
%with the $\textnormal{El}(R[t])$ equivalence classes contained
%in the $\textnormal{Gl}(\calR[t]) $ equivalence
%}

\section{SSE and $\text{Nil}_0(\calR)$} \label{nilsec}
%\section{Strong shift equivalence and $\text{Nil}_0(R)$}

Nilpotent matrices $N, N'$ over $\calR$
represent the same element of
$\text{Nil}_0(\calR)$ if and only if
$I-tN$ and $I-tN'$ represent the same
element of $NK_1(\calR)$.
It therefore follows from Theorem \ref{finecentral} that
there is another characterization of when
nilpotent matrices $N,N'$
represent the same element of
$\text{Nil}_0(\calR)$:
\begin{theorem} \label{sseandnil}
Suppose $N$ and $N'$ are nilpotent matrices over a ring
$\rr$. Then the following are equivalent.
\begin{enumerate}
\item
$[N]=[N']$ in $\textnormal{Nil}_0(\rr )$.
\item
$N$ and $N'$ are SSE over $\rr$.
\end{enumerate}
\end{theorem}
(There is also a shorter proof of Theorem \ref{sseandnil}, avoiding
  Theorem \ref{finecentral}, which we  forego.) Consequently,
  we can think of Theorem \ref{finecentral} as a generalization
  of the correspondence $\textnormal{Nil}_0(\rr )\to \nk_1(\rr )$,
  from nilpotent matrices to arbitrary matrices.

\begin{remark}
{\normalfont
Theorem \ref{finecentral} is an alternate ingredient
for a proof that
$\text{Nil}_0(\calR)$ and $NK_1(\calR)$ are isomorphic.
If the matrix $A$ in $\mathcal M (\rr )$
is nilpotent, then the map $\beta \colon A\mapsto I-tA$ is the
standard map inducing the group isomorphism
$\text{Nil}_0(\calR)\to \text{NK}_1(\rr )$. It is straightforward
to check that $\beta$ induces a well defined homomorphism
$\text{Nil}_0(\calR)\to \text{NK}_1(\rr )$, which is surjective
on account of the Higman trick (see \cite{WeibelBook} Proposition 3.5.3, or \cite{Rosenberg1994} Theorem 3.2.22).
The more difficult part of the proof is to show that this
epimorphism is injective. For example, Weibel proves this
with a sophisticated composition of maps
(see \cite{WeibelBook}, Section III.3.5).
Rosenberg approaches this by defining a map inducing
the inverse, but (he agrees that) the proof
\cite[p.150]{Rosenberg1994}
that the map  is
well defined is incomplete.
The map of Theorem
\ref{finecentral} restricts to define
an inverse to the standard epimorphism
$\text{Nil}_0(\calR)\to \text{NK}_1(\rr )$, and
therefore gives an alternate proof for this step,
in the spirit of Rosenberg's approach.
It also identifies the elements of
$\text{Nil}_0(\rr)$ as SSE-$\rr$ classes. } \end{remark}

\bibliographystyle{plain}

\bibliography{mbssbib}

\end{document}